\documentclass[12pt]{article}
\usepackage{amsfonts}
\usepackage{amsthm}
\input{epsf.sty}
\usepackage{latexsym,amsmath,amsfonts,amscd}
\usepackage{epsfig}
\usepackage{changebar}
\usepackage{pstricks}
\usepackage{pst-plot}
\usepackage{multirow}
\usepackage{subfigure}
\usepackage{placeins}
\usepackage{amssymb}
\usepackage{mathrsfs}
\usepackage[title]{appendix}
\usepackage{bm}
\usepackage{hyperref}

\usepackage{multirow,bigstrut}
\usepackage{enumitem}
\usepackage{booktabs}

\numberwithin{equation}{section}

\theoremstyle{plain}
\newtheorem{thm}{Theorem}[section]

\newtheorem{prop}[thm]{Proposition}

\newtheorem{rem}[thm]{Remark}

\theoremstyle{definition}
\newtheorem{exam}[thm]{Example}

\newcommand{\brac}[1]{\left(#1\right)}
\newcommand{\abs}[1]{\left\vert#1\right\vert}

\def\half{\frac 1 2}
\newcommand{\nn}{\nonumber}

\newcommand{\bk}{\mathbf{k}}
\newcommand{\bx}{\mathbf{x}}

\newfont{\iams}{msbm9}

\newcommand{\commentbis}[1]{}
\newcommand{\be}{\begin{eqnarray}}
\newcommand{\ee}{\end{eqnarray}}
\newcommand{\beno}{\begin{eqnarray*}}
\newcommand{\eeno}{\end{eqnarray*}}

\newcommand{\barr}[1]{\begin{array}{#1}}
\newcommand{\earr}{\end{array}}

\newcommand{\beq}{\begin{equation}}
\newcommand{\eeq}{\end{equation}}
\newcommand{\beqa}{\begin{eqnarray}}
\newcommand{\eeqa}{\end{eqnarray}}

\newcommand{\bv}{{\bf v}}

\newcommand{\bV}{{\bf V}}

\newcommand{\jl}{j-\frac 12}
\newcommand{\jr}{j+\frac 12}
\newcommand{\il}{i-\frac 12}
\newcommand{\ir}{i+\frac 12}
\newcommand{\bfv}{\mathbf{V}_h^k}

\newcommand{\bzero}{\mathbf{0}}
\newcommand{\bone}{\mathbf{1}}
\newcommand{\bl}{\mathbf{l}}
\newcommand{\bi}{\mathbf{i}}

\newcommand{\bj}{\mathbf{j}}
\newcommand{\bW}{\mathbf{W}}

\DeclareMathOperator{\sech}{sech}
\DeclareMathOperator{\arccot}{arccot}
\allowdisplaybreaks[4]
\begin{document}
\baselineskip=1.6pc

\vspace{.5in}

\begin{center}

{\large\bf A class of adaptive multiresolution ultra-weak discontinuous Galerkin methods for some nonlinear dispersive wave equations}

\end{center}

\vspace{.1in}

\centerline{
Juntao Huang \footnote{Department of Mathematics,
Michigan State University, East Lansing, MI 48824, USA.
E-mail: huangj75@msu.edu} \qquad
Yong Liu \footnote{LSEC, Institute of Computational Mathematics, Hua Loo-Keng Center for Mathematical
Sciences, Academy of Mathematics and Systems Science, Chinese Academy of Sciences, Beijing
100190, China.
E-mail address: yongliu@lsec.cc.ac.cn.
Research is supported by the fellowship of China Postdoctoral Science Foundation No. 2020TQ0343.
} \qquad
Yuan Liu\footnote{Department of Mathematics, Statistics and Physics,
Wichita State University, Wichita, KS 67260, USA.
E-mail: liu@math.wichita.edu.
Research supported in part by a grant from the Simons Foundation (426993, Yuan Liu). }\qquad
Zhanjing Tao  \footnote{School of Mathematics, Jilin University, Changchun, Jilin 130012, China. zjtao@jlu.edu.cn. Research is supported by NSFC grant 12001231.}  \qquad
Yingda Cheng  \footnote{Department of Mathematics, Department of Computational Mathematics, Science and Engineering, Michigan State University, East Lansing, MI 48824, USA. E-mail: ycheng@msu.edu. Research is supported by NSF grant DMS-2011838.}
}

\vspace{.4in}

\centerline{\bf Abstract}

\vspace{.1in}
In this paper, we propose a class of adaptive multiresolution (also called adaptive sparse grid) ultra-weak discontinuous Galerkin (UWDG) methods for solving some nonlinear dispersive wave equations including the Korteweg-de Vries (KdV) equation and its two dimensional generalization, the Zakharov-Kuznetsov (ZK) equation.  The UWDG formulation, which relies on repeated integration by parts, was proposed for KdV equation  in \cite{cheng2008discontinuous}. For the ZK equation which contains mixed derivative terms, we develop a new UWDG formulation. The $L^2$ stability and the optimal error estimate with a novel local projection are established for this new scheme on regular meshes. Adaptivity   is achieved based on multiresolution and is particularly effective for capturing solitary wave structures. Various numerical examples are presented to demonstrate the accuracy and capability of our methods.
\bigskip

\bigskip
%\vfill


%{\bf AMS(MOS) subject classification:} 65M99

\pagenumbering{arabic}

%section 1
\section{Introduction}\label{sec:intro}
\setcounter{equation}{0}
\setcounter{figure}{0}
\setcounter{table}{0}

In this paper, we are interested in numerically solving dispersive equations including the Korteweg-de Vries (KdV) equation in one-dimension (1D) \cite{korteweg1895}
\begin{equation}\label{eq:kdv}
	u_t + f(u)_x + u_{xxx} = 0,
\end{equation}
and the Zakharov-Kuznetsov (ZK) equation in two-dimension (2D) \cite{zakharov1974threedimensional}
\begin{equation}\label{eq:zk}
	u_t + f(u)_x + u_{xxx} + u_{xyy} = 0.
\end{equation}
Our work can be easily generalized to 3D ZK equation. For simplicity of discussion, we did not include it in this paper.

The KdV-type equations first arise in the study of shallow wave propagations, and later are widely used to describe the propagation of waves in a variety of nonlinear and dispersive media. In the literature, these equations have wide applications in scientific fields including acoustics, nonlinear optics, hydromagnetics, and among others \cite{benjamin1972model}. As an important model in the family of  KdV-type equations, the ZK equation \eqref{eq:zk} is a high dimensional generalization of the 1D KdV equation \eqref{eq:kdv}, which governs the behavior of weakly nonlinear ion-acoustic waves in a plasma comprising cold ions and hot isothermal electrons in the presence of a uniform magnetic field. It was found that the solitary wave solutions of the ZK equation are inelastic \cite{Wazwaz2009}.

A vast amount of numerical methods have been developed to solve KdV-type equations, including finite difference method, finite volume method, finite element method as well as spectral method. In this paper, we are interested in further exploring discontinuous Galerkin (DG) method, which is a class of finite element method using discontinuous piecewise polynomials for numerical solutions \cite{cockburn2001runge,cockburn2012discontinuous}. The main advantages of DG method include flexibility in handling geometry, provable convergence properties and accommodating $h$-$p$ adaptivity. Various DG schemes have been applied to solve the KdV-type equations.  In \cite{yan2002local, xu2005local, xu2007error}, local DG (LDG) schemes are investigated to solve the KdV-type equations, including the ZK equation. Hybridizable DG (HDG) is used to solve 1D KdV equations in \cite{dong2017optimally}. In contrast of LDG schemes which need to introduce auxiliary variables and rewrite the equation into first order system, direct DG (DDG) schemes \cite{yi2013direct} and ultra-weak DG (UWDG) schemes  \cite{cheng2008discontinuous, bona2013conservative, fu2019energy} are developed to simulate the 1D KdV equation.  While optimal $L^2$-error estimates of various types of semi-discrete DG schemes are obtained for the 1D KdV equation \cite{xu2012optimal, dong2017optimally, fu2019energy, bona2013conservative} using various projection techniques, only sub-optimal error estimates of LDG schemes are available for the ZK equation \cite{xu2007error}. Also, the ZK equation, which contains   mixed derivative term, has not been considered yet under the UWDG framework.  One aim of this work is to design stable and accurate UWDG scheme for the ZK equation, which can also shed light on the design of the UWDG formulation for general mixed derivative terms. 
We showed that by applying  integration by parts in a proper order and continuing to integrate on the trace of elements until the technique can no longer be used, the resulting weak formulation will yield a $L^2$ stable scheme.
Further,  a novel local projection can be designed to prove  optimal $L^2$ error estimate on regular meshes.

KdV-type equations admit solitary wave solutions, which means that the solutions often contain localized structures. This makes adaptivity very appealing in efficient numerical simulations, which is another aim of this work.
In particular, this paper continues our line of research \cite{guo2017adaptive, huang2019adaptive, huang2020adaptive, tao2020Schrodinger, guo2020hj}   on the development of adaptive multiresolution DG methods. The adaptivity in our proposed schemes is realized through the hierarchical polynomial spaces and multiresolution analysis (MRA), which has built in error indicators. By further incorporating the sparse tensor product in 2D space, adaptive multiresolution  UWDG schemes for the ZK equation are capable of capturing solitary waves efficiently. Two classes of multiwavelets are employed in our schemes to attain MRA. First, the Alperts orthonormal multiwavelets are adopted as the DG basis functions \cite{wang2016elliptic}, and then the interpolatory multiwavelets \cite{tao2021collocation} are introduced to deal with the nonlinear terms in the scheme.

The organization of the paper is as follows. In Section \ref{sec:zk}, we review the UWDG method for the KdV equation and present our new UWDG method for the ZK equation. The $L^2$ stability and the optimal error estimate on the ZK equation will be provided. In Section \ref{sec:method}, we review the fundamentals of Alpert's and interpolatory wavelets and propose adaptive multiresolution UWDG schemes for both the KdV equation and the ZK equation. In Section \ref{sec:result}, we provide several numerical examples to illustrate the capability of our proposed adaptive multiresolution UWDG method. Concluding remarks are given in Section \ref{sec:conclusion}.

\section{Semi-discrete UWDG schemes}\label{sec:zk}

In this section, we will present the UWDG method for the KdV equation and the ZK equation on full grid, which serve as the building blocks of the adaptive multiresolution UWDG methods. Specifically, we review the UWDG method proposed in \cite{cheng2008discontinuous} for the KdV equation.
Moreover, we develop a new UWDG method for the ZK equation with proofs on $L^2$ stability and optimal error estimates.

\subsection{KdV equation}
Let the computational domain be the interval $\Omega\subset\mathbb{R}$, which is divided into $N$ intervals $\Omega=\bigcup_{i=1}^{N_x}I_i$ with $I_i = (x_{i-\frac{1}{2}}, x_{i+\frac{1}{2}})$, $x_i = \frac{1}{2} (x_{i-\frac{1}{2}} + x_{i+\frac{1}{2}})$, $h_i = x_{i+\frac{1}{2}} - x_{i-\frac{1}{2}}$, and $h = \max_i h_i$. Define the DG finite element space as
\begin{equation}\label{eq:dg-space}
	V^k_h=\{v\in L^2(\Omega):v|_{I_i}\in P^k(I_i),\, \forall i = 1,\ldots,N_x \},
\end{equation}
where $P^k(I_i)$ represents all polynomials of degree at most $k$ on $I_i$. Then, the semi-discrete UWDG scheme proposed in \cite{cheng2008discontinuous} is based on repeated integration by parts, and to find $u^h \in V^k_h$, such that for any test function $v^h \in V^k_h$, there holds the equality:
\begin{align}\label{eq:kdv-weak-formulation0} \nonumber
& \int_{\Omega} (u^h)_t v^h dx = \int_{\Omega} f(u^h)v^h_x dx +\int_{\Omega} u^h v^h_{xxx} dx
-\sum_{i=1}^{N_x} ( \hat{f}_{i+\frac{1}{2}} (v^h)^-_{i+\frac{1}{2}}- \hat{f}_{i-\frac{1}{2}} (v^h)^+_{i-\frac{1}{2}})  \\ \nonumber &
 - \sum_{i=1}^{N_x} (\hat{u}^h_{i+\frac{1}{2}} (v^h_{xx})^-_{i+\frac{1}{2}} -\hat{u}^h_{i-\frac{1}{2}} (v^h_{xx})^+_{i-\frac{1}{2}} )
 +\sum_{i=1}^{N_x} ( (\widetilde{u}^h_x)_{i+\frac{1}{2}}(v^h_x)^-_{i+\frac{1}{2}} - (\widetilde{u}^h_x)_{i-\frac{1}{2}}(v^h_x)^+_{i-\frac{1}{2}}  ) \\&
 -\sum_{i=1}^{N_x} ( (\check{u}^h_{xx})_{i+\frac{1}{2}} (v^h)^-_{i+\frac{1}{2}}  - (\check{u}^h_{xx})_{i-\frac{1}{2}} (v^h)^+_{i-\frac{1}{2}}).
\end{align}
Here, $\hat{f}$, $\hat{u}^h$, $\widetilde{u}^h_x$ and  $\check{u}^h_{xx}$ are numerical fluxes and will be chosen as follows.
For the convection term, the global Lax-Friedrichs flux
$$\hat{f}((u^h)^+, (u^h)^-) = \frac{1}{2} (f((u^h)^+) + f((u^h)^-)) -\alpha ((u^h)^+ - (u^h)^-))$$
with $\alpha = \max\abs{f'(u)}$ is used.
Numerical fluxes about $u$ are taken by following the alternating principle \cite{cheng2008discontinuous}
\begin{align}\label{kdv3_00}
\hat{u}_{i+\frac{1}{2}}^h = u^h(x_{i+\frac{1}{2}}^-), \qquad (\widetilde{u}^h_x)_{i+\frac{1}{2}} = u^h_x(x_{i+\frac{1}{2}}^+), \qquad (\check{u}^h_{xx})_{i+\frac{1}{2}} = u^h_{xx}(x_{i+\frac{1}{2}}^+)
\end{align}
or alternatively
\begin{align}\label{kdv3_01}
\hat{u}_{i+\frac{1}{2}}^h = u^h(x_{i+\frac{1}{2}}^+), \qquad (\widetilde{u}^h_x)_{i+\frac{1}{2}} = u_x^h(x_{i+\frac{1}{2}}^+), \qquad (\check{u}^h_{xx})_{i+\frac{1}{2}} = u_{xx}^h(x_{i+\frac{1}{2}}^-).
\end{align}

As shown in \cite{bona2013conservative, fu2019energy}, optimal error estimate can be obtained for the semi-discrete the scheme (\ref{eq:kdv-weak-formulation0}), when no convection term presents in (\ref{eq:kdv}). Also energy conserving schemes can be designed by switching the fluxes \eqref{kdv3_00}-\eqref{kdv3_01} to central fluxes.

\subsection{ZK equation}\label{sec.zk}
In this part, we present a new ultra-weak DG scheme for solving the ZK equation \eqref{eq:zk}. For simplicity, we only consider periodic boundary condition here and similar results can be obtained for Dirichlet boundary conditions.

Given a rectangular domain $\Omega\subset\mathbb{R}^2$, we have a shape regular partition of $\Omega$ into rectangular cells $\Omega=\bigcup_{i,j}K_{i,j}$ with $K_{i,j}=I_i\times J_j=[x_{\il},x_{\ir}]\times[y_{\jl},y_{\jr}]$ for $i=1\ldots,N_x$ and $j=1\ldots,N_y$.
Let $x_{i}=\frac{1}{2}(x_{\il}+x_{\ir})$, $y_j=\frac 12 (y_{\jl}+y_{\jr})$, $h_{x_i}=(x_{\ir}-x_{\il})$, $h_{y_j}=(y_{\jr}-y_{\jl})$ and $h=\max_{i,j}(h_{x_i},h_{y_j})$. The DG finite element space is defined as
\begin{equation}
\bfv =\{v\in L^2(\Omega):v|_{K_{i,j}}\in Q^k(K_{i,j}),\, \forall i=1\ldots,N_x, ~ j=1\ldots,N_y \},	
\end{equation}
where $Q^k(K_{i,j})$ denotes the space of tensor product polynomials of degree at most $k$ in each dimension.

The main idea of UWDG scheme is to repeatedly apply integration by parts so all the spatial derivatives are shifted from the solution to the test function in the weak formulations. Due to the existence of the mixed derivative in the ZK equation, we apply integration by part in a proper order and continue to integrate on the trace of elements until the integration by part can no longer be used. This will result in some terms involving the vertices of each element which does not appear in the previous DG method. Specifically, we propose our new UWDG scheme for \eqref{eq:zk} as follows: find $u^h \in \bfv$, such that for any test function $v^h \in \bfv$,
\begin{align}\label{zk3.1.2}
 \int_{\Omega} u^h_t (v^h) dxdy  ={}& \int_{\Omega} f(u^h)v^h_x dxdy + \int_{\Omega} u^h v^h_{xxx} dxdy +\int_{\Omega} u^h v^h_{xyy} dxdy \\ \nonumber
 & -\sum_{i=1}^{N_x}\sum_{j=1}^{N_y} \int_{J_j} \hat{f}_{i+\frac{1}{2}, y} (v^h)^-_{i+\frac{1}{2}, y} -\hat{f}_{i-\frac{1}{2}, y} (v^h)^+_{i-\frac{1}{2}, y} dy \\ \nonumber
 & -\sum_{i=1}^{N_x}\sum_{j=1}^{N_y} \int_{J_j} (\check{u}^h_{xx})_{i+\frac{1}{2}, y} (v^h)^-_{i+\frac{1}{2}, y} - (\check{u}^h_{xx})_{i-\frac{1}{2}, y} (v^h)^+_{i-\frac{1}{2}, y} dy \\ \nonumber
 & + \sum_{i=1}^{N_x}\sum_{j=1}^{N_y} \int_{J_j}  (\widetilde{u}^h_{x})_{i+\frac{1}{2}, y} (v^h_x)^-_{i+\frac{1}{2}, y} - (\widetilde{u}^h_{x})_{i-\frac{1}{2}, y} (v^h_x)^+_{i-\frac{1}{2}, y} dy \\ \nonumber
 & - \sum_{i=1}^{N_x}\sum_{j=1}^{N_y} \int_{J_j}  \hat{u}^h_{i+\frac{1}{2}, y} (v^h_{xx})^-_{i+\frac{1}{2}, y} - \hat{u}^h_{i-\frac{1}{2}, y} (v^h_{xx})^+_{i-\frac{1}{2}, y} dy \\ \nonumber
 & + \sum_{i=1}^{N_x}\sum_{j=1}^{N_y} \int_{I_i} (\check{u}^h_{y})_{x, j+\frac{1}{2}} (v^h_x)^-_{x, j+\frac{1}{2}} - (\check{u}^h_{y})_{x, j-\frac{1}{2}} (v^h_x)^+_{x, j-\frac{1}{2}} dx\\ \nonumber
 & +\sum_{i=1}^{N_x}\sum_{j=1}^{N_y}-(\check{u}^h_{y})_{\ir,\jr}(v^h)^{-,-}_{\ir,\jr}+(\check{u}_y^h)_{\il,\jr}(v^h)_{\il,\jr}^{+,-} \\ \nonumber
 & +\sum_{i=1}^{N_x}\sum_{j=1}^{N_y}(\check{u}_y^h)_{\ir,\jl}(v^h)_{\ir,\jl}^{-,+}-(\check{u}_y^h)_{\il,\jl}(v^h)_{\il,\jl}^{+,+}\\\nonumber
&- \sum_{i=1}^{N_x}\sum_{j=1}^{N_y} \int_{J_j}  (\widetilde{u}^h)_{i+\frac{1}{2}, y} (v^h_{yy})^-_{i+\frac{1}{2}, y} - (\widetilde{u}^h)_{i-\frac{1}{2}, y} (v^h_{yy})^+_{i-\frac{1}{2}, y} dy \\ \nonumber
&+ \sum_{i=1}^{N_x} \sum_{j=1}^{N_y}  (\widetilde{u}^h)_{\ir,\jr}(v^h_y)_{\ir,\jr}^{-,-}-(\widetilde{u}^h)_{\ir,\jl} (v^h_y)_{\ir,\jl}^{-,+}\\\nonumber
&+ \sum_{i=1}^{N_x} \sum_{j=1}^{N_y} -(\widetilde{u}^h)_{\il,\jr}(v^h_y)_{\il,\jr}^{+,-}+(\widetilde{u}^h)_{\il,\jl}(v^h_y)_{\il,\jl}^{+,+} \nonumber\\
& - \sum_{i=1}^{N_x}\sum_{j=1}^{N_y} \int_{I_i}  (\hat{u}^h)_{x, j+\frac{1}{2}} (v^h_{xy})^-_{x, j+\frac{1}{2}} - (\hat{u}^h)_{x, j-\frac{1}{2}} (v^h_{xy})^+_{x, j-\frac{1}{2}} dx,\nonumber
\end{align}
Here, the numerical fluxes on the element interfaces are chosen as
\begin{align}\label{zk3.1.3.1}
& (\check{u}^h_{xx})_{i+\frac{1}{2}, y} = u^h_{xx}(x_{\ir}^+,y), \qquad (\widetilde{u}^h_x)_{i+\frac{1}{2}, y} = u^h_x(x_{\ir}^+, y), \qquad (\hat{u}^h)_{i+\frac{1}{2}, y} = u^h(x_{\ir}^-,y),\nonumber\\
& (\check{u}^h_{y})_{x, j+\frac{1}{2}} = u^h_{y}(x,y_{\jr}^+), \qquad (\widetilde{u}^h)_{i+\frac{1}{2}, y} = u^h(x_{\ir}^+, y), \qquad (\hat{u}^h)_{x, j+\frac{1}{2}} = u^h(x, y_{\jr}^-),
\end{align}
and the numerical fluxes on the element vertices are taken to be
\begin{align}\label{zk3.1.3.2}
(\check{u}^h_{y})_{\ir, \jr}=u^h_{y}(x_{\ir}^-,y_{\jr}^+), \qquad (\widetilde{u}^h)_{\ir,\jr}=u^h(x_{\ir}^+,y_{\jr}^-),
\end{align}
and $\hat{f}$ is taken as the global Lax-Friedrichs flux. The notations in \eqref{zk3.1.2} for $v^h$ are
\begin{align}\label{zk3.1.4}
(v^h)_{\ir,y}^{\pm} := v^h(x_{\ir}^{\pm},y), \quad (v^h)_{x, \jr}^{\pm} := v^h(x, y_{\jr}^{\pm}),\quad (v^h)_{\ir,\jr}^{\pm,\pm} := v^h(x_{\ir}^{\pm},y_{\jr}^{\pm}).
\end{align}

\begin{rem}
The numerical fluxes can also be alternatively chosen as
\begin{align}\label{zkremk}
& (\check{u}^h_{xx})_{i+\frac{1}{2}, y} = u^h_{xx}(x_{\ir}^-,y), \qquad (\widetilde{u}^h_x)_{i+\frac{1}{2}, y} = u^h_x(x_{\ir}^+, y), \qquad (\hat{u}^h)_{i+\frac{1}{2}, y} = u^h(x_{\ir}^+,y),\nonumber\\
& (\check{u}^h_{y})_{x, j+\frac{1}{2}} = u^h_{y}(x, y_{\jr}^-), \qquad
(\check{u}^h_{y})_{\ir, \jr}=u^h_{y}(x_{\ir}^-,y_{\jr}^-),\nonumber\\ &(\widetilde{u}^h)_{i+\frac{1}{2}, y} = u^h(x_{\ir}^+, y), \qquad (\widetilde{u}^h)_{\ir,\jr}=u^h(x_{\ir}^+,y_{\jr}^+), \nonumber\\
&(\hat{u}^h)_{x, j+\frac{1}{2}} = u^h(x, y_{\jr}^+).
\end{align}
\end{rem}

\subsubsection{$L^2$ stability}

We show the $L^2$ stability on the ultra-weak DG scheme (\ref{zk3.1.2})-(\ref{zk3.1.4}) for ZK equation \eqref{eq:zk} in the following proposition:
\begin{prop}\label{prop:L2-stability}
The numerical solution to the DG scheme (\ref{zk3.1.2})-(\ref{zk3.1.4}) satisfies the $L^2$ stability
\begin{align}
\frac{d}{ dt} \int_{\Omega} (u^h(x, y, t))^2 dxdy \le 0.
\end{align}
\end{prop}
\begin{proof}	
By taking $v = u^h$ in \eqref{zk3.1.2}, we obtain
\begin{align}
0 = & \int_{\Omega} u^hu^h_t dxdy - \int_{\Omega} f(u^h)u^h_x dxdy + \sum_{i=1}^{N_x}\sum_{j=1}^{N_y} \int_{J_j} \hat{f}_{i+\frac{1}{2}, y} u^h(x_{\ir}^-,y) - \hat{f}_{i-\frac{1}{2}, y}  u^h(x_{\il}^+,y) dy \nonumber\\
& +  \int_{K_{i, j}} u^h_xu^h_{xx} dxdy
 - \int_{J_j} u^h_x(x_{\ir}^+, y)u^h_x(x_{\ir}^-, y) - u^h_x(x_{\il}^+, y) u^h_x(x_{\il}^+, y) dy \nonumber\\
 &- \int_{K_{i, j}} u^h u^h_{xyy} dxdy+\int_{I_i}u_{xy}^h(x,y_{\jr}^+)u^h(x,y_{\jr}^-)-u_{xy}^h(x,y_{\jl}^+)u^h(x,y_{\jl}^+)\, dx\nonumber\\
&-(u_y^h(x_{\il}^-,y_{\jr}^+)-u_y^h(x_{\il}^+,y_{\jr}^+))u^h(x_{\il}^+,y_{\jr}^-)\nonumber\\
&+(u_y^h(x_{\il}^-,y_{\jl}^+)-u_y^h(x_{\il}^+,y_{\jl}^+))u^h(x_{\il}^+,y_{\jl}^+)\nonumber\\
&-\int_{J_j}u_y^h(x_{\ir}^+,y)u^h_y(x_{\ir}^-,y)-u_y^h(x_{\il}^+,y)u^h_y(x_{\il}^+,y)\,dy\nonumber\\
&-(u^h(x_{\ir}^+,y_{\jl}^+)-u^h(x_{\ir}^+,y_{\jl}^-))u^h_y(x_{\ir}^-,y_{\jl}^+)\nonumber\\
&+(u^h(x_{\il}^+,y_{\jl}^+)-u^h(x_{\il}^+,y_{\jl}^-))u^h_y(x_{\il}^+,y_{\jl}^+)\nonumber\\
&+\int_{I_i}u^h(x,y_{\jr}^-)u^h_{xy}(x,y_{\jr}^-)-u^h(x,y_{\jl}^-)u^h_{xy}(x,y_{\jl}^+)\, dx.\nonumber\\
= & \int_{\Omega} u^hu^h_t dxdy - \int_{\Omega} f(u^h)u^h_x dxdy + \sum_{i=1}^{N_x}\sum_{j=1}^{N_y} \int_{J_j} \hat{f}_{i+\frac{1}{2}, y}  u^h(x_{\ir}^-,y) - \hat{f}_{i-\frac{1}{2}, y}  u^h(x_{\il}^+,y) dy  \nonumber\\
&+  \int_{J_j} \frac{1}{2}  [u^h_x]^2_{i+\frac{1}{2}, y} dy+ \int_{K_{i,j}}u^h_yu^h_{xy}\,dx dy\nonumber\\
&- \int_{J_j}u_y^h(x_{\ir}^+,y)u^h_y(x_{\ir}^-,y)-u_y^h(x_{\il}^+,y)u^h_y(x_{\il}^+,y)\,dy. \nonumber
\end{align}
Denote $F(u) = \int^u f(u) du$, we can further get
\begin{align}\label{zk5}
& \frac{d}{dt} \int_{\Omega} \frac{1}{2} (u^h)^2 dxdy + \sum_{i=1}^{N_x}\sum_{j=1}^{N_y} \int_{J_j} [ F(u^h)]_{i+\frac{1}{2}, y} - \hat{f}_{i+\frac{1}{2}, j} [u^h]_{i+\frac{1}{2}, y} dy \\ \nonumber
& ~ + \int_{J_j} \frac{1}{2} [ u^h_x] ^2_{i+\frac{1}{2}, y} dy + \int_{J_j}\frac{1}{2} [u^h_y] ^2_{i+\frac{1}{2}, y} dy = 0
\end{align}
with $[u] := u^+ - u^-$. Notice that $ [F(u)] -\hat{f} [ u ] \ge 0 $ due to the monotonicity of numerical flux $\hat{f}$. Therefore, we obtain
\begin{align}\label{zk6}
 \frac{d}{dt} \int_{\Omega} \frac{1}{2} (u^h)^2 dxdy  \le 0
\end{align}
and it completes the proof.
\end{proof}

\subsubsection{Optimal error estimate}
In this subsection, we will prove the optimal error estimate of the UWDG scheme (\ref{zk3.1.2})-(\ref{zk3.1.4}) for the following simplified ZK equation
\begin{align}\label{zk_eq_err}
&u_t+u_{xyy}=0,
\end{align}
with periodic boundary conditions.

In this simplified case, the semi-discrete scheme (\ref{zk3.1.2})-(\ref{zk3.1.4}) reduces to: find $u^h\in \bfv$, such that for any test function $v^h \in \bfv$
\begin{align}\label{uwdg_zk}
&\int_{\Omega}u_t^hv^h\,dx dy=\sum_{i=1}^{N_x}\sum_{j=1}^{N_y}\mathcal{H}_{i,j}(u^h,v^h)
\end{align}
where
\begin{align}\label{hij}
\mathcal{H}_{i,j}(u^h,v^h)=&\int_{K_{i,j}}u^hv^h_{xyy} \, dx dy+\int_{I_i}u_{y}^h(x,y_{\jr}^+)v^h_x(x,y_{\jr}^-)-u_{y}^h(x,y_{\jl}^+)v^h_x(x,y_{\jl}^+)\, dx\nonumber\\
&-u_y^h(x_{\ir}^-,y_{\jr}^+)v^h(x_{\ir}^-,y_{\jr}^-)+u_y^h(x_{\il}^-,y_{\jr}^+)v^h(x_{\il}^+,y_{\jr}^-)\nonumber\\
&+u_y^h(x_{\ir}^-,y_{\jl}^+)v^h(x_{\ir}^-,y_{\jl}^+)-u_y^h(x_{\il}^-,y_{\jl}^+)v^h(x_{\il}^+,y_{\jl}^+)\nonumber\\
&-\int_{J_j}u^h(x_{\ir}^+,y)v^h_{yy}(x_{\ir}^-,y)-u^h(x_{\il}^+,y)v^h_{yy}(x_{\il}^+,y)\,dy\nonumber\\
&+u^h(x_{\ir}^+,y_{\jr}^-)v^h_y(x_{\ir}^-,y_{\jr}^-)-u^h(x_{\ir}^+,y_{\jl}^-)v^h_y(x_{\ir}^-,y_{\jl}^+)\nonumber\\
&-u^h(x_{\il}^+,y_{\jr}^-)v^h_y(x_{\il}^+,y_{\jr}^-)+u^h(x_{\il}^+,y_{\jl}^-)v^h_y(x_{\il}^+,y_{\jl}^+)\nonumber\\
&-\int_{I_i}u^h(x,y_{\jr}^-)v^h_{xy}(x,y_{\jr}^-)-u^h(x,y_{\jl}^-)v^h_{xy}(x,y_{\jl}^+)\, dx.
\end{align}
To obtain the optimal error estimates for the semi-discrete UWDG scheme \eqref{uwdg_zk}-\eqref{hij}, we first introduce a special local projection. For each index $i,j$ and $k\ge1$, we define the projection $\Pi^{\star}: H^2(K_{i,j})\rightarrow Q^k(K_{i,j})$, which satisfies
\begin{align}
&\int_{K_{i,j}}(\Pi^{\star}u-u)\varphi\,dx dy=0,\quad \forall \varphi \in P^{k-1}(I_i)\otimes P^{k-2}(J_j), \label{proj1}\\
&\int_{I_i}(\Pi^{\star}u-u)_y(x,y_{\jl}^+)\varphi(x)\, dx=0 \quad \forall \varphi\in P^{k-1}(I_i), \label{proj2}\\
&\int_{I_i}(\Pi^{\star}u-u)(x,y_{\jr}^-)\varphi(x)\, dx=0 \quad \forall \varphi\in P^{k-1}(I_i), \label{proj3}\\
&\int_{J_j}(\Pi^{\star}u-u)(x_{\il}^+,y)\varphi(y)\, dy=0 \quad \forall \varphi\in P^{k-2}(J_j), \label{proj4}\\
&\Pi^{\star}u(x_{\il}^+,y_{\jr}^-)=u(x_{\il},y_{\jr})\label{proj5}, \\
&(\Pi^{\star}u)_y(x_{\ir}^-,y_{\jl}^+)=u_y(x_{\ir},y_{\jl}). \label{proj6}
\end{align}
The projection $\Pi^{\star}$ is well-defined and holds the optimal approximation property.
\begin{prop}\label{prop:projection}
The projection $\Pi^{\star}$ defined by \eqref{proj1}-\eqref{proj6} on the cell $K_{i,j}$ exists and is unique for any smooth function $u\in H^{k+1}(K_{i,j})$, and the projection has the optimal approximation:
\begin{equation}
	\|u-\Pi^{\star}u\|_{L^2(K_{i,j})} \leq C h^{k+1}\|u\|_{H^{k+1}(K_{i,j})}
\end{equation}
where $C$ is independent of the element $K_{i,j}$ and the mesh size $h$.
\end{prop}

\begin{proof}
The proof is given in Appendix \ref{sec:appendix-projection}.
\end{proof}

Then we can prove the following error estimate.
\begin{thm}\label{thm:error-estimate}
Suppose that $u^h$ is the numerical solution of the UWDG scheme \eqref{uwdg_zk}-\eqref{hij} and the exact solution to the ZK equation \eqref{zk_eq_err} $u(x,y,t)\in C^{1}(0,T;H^{k+1}(\Omega))$, then the $L^2$-error satisfies the following estimation
\begin{align}\label{opterror}
\|u(\cdot,T) - u^h(\cdot,T)\|_{L^2(\Omega)} \leq C(T+1)h^{k+1}\sup_{0\leq t\leq T}\|u_t(\cdot,t)\|_{H^{k+1}(\Omega)},
\end{align}
where $k\geq 1$ is the degree of the piecewise tensor product polynomials in finite element spaces $\bfv$, and the constant $C$ only depends on $k$ but is independent of the mesh size $h$.
\end{thm}
\begin{proof}
Denote the error by $e:=u-u^h = (u-\Pi^{\star}u) + (\Pi^{\star}u-u^h)$.
% \begin{align}
% u-u^h=.
% \end{align}
Thanks to the consistency of the scheme \eqref{uwdg_zk}, we have the following error equation
\begin{align}\label{errorequ1}
\int_{\Omega}(u-u^h)_tv\,dx dy=\sum_{i=1}^{N_x}\sum_{j=1}^{N_y}\mathcal{H}_{i,j}(u-u^h,v)\quad \forall v\in \bfv.
\end{align}
Thus, we have
\begin{align}\label{errorequ2}
&\int_{\Omega}(\Pi^{\star}u-u^h)_tv\,dx dy-\sum_{i=1}^{N_x}\sum_{j=1}^{N_y}\mathcal{H}_{i,j}(\Pi^{\star}u-u^h,v)\nonumber\\
&=\int_{\Omega}(\Pi^{\star}u-u)_tv\,dx dy-\sum_{i=1}^{N_x}\sum_{j=1}^{N_y}\mathcal{H}_{i,j}(\Pi^{\star}u-u,v)\quad \forall v\in \bfv.
\end{align}
From the definition of the projection $\Pi^{\star}u$ in \eqref{proj1}-\eqref{proj6} and the bilinear form \eqref{hij}, it is easy to verify
\begin{align}
\mathcal{H}_{i,j}(\Pi^{\star}u-u,v)=0, \quad \forall v\in Q^k(K_{i,j}),\forall i,j.
\end{align}
Next, by taking $v=\Pi^{\star}u-u^h\in \bfv$ in \eqref{errorequ2} and applying the Cauchy-Schwarz inequality, we have the estimate
\begin{align}
\frac{1}{2}\frac{d}{dt}\|\Pi^{\star}u-u^h\|^2 &\leq \|\Pi^{\star}u_t-u_t\|\|\Pi^{\star}u-u^h\|\nonumber\\
&\leq Ch^{k+1}\|u_t\|_{H^{k+1}(\Omega)}\|\Pi^{\star}u-u^h\|
\end{align}
Here, we also use the $L^2$ stability in Proposition \ref{prop:L2-stability}.
Finally, by integrating the above equation over $t\in[0,T]$ and using the initial projection, we obtain the optimal error estimate \eqref{opterror}.
\end{proof}

\section{Adaptive multiresolution UWDG schemes}\label{sec:method}

In this section, we will present our adaptive multiresolution UWDG schemes for the KdV equation and the ZK equation.

\subsection{Multiresolution analysis and multiwavelets}\label{sec:mra}

We first review the fundamentals of MRA of DG approximation spaces and the associated multiwavelets. The $L^2$ orthonormal Alpert's multiwavelets \cite{alpert1993class} are presented and will be used later. We also introduce a set of key notations.

Our construction of the UWDG schemes starts with the hierarchical decomposition of the DG finite element space. Without loss of generality, we assume the computational domain to be the unit interval $\Omega=[0,1]$. 
We define a set of nested grids $\Omega_0,\,\Omega_1,\ldots$ on $\Omega$, for which the $n$-th level grid $\Omega_n$ consists of $2^n$ uniform cells:
\begin{equation*}
I_{n}^j=(2^{-n}j, 2^{-n}(j+1)], \quad j=0, \ldots, 2^n-1.
\end{equation*}
The piecewise polynomial space of degree at most $k\ge1$ on grid $\Omega_n$ for $n\ge 0$ is denoted by
\begin{equation}\label{eq:DG-space-Vn}
V_n^k:=\{v\in L^2(\Omega): v \in P^k(I_{n}^j),\, \forall \,j=0, \ldots, 2^n-1\}.
\end{equation}
Observing the nested structure
$$V_0^k \subset V_1^k \subset V_2^k \subset V_3^k \subset  \cdots,$$
we can define the multiwavelet subspace $W_n^k$, $n=1, 2, \ldots $ as the orthogonal complement of $V_{n-1}^k$ in $V_{n}^k$ with respect to the $L^2$ inner product on $[0,1]$, i.e.,
\begin{equation*}
V_{n-1}^k \oplus W_n^k=V_{n}^k, \quad W_n^k \perp V_{n-1}^k.
\end{equation*}
By letting $W_0^k:=V_0^k$, we obtain a hierarchical decomposition $V_n^k=\bigoplus_{0 \leq l \leq n} W_l^k$, i.e., MRA of space $V_n^k$.
A set of orthonormal basis can be defined on $W_l^k$ as follows. When $l=0$, the basis $v^0_{i,0}(x)$, $ i=0,\ldots,k$ are the normalized shifted Legendre polynomials in $[0,1]$. When $l>0$, we will use the Alpert's orthonormal multiwavelets \cite{alpert1993class} as the bases, which have been employed to develop a class of sparse grid DG methods for solving PDEs in high dimensions \cite{wang2016elliptic,guo2016sparse, tao2020Schrodinger,guo2020hj}. In this paper, we adopt the notation
$$v^j_{i,l}(x),\quad i=0,\ldots,k,\quad j=0,\ldots,2^{l-1}-1.$$
for Alpert's multiwavelets.

We then follow a tensor-product approach to construct the hierarchical finite element space in multi-dimensional space.  Denote $\bl=(l_1,\cdots,l_d)\in\mathbb{N}_0^d$ as the mesh level in a multivariate sense, where $\mathbb{N}_0$ denotes the set of nonnegative integers. Then we can define the tensor-product mesh grid $\Omega_\bl=\Omega_{l_1}\otimes\cdots\otimes\Omega_{l_d}$ and the corresponding mesh size $h_\bl=(h_{l_1},\cdots,h_{l_d}).$ Based on the grid $\Omega_\bl$, we denote  $I_\bl^\bj=\{\bx:x_m\in(h_mj_m,h_m(j_{m}+1)),m=1,\cdots,d\}$ as an elementary cell, and
$$\bV_\bl^k:=\{\bv:  \bv \in Q^k(I^{\bj}_{\bl}), \,\,  \bzero \leq \bj  \leq 2^{\bl}-\bone \}= V_{l_1,x_1}^k\times\cdots\times  V_{l_d,x_d}^k$$
as the tensor-product piecewise polynomial space, where $Q^k(I^{\bj}_{\bl})$ represents the collection of polynomials of degree up to $k$ in each dimension on cell $I^{\bj}_{\bl}$.
If we use equal mesh refinement of size $h_N=2^{-N}$ in each coordinate direction, the  grid and space will be denoted by $\Omega_N$ and $\bV_N^k$, respectively.
Based on a tensor-product construction, the multi-dimensional increment space can be  defined as
$$\bW_\bl^k=W_{l_1,x_1}^k\times\cdots\times  W_{l_d,x_d}^k.$$
The basis functions in multi-dimensions are defined as
\begin{equation}\label{eq:multidim-basis}
v^\bj_{\bi,\bl}(\bx) := \prod_{m=1}^d v^{j_m}_{i_m,l_m}(x_m),
\end{equation}
for $\bl \in \mathbb{N}_0^d$, $\bj \in B_\bl := \{\bj\in\mathbb{N}_0^d: \,\mathbf{0}\leq\bj\leq\max(2^{\bl-\mathbf{1}}-\mathbf{1},\mathbf{0}) \}$ and $\mathbf{1}\leq\bi\leq \bk+\mathbf{1}$.

Using the notation of $$
|\bl|_1:=\sum_{m=1}^d l_m, \qquad   |\bl|_\infty:=\max_{1\leq m \leq d} l_m.
$$
and  the same component-wise arithmetic operations and relations   as defined in \cite{wang2016elliptic},  we reach the decomposition
\begin{equation}\label{eq:hiere_tp}
\bV_N^k=\bigoplus_{\substack{ |\bl|_\infty \leq N\\\bl \in \mathbb{N}_0^d}} \bW_\bl^k.
\end{equation}
When $d=2$, it is easy to see that $\bV_N^k$ is the same space of $\bV_h^k$ in Section \ref{sec.zk} with uniform partition and $N_x = N_y =2^N$.
On the other hand, a standard choice of sparse grid   space  \cite{wang2016elliptic, guo2016sparse} is
\begin{equation}
\label{eq:hiere_sg}
\hat{\bV}_N^k=\bigoplus_{\substack{ |\bl|_1 \leq N\\\bl \in \mathbb{N}_0^d}}\bW_\bl^k \subset \bV_N^k.
\end{equation}
We skip the details about the property of the space, but refer the readers to \cite{wang2016elliptic, guo2016sparse}. In Section \ref{subsec:scheme}, we will describe the adaptive scheme which adapts a subspace of $\bV_N^k$ according to the numerical solution, hence offering more flexibility and efficiency.

For nonlinear convection terms in the KdV equation \eqref{eq:kdv} and the ZK equation \eqref{eq:zk}, we use the interpolatory multiwavelets based on Hermite interpolations introduced in \cite{tao2021collocation}. The treatment of the nonlinear convection terms is the same as that in the adaptive multiresolution DG scheme for solving conservation laws in \cite{huang2019adaptive}. For saving space, we omit the details and refer readers to \cite{tao2021collocation,huang2019adaptive}.

\subsection{Semi-discrete schemes}\label{subsec:scheme}
Based on the construction of MRA of DG approximation space, we are ready to present the adaptive multiresolution UWDG schemes for simulating the KdV equation \eqref{eq:kdv} and the ZK equation \eqref{eq:zk}. For illustrative purposes, we first introduce some basis notation about jumps and averages for piecewise functions defined on a grid $\Omega_N$. Denote by $\Gamma$ the union of the boundaries for all the elements in the partition $\Omega_N$. In 2D case, $\Gamma$ is further decomposed into two parts: $\Gamma = \Gamma_x \cup \Gamma_y$ with $\Gamma_x$ and $\Gamma_y$ are the union of the boundaries in $x$- and $y$- directions, respectively.
The jump and average of $q\in L^2(\Gamma)$ and $\textbf{q}\in [L^2(\Gamma)]^d$ are defined as follows. Suppose $e$ is an edge (degenerate to a point in 1D) shared by elements $T^+$ and $T^-$, we define the unit normal vectors $\textbf{n}^+$ and $\textbf{n}^-$ on $e$ pointing exterior to $T^+$ and $T^-$, and then
\begin{align}
[q]=q^- \textbf{n}^- + q^+ \textbf{n}^+, \qquad & \{ q\} = \frac{1}{2} (q^-+q^+),  \nonumber  \\
[\textbf{q}] = \textbf{q}^- \cdot \textbf{n}^-  + \textbf{q}^+ \cdot \textbf{n}^+ ,  \qquad & \{ \textbf{q}\} = \frac{1}{2}(\textbf{q}^- + \textbf{q}^+). \nonumber
\end{align}
 Moreover, in 2D case, we denote $S$ the set of all the vertices in the partition $\Omega_N$. For any $p = (x_p, y_p) \in S$, we denote
\begin{equation}\label{zk_4.2.0}
\{[v]\}_p = -v(x_p^-, y_p^-)-v(x_p^+, y_p^+) + v(x_p^-, y_p^+) + v(x_p^-, y_p^+)
\end{equation}
which will be used for the special treatment in the UWDG scheme for the ZK equation.

The semi-discrete multiresolution UWDG scheme for KdV equation is to find $u^h \in \bV$, such that for any test function $v^h \in \bV$,
\begin{align}\label{eq:kdv-weak-formulation}
\int_{\Omega} (u^h)_t v^h dx ={}& \int_{\Omega} \mathcal{I}_h[f(u^h)] v^h_x dx + \sum_{e\in\Gamma} \left(\mathcal{I}_h[\hat{f}(u^h)] [v^h]  \right)_e + \int_{\Omega} u^h v^h_{xxx} dx \nonumber \\
& + \sum_{e\in\Gamma} \left(\hat{u}^h [v^h_{xx}] \right)_e - \sum_{e\in\Gamma} \left( (\widetilde{u}^h_x)[v^h_x]  \right)_e + \sum_{e\in\Gamma} \left((\check{u}^h_{xx}) [v^h] \right)_e,
\end{align}
where the choices of $\hat{f}$, $\hat{u}^h$, $\widetilde{u}^h_x$ and  $\check{u}^h_{xx}$ are the same as in (\ref{kdv3_00})-(\ref{kdv3_01}). Here, $\bV$ is a subspace of $\bV_N^k$ which dynamically evolves over time \cite{guo2017adaptive}. The adaptive procedure follows the technique developed in \cite{guo2017adaptive} to determine the UWDG function space that dynamically evolves over time. The main idea is that in light of the distinguished property of the orthonormal multiwavelets, we keep track of multiwavelet coefficients as a natural error indicator for refining and coarsening, aiming to efficiently capture the soliton solutions. For the details, we refer readers to \cite{guo2017adaptive,huang2019adaptive}.

In \eqref{eq:kdv-weak-formulation}, we follow the approach in \cite{huang2019adaptive} and apply the multiresolution Hermite interpolation $\mathcal{I}_h$ to efficiently compute the nonlinear terms. It is required that the polynomial degree of Hermite interpolation $M \geq k+1$ \cite{huang2019adaptive}. For example, if we take quadratic polynomials for the DG space, then it is required to apply a cubic interpolation operator to treat the nonlinear terms.

Similarly, the semi-discrete adaptive multiresolution UWDG for ZK equation is to find $u^h \in \bV$, such that for any test function $v^h \in \bV$, the following equation holds,
\begin{align}\label{zk2}
 \int_{\Omega} u^h_t v^h dxdy & = \int_{\Omega} \mathcal{I}_h[f(u^h)]v^h_x dxdy + \int_{\Omega} u^hv^h_{xxx} dxdy +\int_{\Omega} u^h v^h_{xyy} dxdy \\ \nonumber
 & +\sum_{e\in \Gamma_y} \int_{e} \mathcal{I}_h[\hat{f}] [v^h] +  \check{u}^h_{xx} [v^h] - \widetilde{u}^h_{x} [v^h_x] + \hat{u}^h[v_{xx}^h] + \widetilde{u}^h[v^h_{yy}]  ds \\ \nonumber
 & - \sum_{e \in \Gamma_x} \int_{e} \check{u}^h_{y}[v_x^h] -  \hat{u}^h [v^h_{xy}] ds
  +\sum_{p \in S}( \check{u}^h_{y} \{[v^h] \}  -\widetilde{u}^h\{[v^h]\} )_p, \nonumber
\end{align}
Here, the numerical fluxes are given in \eqref{zk3.1.3.1}-\eqref{zk3.1.3.2}. A sparse grid UWDG scheme can be defined similarly if $\bV=\hat{\bV}_N^k.$

\section{Numerical examples}\label{sec:result}

In this section, we present several numerical examples to demonstrate the performance of the proposed adaptive multiresolution UWDG schemes for solving the KdV equation and the ZK equation. For the time discretization, we employ the third-order implicit-explicit (IMEX) Runge-Kutta (RK) method in \cite{pareschi2005implicit}. All adaptive calculations are obtained by $k=2$, unless otherwise stated. The maximum mesh level $N$ is set to be 8.  $\textrm{DoF}=\textrm{dim}(\bV)$ refers to the number of Alpert's multiwavelets basis functions in the adaptive grids. The coarsening threshold $\eta$ is taken to be $\epsilon/10$ with $\epsilon$ to be the refinement threshold \cite{guo2017adaptive}.

\subsection{KdV equation}

\begin{exam}[accuracy test for the KdV equation]\label{exam:accuracy-kdv-1d}
We first test accuracy of our scheme for the nonlinear KdV equation on the domain $[0,1]$:
	\begin{equation}
		u_t + \brac{\frac{u^2}{2}}_x + u_{xxx} = s(x,t).
	\end{equation}	
	By adding the additional source term
	\begin{equation}
		s(x,t) = 2\pi\cos(2\pi(x-t)) (-4\pi^2-1 + \sin(2\pi(x-t))),
	\end{equation}
	we have an explicit exact solution to test the accuracy:
	\begin{equation}
		u(x,t) = \sin(2\pi(x-t))
	\end{equation}
	The periodic boundary condition is applied here.

Similar as in \cite{bokanowski2013adaptive,guo2017adaptive}, two types rates of convergence will be investigated. The first one is the convergence rate with respect to the error thresold:
\begin{equation*}
    R_{\epsilon_l} = \frac{\log(e_{l-1}/e_l)}{\log(\epsilon_{l-1}/\epsilon_l)},
\end{equation*}
where $e_l$ is the standard $L^2$ error with refinement parameter $\epsilon$.
The second one is the convergence rate with respect to degrees of freedom:
\begin{equation*}
    R_{\textrm{DoF}_l} = \frac{\log(e_{l-1}/e_l)}{\log(\textrm{DoF}_{l}/\textrm{DoF}_{l-1})}.
\end{equation*}

For reference, the numerical results with UWDG scheme on full grid for $k=2$ are shown in Table \ref{tab:accuracy-kdv-full}. We observe clear third-order accuracy. Table \ref{tab:smooth-adaptive-kdv} presents the accuracy of the UWDG schemes with
adaptivity, from which we can observe the effectiveness of our adaptive scheme.
\begin{table}[!hbp]
\centering
\caption{Example \ref{exam:accuracy-kdv-1d}: accuracy test for KdV equation in 1D. Full grid, $k=2$. $t=0.1$.}
\label{tab:accuracy-kdv-full}
\begin{tabular}{c|c|c|c|c|c|c|c}
  \hline
  & $N$ & $L^1$-error & order & $L^2$-error & order & $L^{\infty}$-error & order \\
  \hline
\multirow{5}{3em}{$k$ = 2}
& 2 & 2.67e-01 & - & 3.47e-01 & - & 6.56e-01 \\
& 3 & 3.19e-02 & 3.07 & 3.82e-02 & 3.18 & 6.90e-02 & 3.25 \\
& 4 & 2.46e-03 & 3.70 & 2.78e-03 & 3.78 & 5.94e-03 & 3.54 \\
& 5 & 2.88e-04 & 3.09 & 3.32e-04 & 3.07 & 8.96e-04 & 2.73 \\
& 6 & 3.58e-05 & 3.01 & 4.15e-05 & 3.00 & 1.15e-04 & 2.96 \\
\hline
\end{tabular}
\end{table}

\end{exam}
\begin{table}[!hbp]
\centering
\caption{Example \ref{exam:accuracy-kdv-1d}: accuracy test for KdV equation in 1D.  Adaptive scheme, $k=2$, $t=0.1$.}
\label{tab:smooth-adaptive-kdv}
\begin{tabular}{c|c|c|c|c|c}
  \hline
  & $\epsilon$ & DoF & $L^2$-error & $R_{\textrm{DoF}}$ & $R_{\epsilon}$ \\
  \hline
\multirow{4}{3em}{$k=2$}
& 1e-2 & 24 & 3.82e-2 & - & - \\
& 1e-3 & 48 & 2.78e-3 & 3.78 & 1.14 \\
& 1e-4 & 90 & 7.14e-4 & 2.16 & 0.59 \\
& 1e-5 & 180 & 5.60e-5 & 3.67 & 1.11 \\
\hline
\end{tabular}
\end{table}

\begin{exam}[solitons for the KdV equation]\label{exam:kdv-1d}
In this example, we consider the nonlinear KdV equation on the domain $[0,1]$:
\begin{equation}
  u_t + \brac{\frac{u^2}{2}}_x + \sigma u_{xxx} = 0
\end{equation}
with periodic boundary conditions.

We first consider the single soliton case with the exact solution \cite{debussche1999numerical}:
\begin{equation}
  u(x,t) = 3c\sech^2(\kappa(x-ct-x_0))
\end{equation}
with $c=0.3$, $x_0=0.5$, $\sigma=5\times10^{-4}$ and $\kappa=\frac{1}{2}({c}/{\sigma})^{1/2}$. The numerical solutions and the active elements at $t=0.1$ and $t=0.8$ are presented in Figure \ref{fig:kdv-1d-single}. We can see that our adaptive algorithm can capture the evolution of solitons with steep gradients efficiently. The active elements are moving with the soliton.
\begin{figure}
    \centering
    \subfigure[numerical and exact solutions at $t=0.1$]{
    \begin{minipage}[b]{0.46\textwidth}
    \includegraphics[width=1\textwidth]{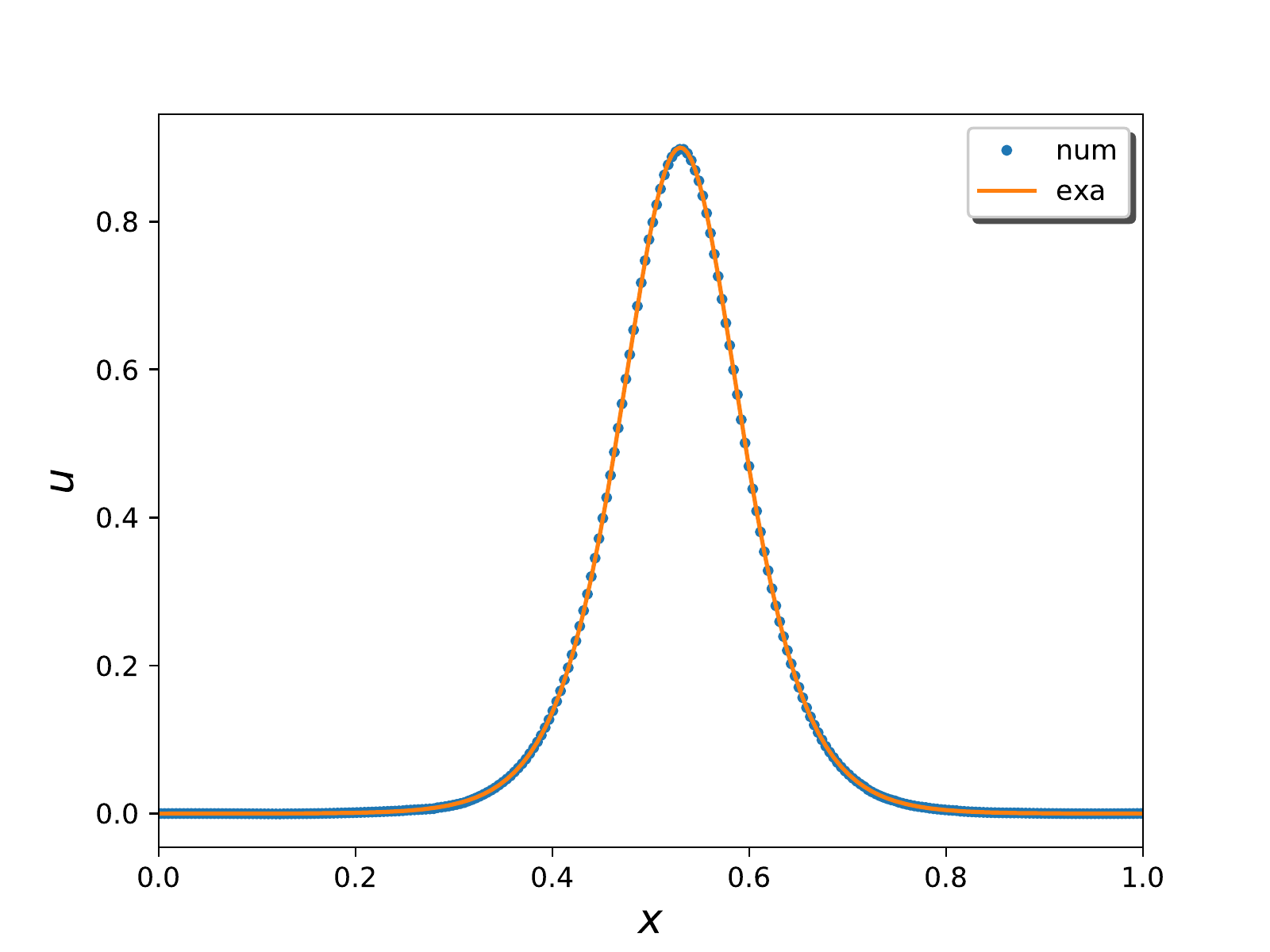}
    \end{minipage}
    }
    \subfigure[active elements at $t=0.1$]{
    \begin{minipage}[b]{0.46\textwidth}
    \includegraphics[width=1\textwidth]{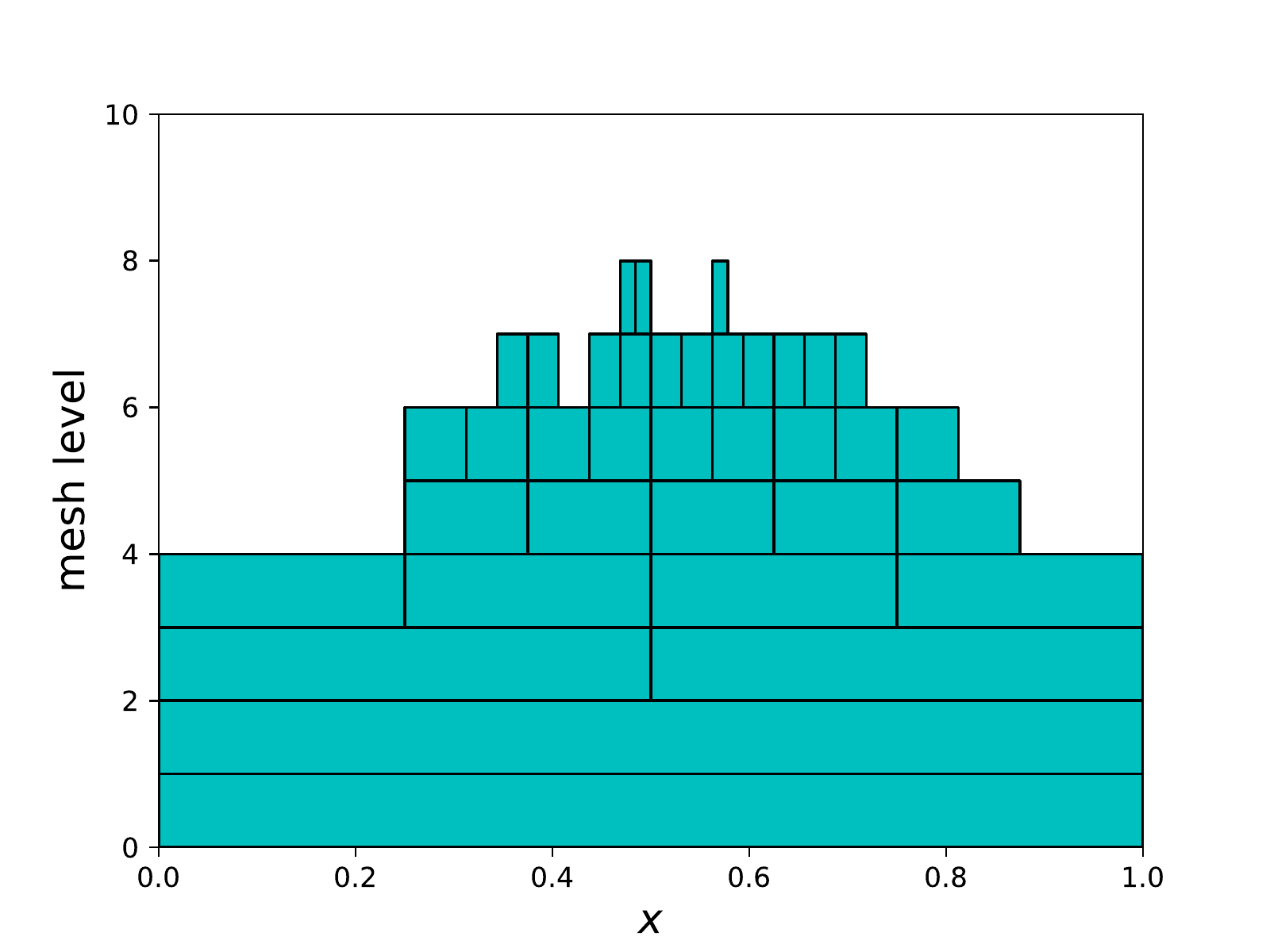}
    \end{minipage}
    }
    \bigskip
    \subfigure[numerical and exact solutions at $t=0.8$]{
    \begin{minipage}[b]{0.46\textwidth}
    \includegraphics[width=1\textwidth]{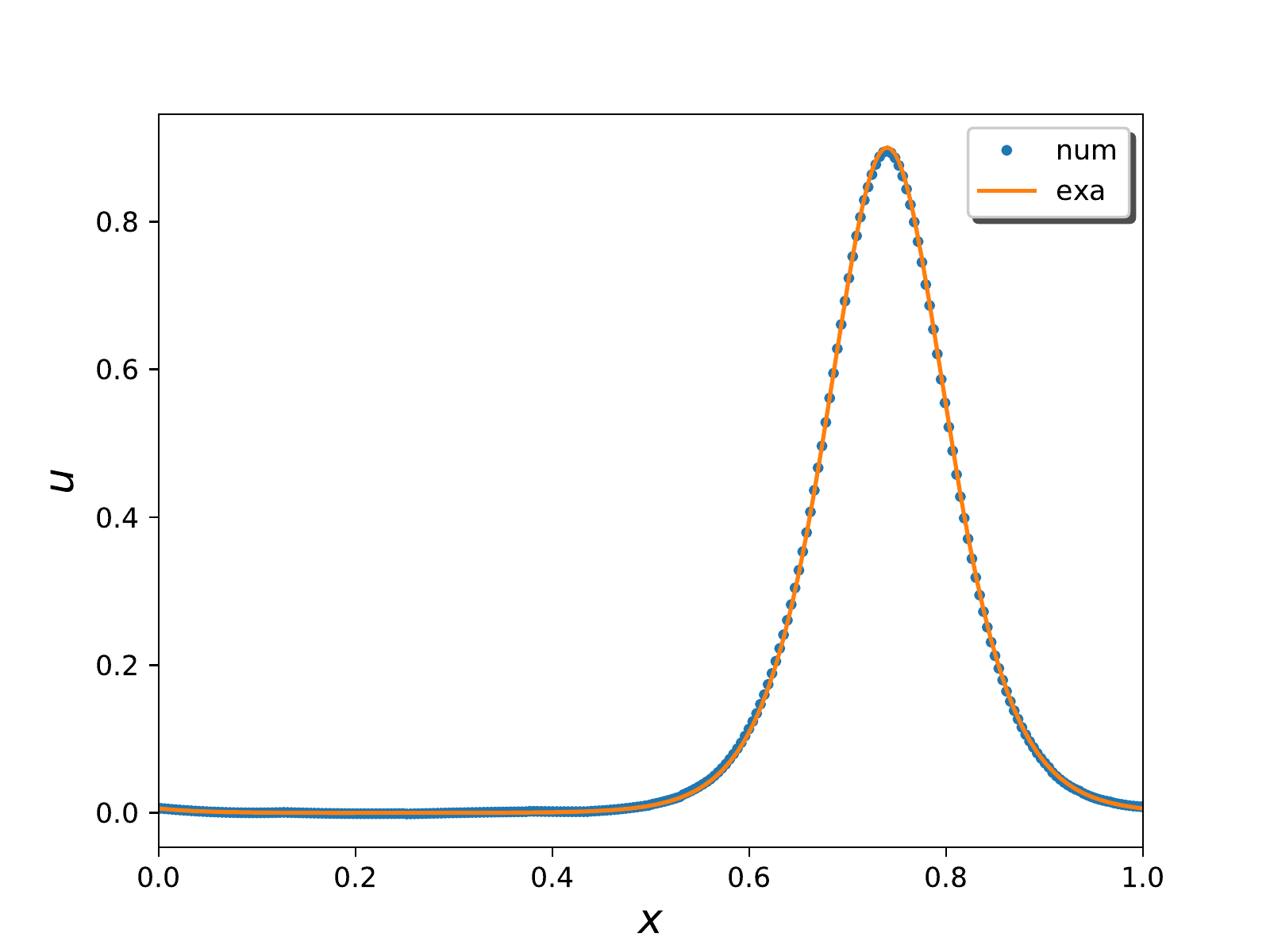}
    \end{minipage}
    }
    \subfigure[active elements at $t=0.8$]{
    \begin{minipage}[b]{0.46\textwidth}
    \includegraphics[width=1\textwidth]{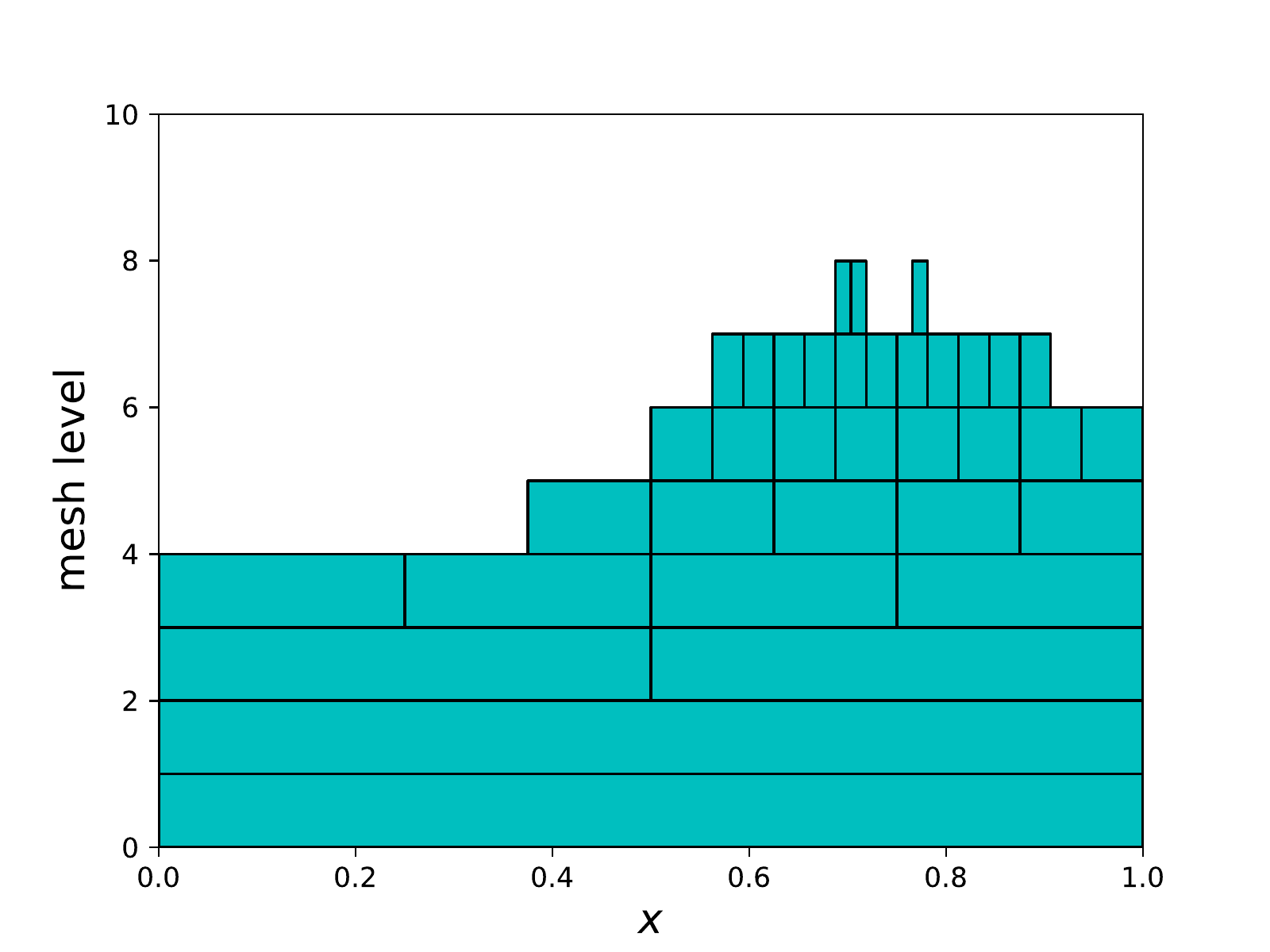}
    \end{minipage}
    }
    \caption{Example \ref{exam:kdv-1d}: nonlinear KdV equation in 1D, single soliton. $t=0.1$ and $t=0.8$. $N=8$ and $\epsilon=10^{-4}$. Left: numerical and exact solutions at $t=0.1$ and $t=0.8$; right: active elements at $t=0.1$ and $t=0.8$.}
    \label{fig:kdv-1d-single}
\end{figure}

Next, we consider the double soliton collision which has the initial condition \cite{debussche1999numerical}:
\begin{equation}
  u(x,0) = 3c_1 \sech^2(\kappa_1(x-x_1)) + 3c_2 \sech^2(\kappa_2(x-x_2)),
\end{equation}
with $c_1=0.3$, $c_2=0.1$, $x_1=0.45$, $x_2=0.65$, $\sigma=1.21\times10^{-4}$ and $\kappa_i=\frac{1}{2}{({c_i}/{\sigma})^{1/2}}$ for $i=1,2$. The numerical solutions and the active elements at $t=0$, 0.7 and 1 are shown in Figure \ref{fig:kdv-1d-double}. It is observed that our adaptive scheme is able to simulate a clean interaction where no dispersive tail or supplementary soliton are created.
\begin{figure}
    \centering
    \subfigure[numerical solution at $t=0$]{
    \begin{minipage}[b]{0.46\textwidth}
    \includegraphics[width=1\textwidth]{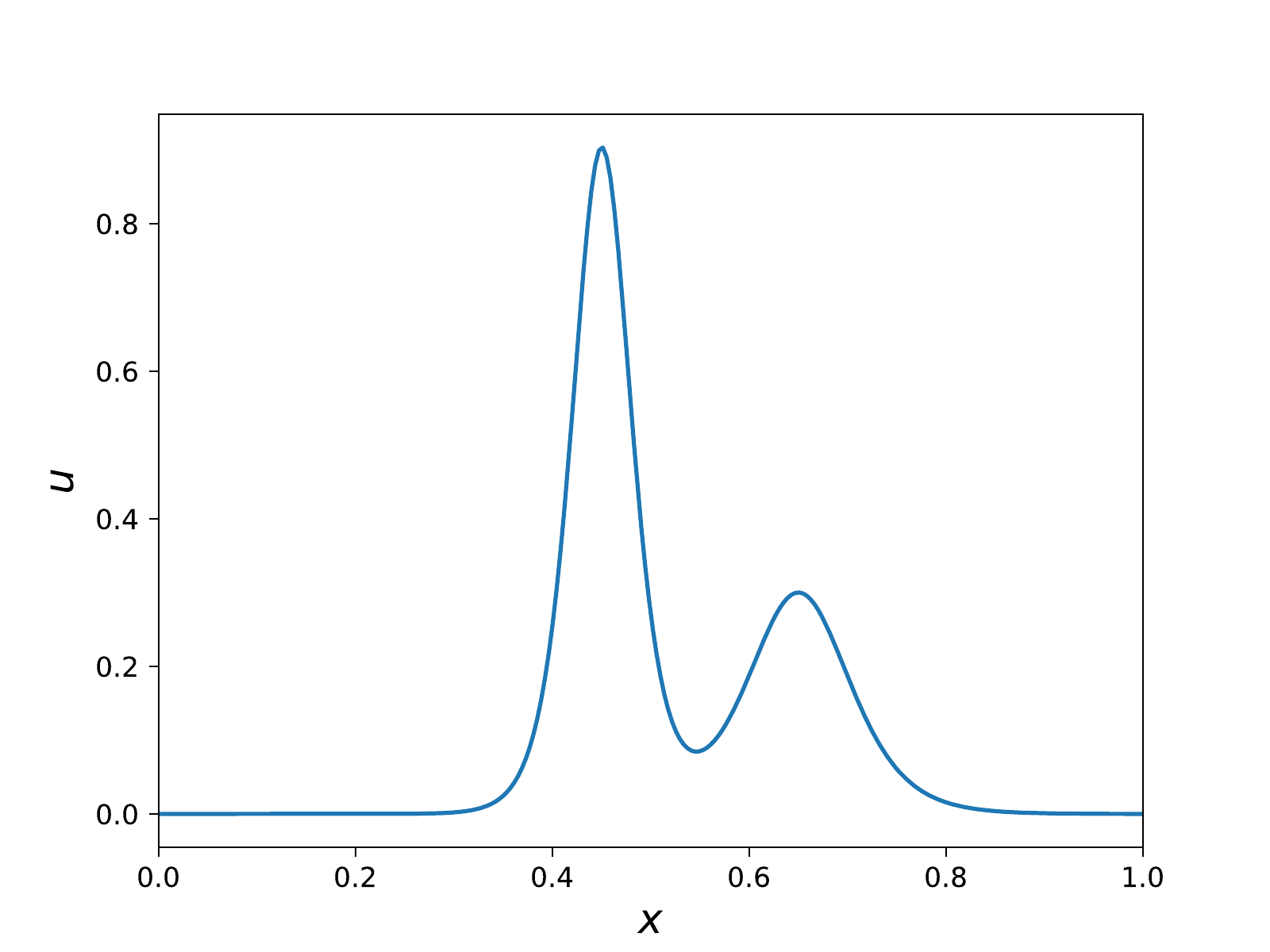}
    \end{minipage}
    }
    \subfigure[active elements at $t=0$]{
    \begin{minipage}[b]{0.46\textwidth}
    \includegraphics[width=1\textwidth]{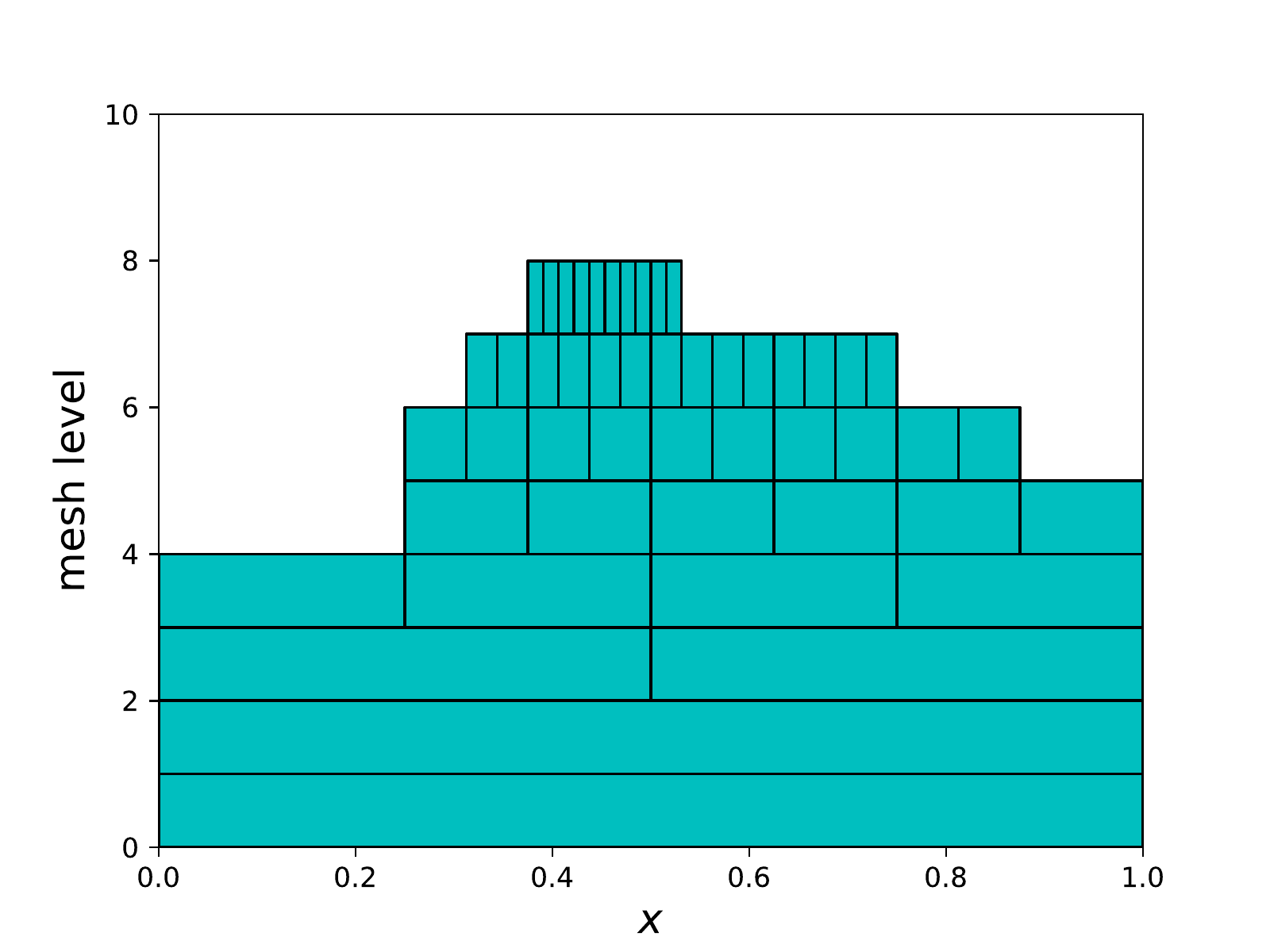}
    \end{minipage}
    }
    \bigskip
    \subfigure[numerical solution at $t=0.7$]{
    \begin{minipage}[b]{0.46\textwidth}
    \includegraphics[width=1\textwidth]{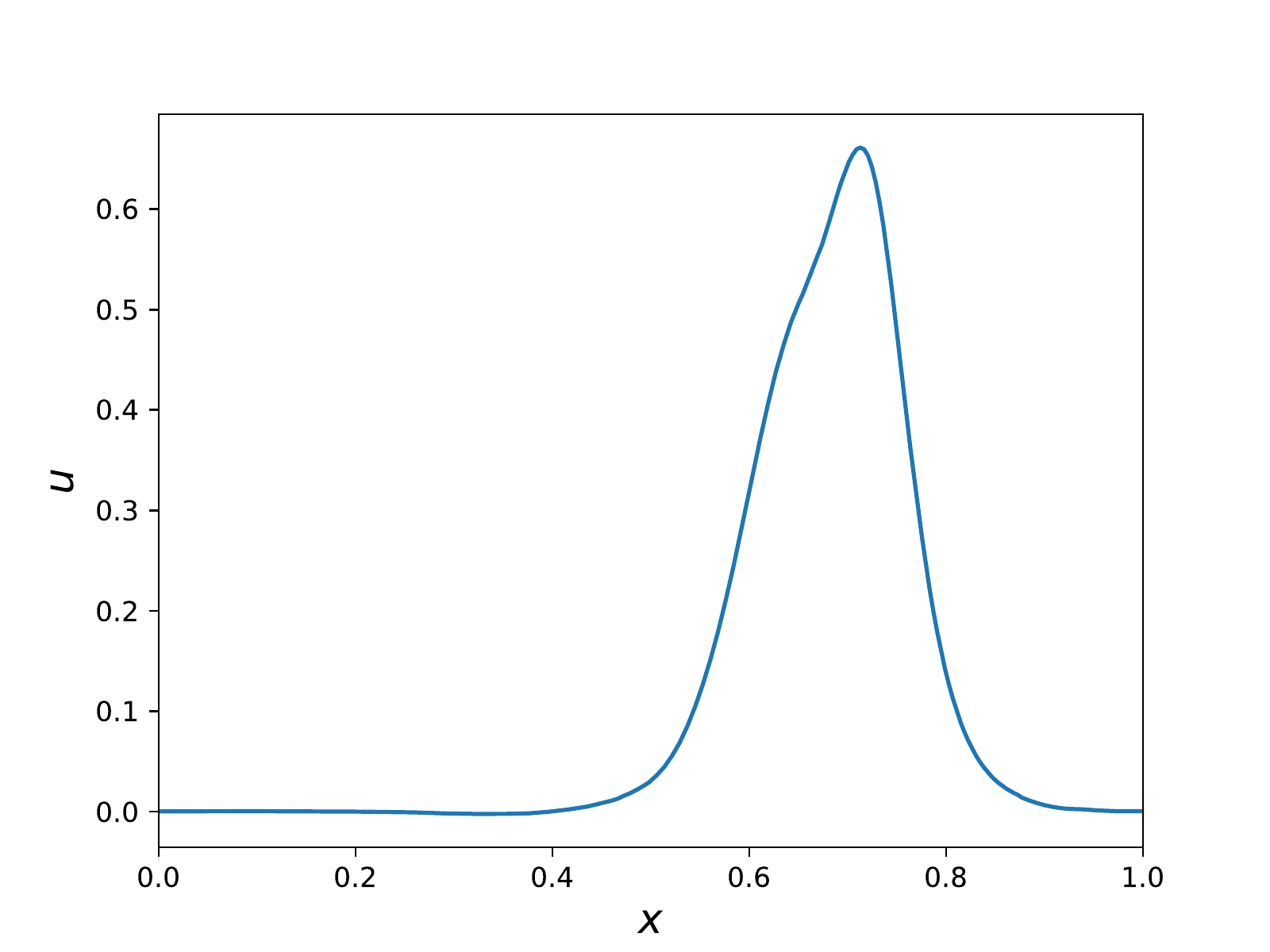}
    \end{minipage}
    }
    \subfigure[active elements at $t=0.7$]{
    \begin{minipage}[b]{0.46\textwidth}
    \includegraphics[width=1\textwidth]{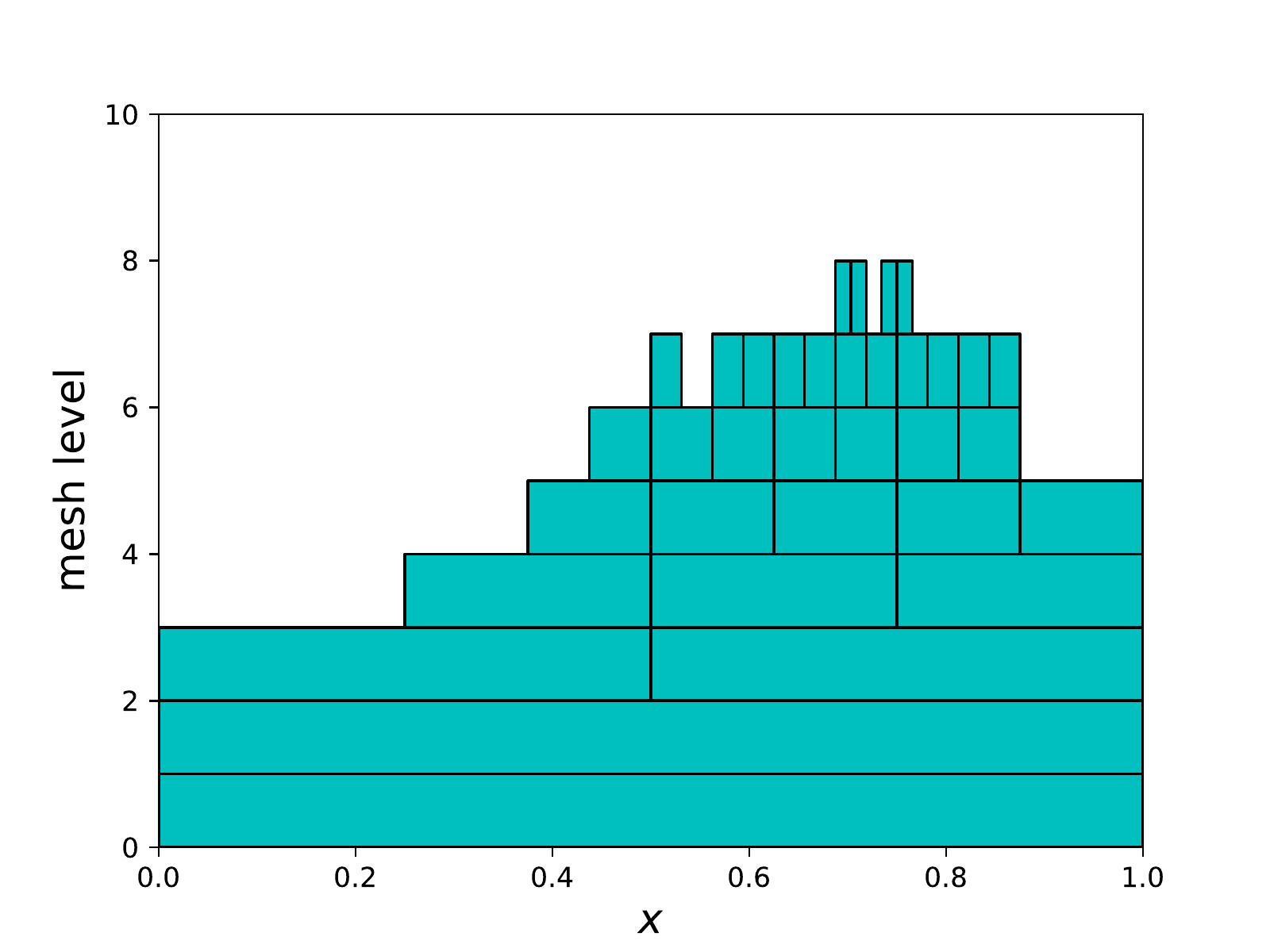}
    \end{minipage}
    }
    \bigskip
    \subfigure[numerical solution at $t=1$]{
    \begin{minipage}[b]{0.46\textwidth}
    \includegraphics[width=1\textwidth]{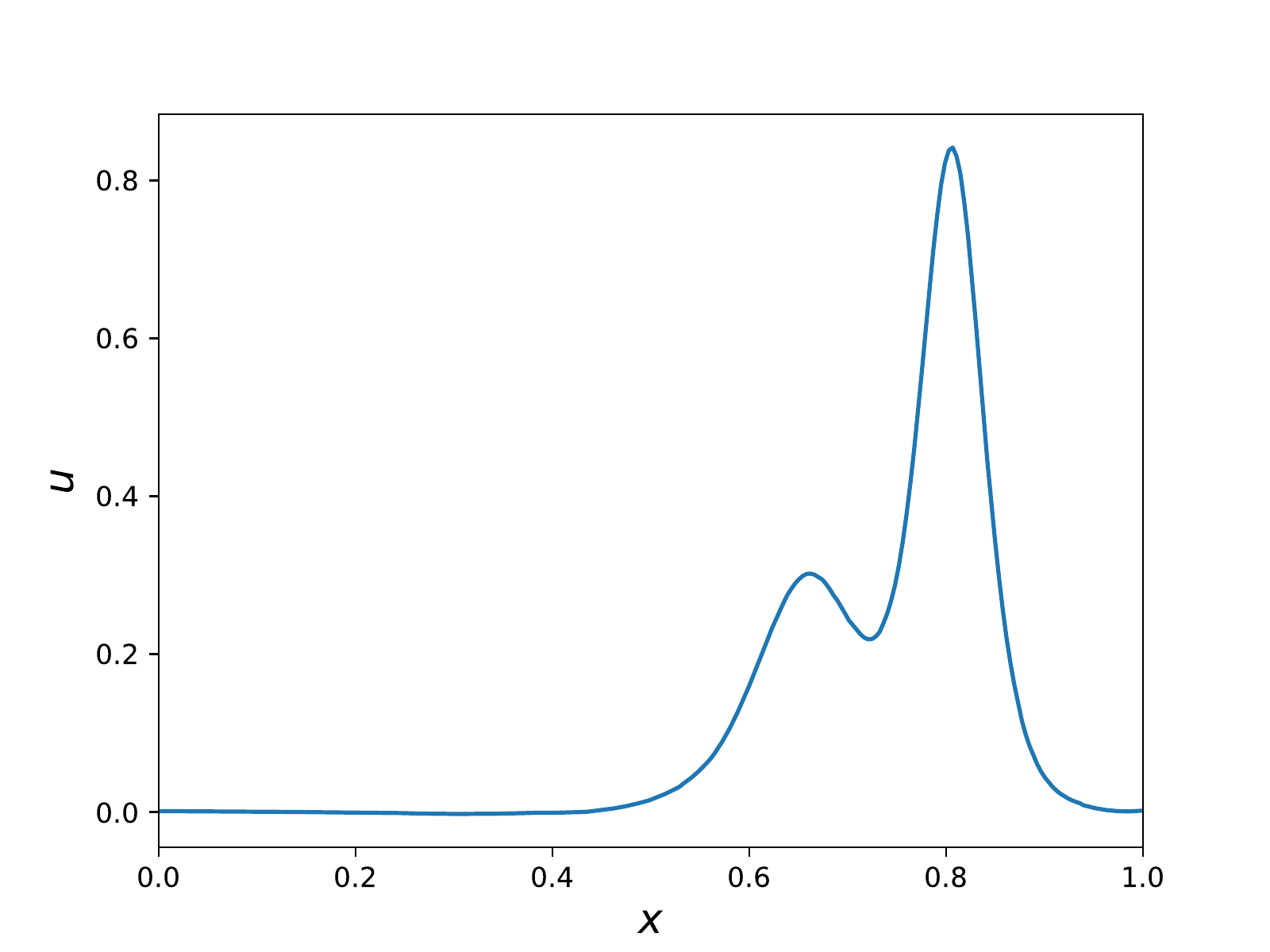}
    \end{minipage}
    }
    \subfigure[active elements at $t=1$]{
    \begin{minipage}[b]{0.46\textwidth}
    \includegraphics[width=1\textwidth]{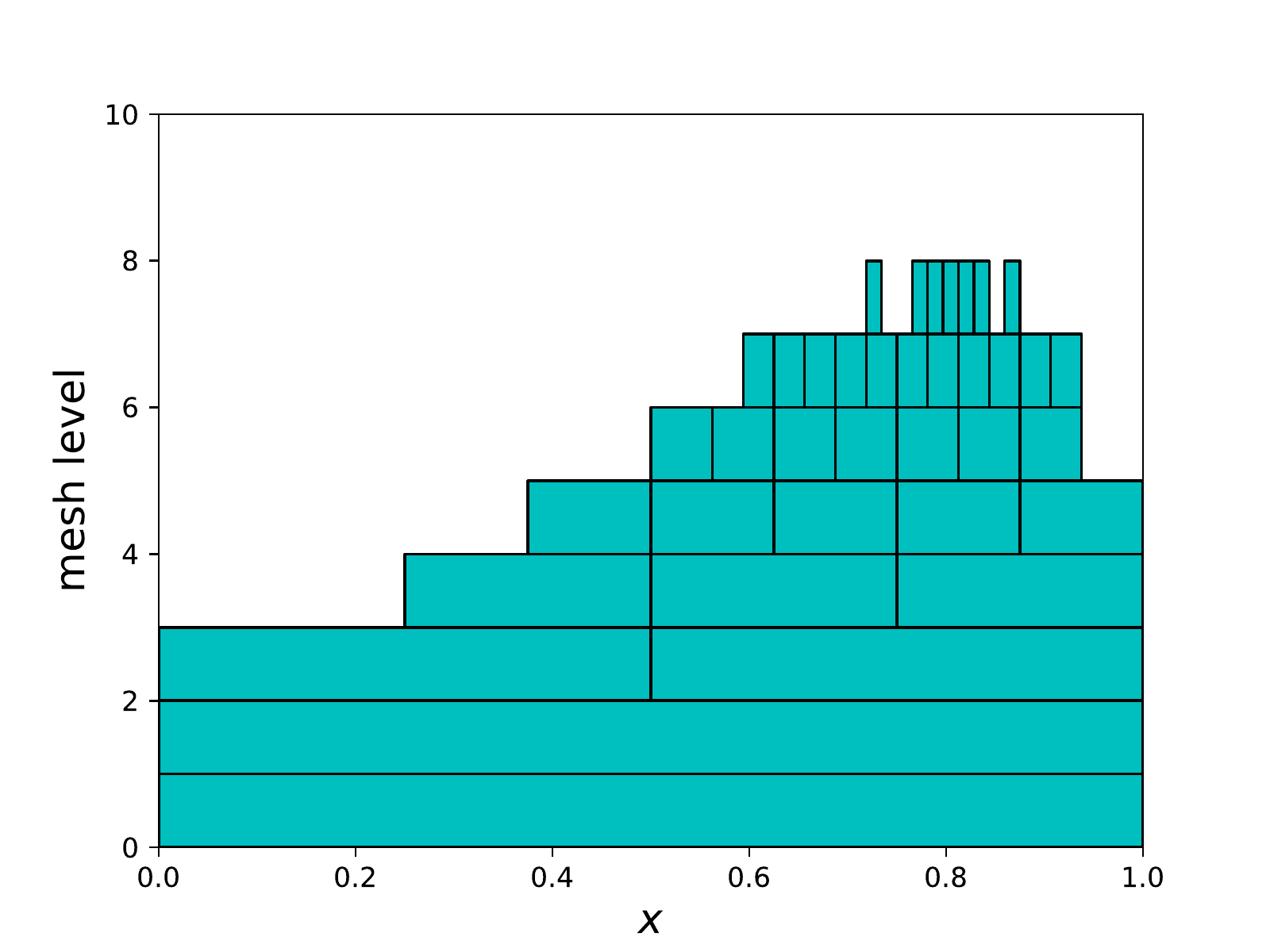}
    \end{minipage}
    }
    \caption{Example \ref{exam:kdv-1d}: nonlinear KdV equation in 1D, double soliton. $t=0$, 0.7 and 1. $N=8$ and $\epsilon=10^{-4}$. Left: numerical solutions at $t=0$, 0.7 and 1; right: active elements at $t=0$, 0.7 and 1.}
    \label{fig:kdv-1d-double}
\end{figure}

The last case in this example is the triple soliton splitting which has the initial condition \cite{yan2002local}
\begin{equation}
	u(x,0) = \frac{2}{3} \sech^2\brac{\frac{x-x_0}{\sqrt{108\sigma}}},
\end{equation}
with $x_0=0.5$ and $\sigma=2.5\times10^{-5}$. The numerical solutions and the active elements at $t=0$, 0.5 and 1 are shown in Figure \ref{fig:kdv-1d-triple}. Again, we observe that our adaptive algorithm can capture the steep gradients of the soliton efficiently. Moreover, the solution profiles are comparable to the results in \cite{yan2002local}.
\begin{figure}
    \centering
    \subfigure[numerical solution at $t=0$]{
    \begin{minipage}[b]{0.46\textwidth}
    \includegraphics[width=1\textwidth]{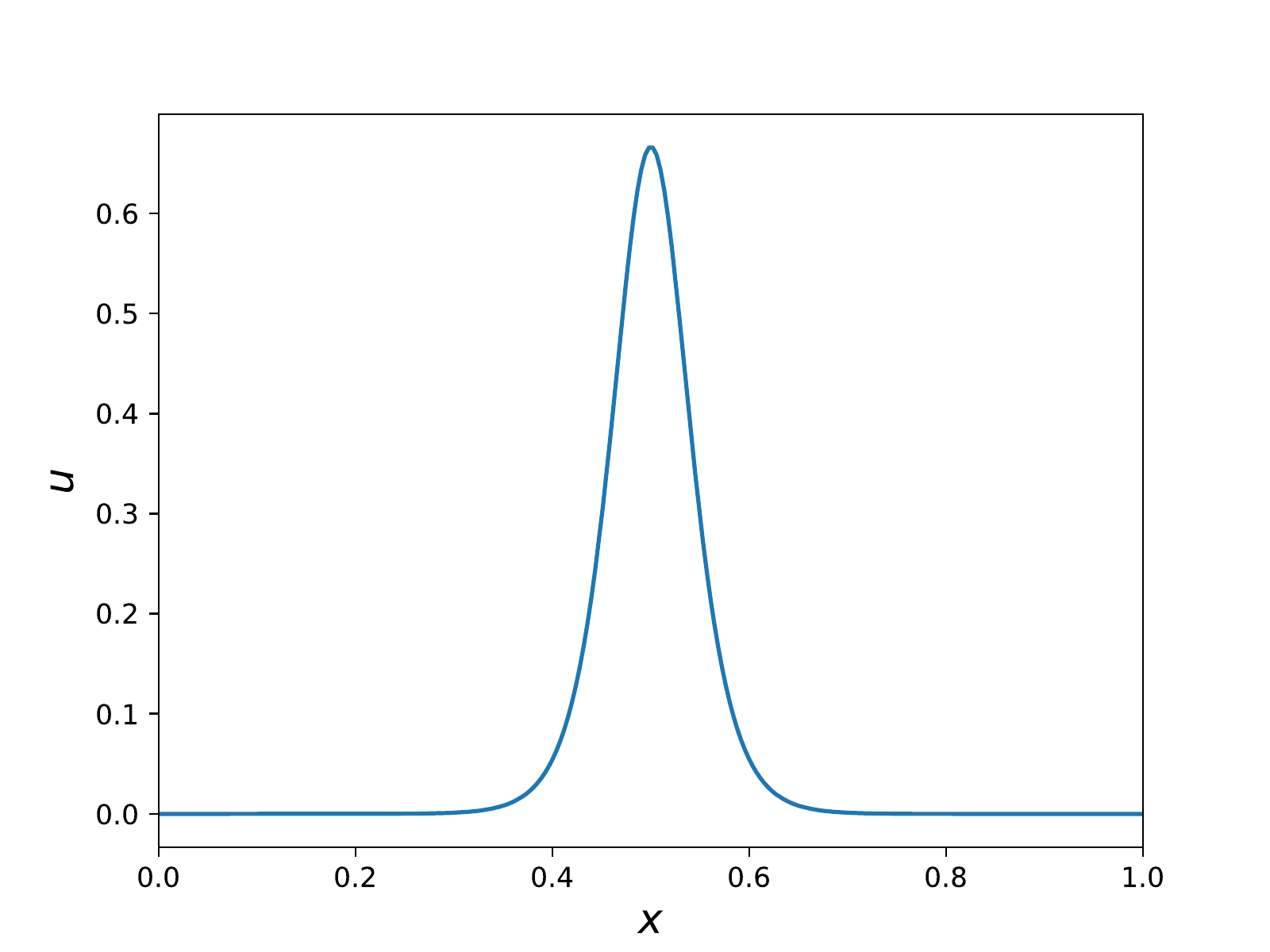}
    \end{minipage}
    }
    \subfigure[active elements at $t=0$]{
    \begin{minipage}[b]{0.46\textwidth}
    \includegraphics[width=1\textwidth]{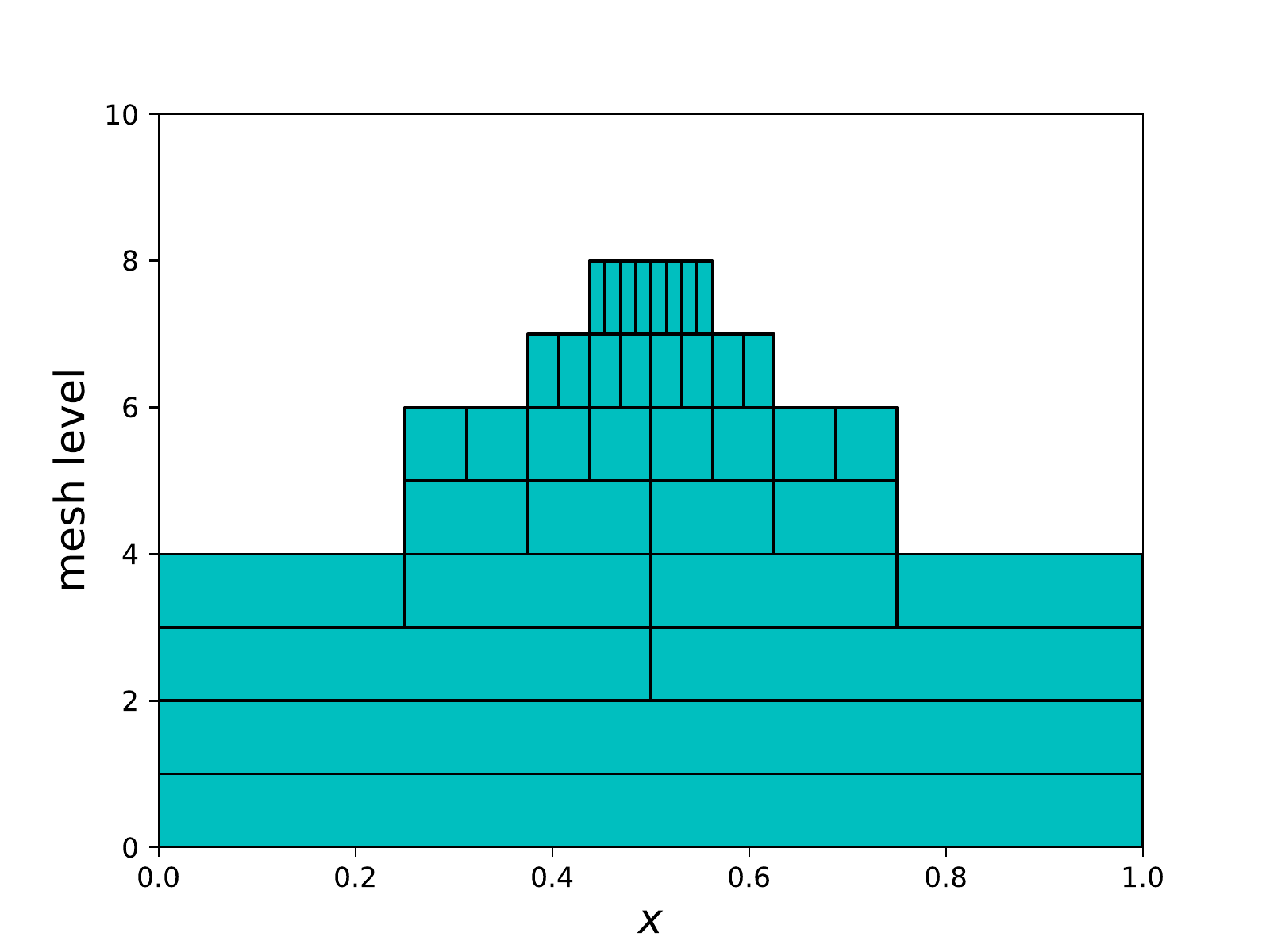}
    \end{minipage}
    }
    \bigskip
    \subfigure[numerical solution at $t=0.5$]{
    \begin{minipage}[b]{0.46\textwidth}
    \includegraphics[width=1\textwidth]{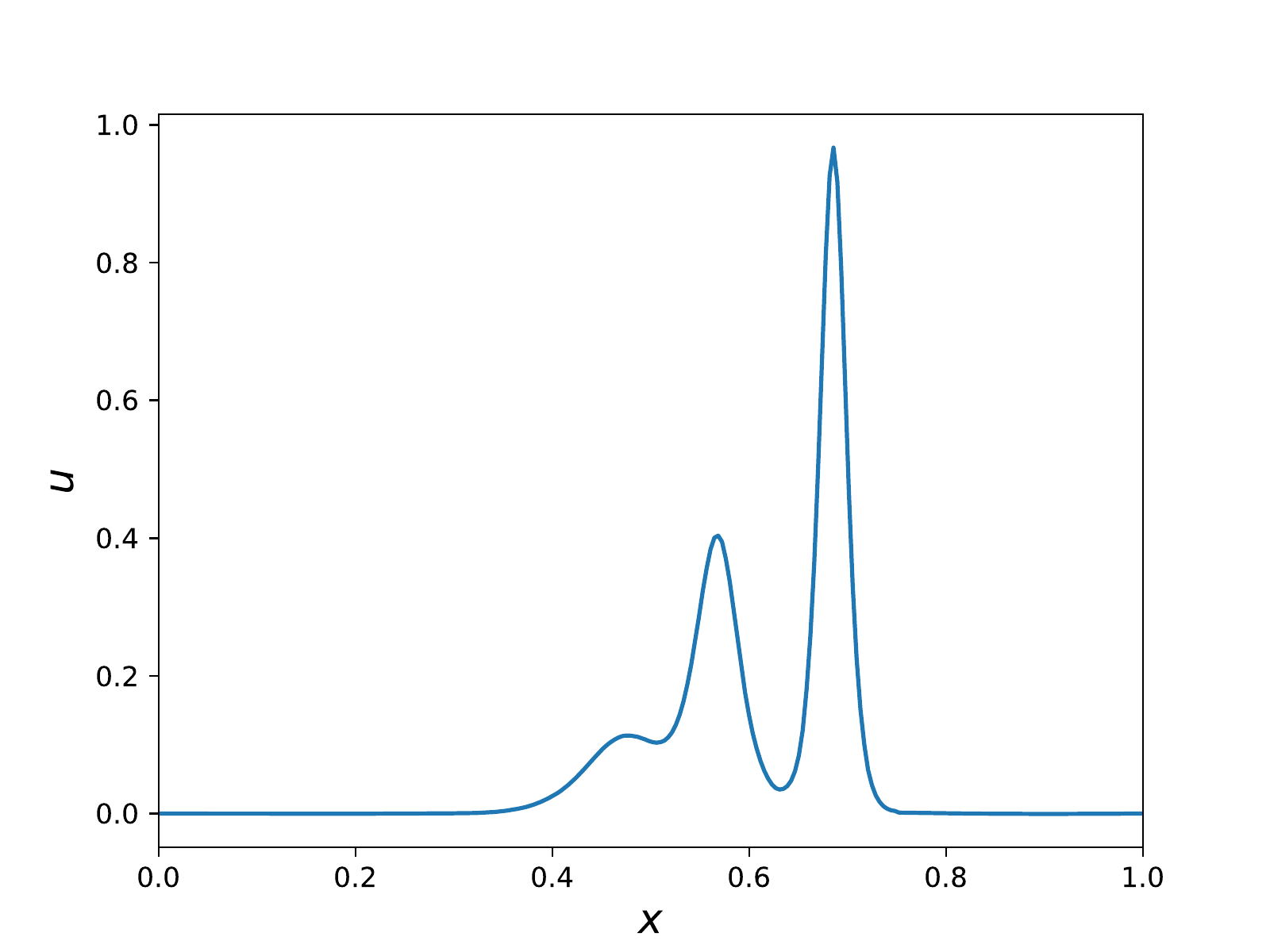}
    \end{minipage}
    }
    \subfigure[active elements at $t=0.5$]{
    \begin{minipage}[b]{0.46\textwidth}
    \includegraphics[width=1\textwidth]{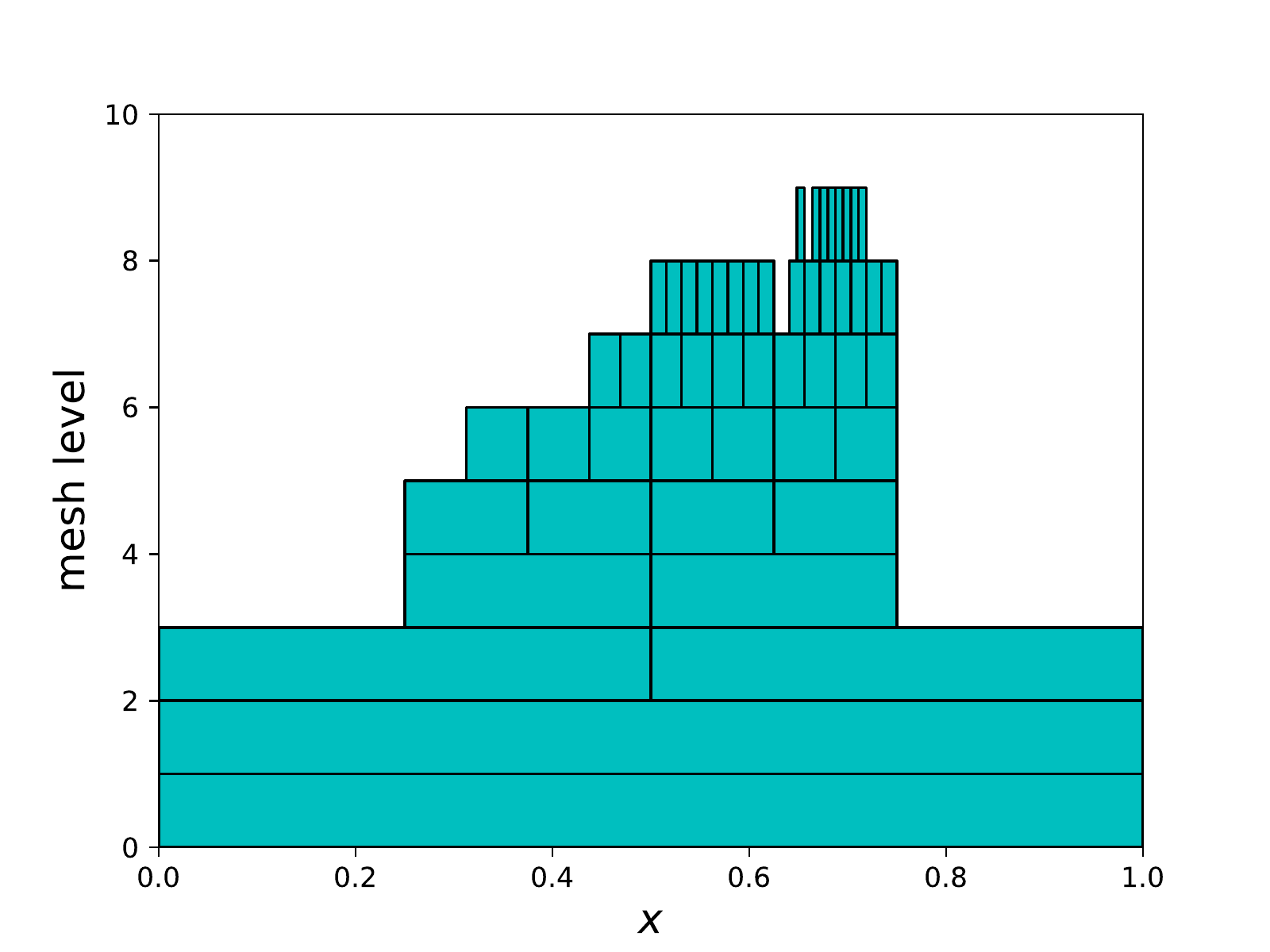}
    \end{minipage}
    }
    \bigskip
    \subfigure[numerical solution at $t=1$]{
    \begin{minipage}[b]{0.46\textwidth}
    \includegraphics[width=1\textwidth]{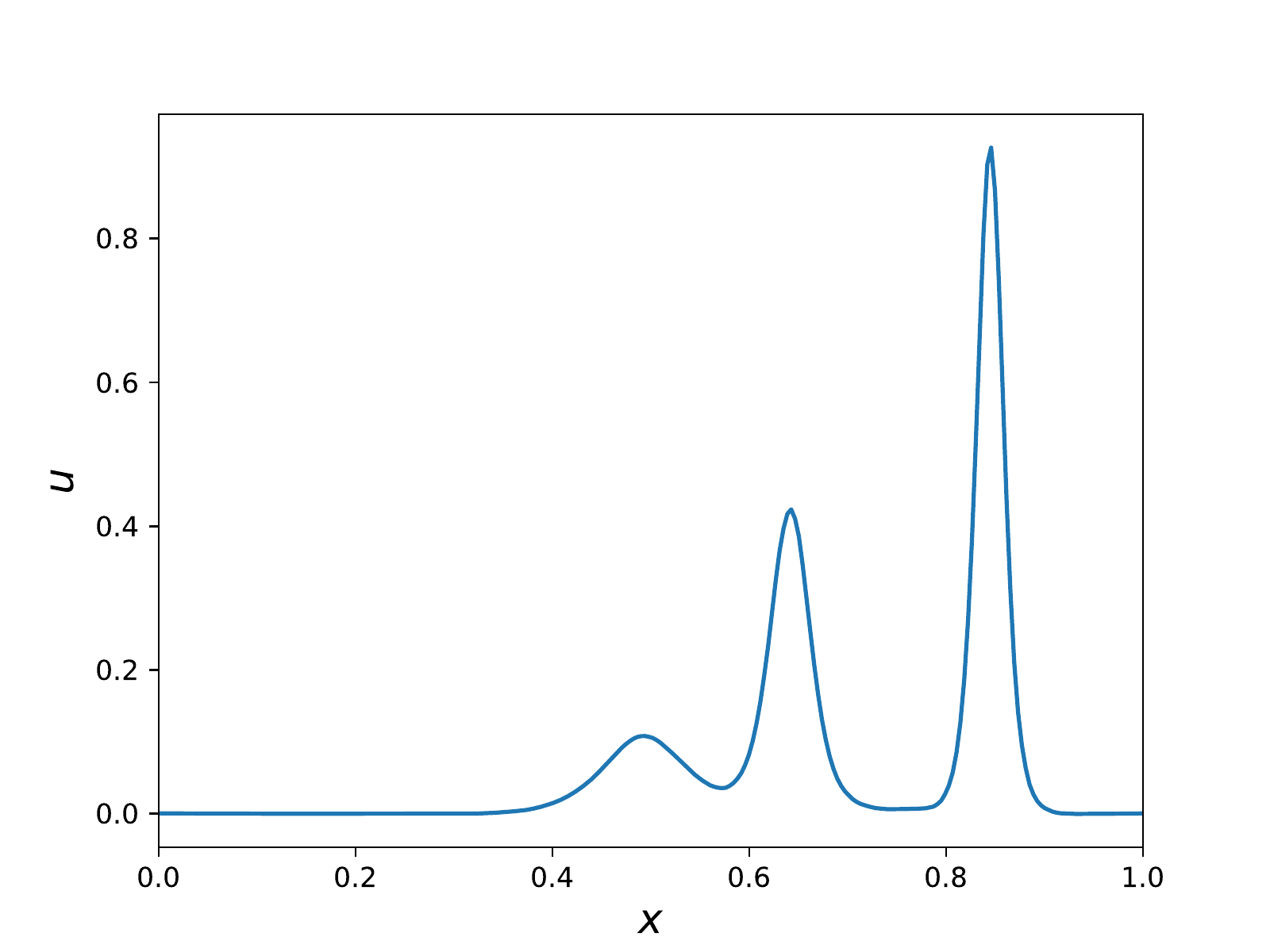}
    \end{minipage}
    }
    \subfigure[active elements at $t=1$]{
    \begin{minipage}[b]{0.46\textwidth}
    \includegraphics[width=1\textwidth]{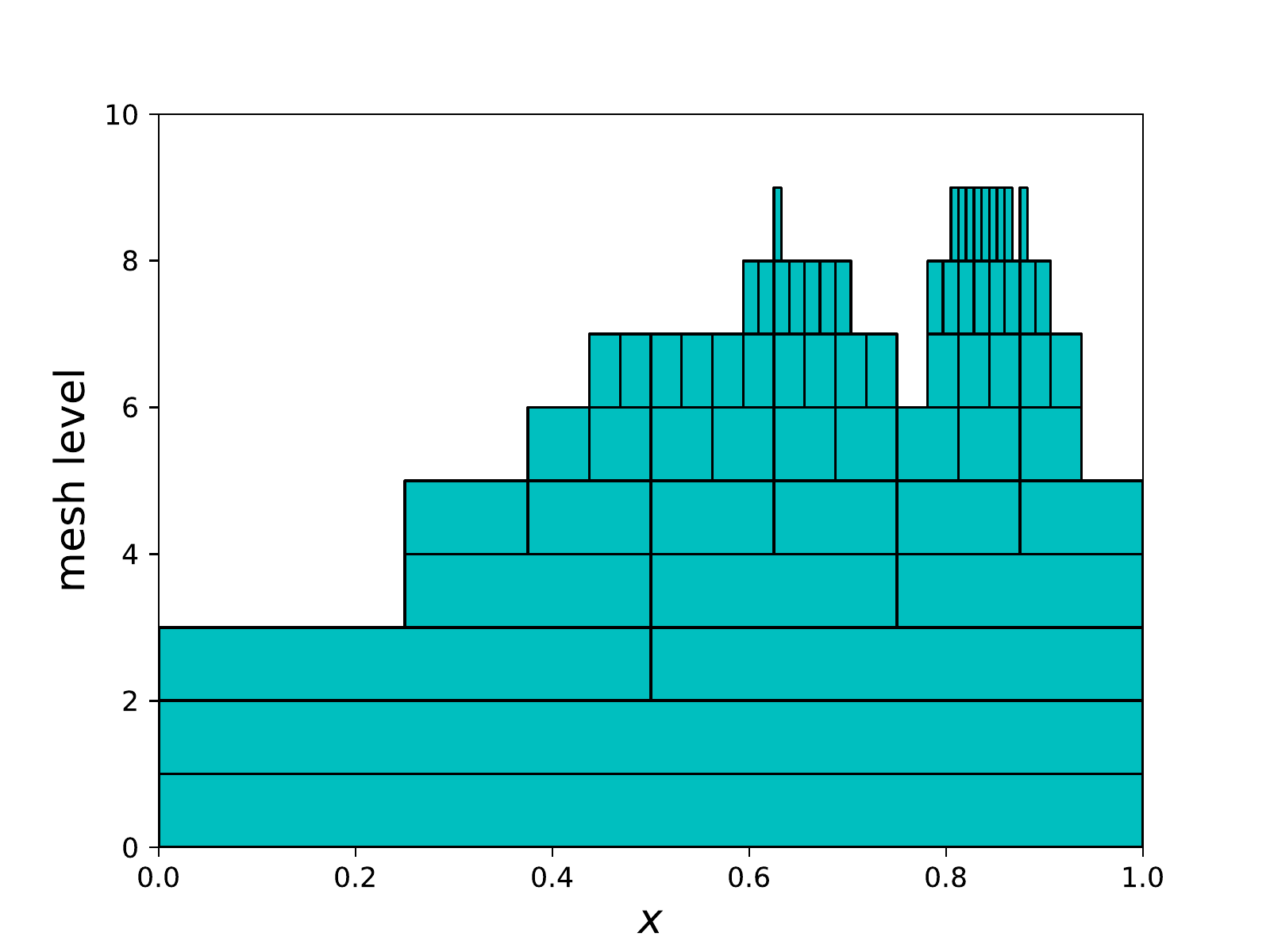}
    \end{minipage}
    }
    \caption{Example \ref{exam:kdv-1d}: nonlinear KdV equation in 1D, triple soliton splitting. $t=0$, 0.5 and 1. $N=8$ and $\epsilon=10^{-4}$. Left: numerical solutions at $t=0$, 0.5 and 1; right: active elements at $t=0$, 0.5 and 1.}
    \label{fig:kdv-1d-triple}
\end{figure}

\end{exam}

\subsection{ZK equation}

\begin{exam}[accuracy test for the simplified ZK equation]\label{exam:accuracy-uxyy}
    We test the convergence order of UWDG scheme on full and sparse grids for the simplified ZK equation:
    \begin{equation}
        u_t + u_{xyy} = 0,
    \end{equation}
    with the periodic boundary conditions. The exact solution is taken to be
    \begin{equation}
        u(x,y,t) = \sin(2\pi(x+y)+8\pi^3t)).
    \end{equation}
\end{exam}
The numerical results with full grid and sparse grid for $k=1,2,3$ are shown in Table \ref{tab:accuracy-uxyy-full} and Table \ref{tab:accuracy-uxyy-sparse}, respectively. We observe clearly $(k+1)$ order of accuracy for full grid, which  verifies our optimal error estimate in Theorem \ref{thm:error-estimate}. Moreover, slightly more than $(k+\frac{1}{2})$ order of accuracy is observed for sparse grid, which is consistent with approximation results for sparse grid, e.g. see \cite{wang2016elliptic} for results on elliptic equations.
\begin{table}[!hbp]
\centering
\caption{Example \ref{exam:accuracy-uxyy}: accuracy test for the simplified ZK equation. Full grid, $k=1, 2, 3$. $t=0.01$.}
\label{tab:accuracy-uxyy-full}
\begin{tabular}{c|c|c|c|c|c|c|c}
  \hline
  & $N$ & $L^1$-error & order & $L^2$-error & order & $L^{\infty}$-error & order \\
  \hline
\multirow{5}{3em}{$k$ = 1}
& 2 & 8.31e-01 & - & 9.75e-01 & - & 2.00e+00 & - \\
& 3 & 3.52e-01 & 1.24 & 3.93e-01 & 1.31 & 6.40e-01 & 1.64 \\
& 4 & 8.50e-02 & 2.05 & 9.43e-02 & 2.06 & 1.47e-01 & 2.12 \\
& 5 & 2.08e-02 & 2.03 & 2.31e-02 & 2.03 & 3.59e-02 & 2.03 \\
& 6 & 5.18e-03 & 2.01 & 5.75e-03 & 2.01 & 8.93e-03 & 2.01 \\
  \hline
\multirow{5}{3em}{$k$ = 2}
& 2 & 3.02e-02 & - & 4.08e-02 & - & 1.19e-01 & - \\
& 3 & 3.55e-03 & 3.09 & 4.75e-03 & 3.10 & 1.60e-02 & 2.90 \\
& 4 & 4.53e-04 & 2.97 & 5.84e-04 & 3.02 & 1.97e-03 & 3.02 \\
& 5 & 5.75e-05 & 2.98 & 7.27e-05 & 3.01 & 2.42e-04 & 3.02 \\
& 6 & 7.26e-06 & 2.99 & 9.07e-06 & 3.00 & 2.98e-05 & 3.02 \\
\hline
\multirow{5}{3em}{$k$ = 3}
& 2 & 2.22e-03 & - & 2.93e-03 & - & 7.59e-03 & - \\
& 3 & 1.64e-04 & 3.76 & 2.17e-04 & 3.75 & 5.77e-04 & 3.72 \\
& 4 & 9.57e-06 & 4.10 & 1.31e-05 & 4.06 & 3.85e-05 & 3.91 \\
& 5 & 6.22e-07 & 3.94 & 8.52e-07 & 3.94 & 2.55e-06 & 3.92 \\
& 6 & 3.90e-08 & 4.00 & 5.35e-08 & 3.99 & 1.60e-07 & 3.99 \\
\hline
\end{tabular}
\end{table}

\begin{table}[!hbp]
\centering
\caption{Example \ref{exam:accuracy-uxyy}: accuracy test for the simplified ZK equation. Sparse grid, $k=1, 2, 3$. $t=0.01$.}
\label{tab:accuracy-uxyy-sparse}
\begin{tabular}{c|c|c|c|c|c|c|c}
  \hline
  & $N$ & $L^1$-error & order & $L^2$-error & order & $L^{\infty}$-error & order \\
  \hline
\multirow{5}{3em}{$k$ = 1}
   & 2 & 7.60e-01 & - & 8.66e-01 & - & 1.80e+00 & - \\
   & 3 & 6.67e-01 & 0.19 & 7.56e-01 & 0.20 & 1.59e+00 & 0.18 \\
   & 4 & 4.10e-01 & 0.70 & 4.87e-01 & 0.63 & 1.06e+00 & 0.58 \\
   & 5 & 1.67e-01 & 1.29 & 1.92e-01 & 1.35 & 4.27e-01 & 1.32 \\
   & 6 & 5.31e-02 & 1.66 & 6.24e-02 & 1.62 & 1.81e-01 & 1.24 \\
  \hline
\multirow{5}{3em}{$k$ = 2}
   & 2 & 2.04e-01 & - & 2.58e-01 & - & 7.13e-01 & - \\
   & 3 & 3.73e-02 & 2.45 & 4.73e-02 & 2.45 & 1.67e-01 & 2.10 \\
   & 4 & 5.63e-03 & 2.73 & 7.53e-03 & 2.65 & 4.38e-02 & 1.93 \\
   & 5 & 9.11e-04 & 2.63 & 1.20e-03 & 2.65 & 7.93e-03 & 2.46 \\
   & 6 & 1.30e-04 & 2.81 & 1.73e-04 & 2.79 & 1.18e-03 & 2.75 \\
\hline
\multirow{5}{3em}{$k$ = 3}
   & 2 & 1.10e-02 & - & 1.36e-02 & - & 5.63e-02 & - \\
   & 3 & 1.08e-03 & 3.34 & 1.43e-03 & 3.25 & 9.37e-03 & 2.59 \\
   & 4 & 7.93e-05 & 3.77 & 1.07e-04 & 3.74 & 7.17e-04 & 3.71 \\
   & 5 & 6.02e-06 & 3.72 & 7.89e-06 & 3.76 & 5.76e-05 & 3.64 \\
   & 6 & 4.15e-07 & 3.86 & 5.55e-07 & 3.83 & 4.86e-06 & 3.57 \\
\hline
\end{tabular}
\end{table}

\begin{exam}[accuracy test for the ZK equation]\label{exam:zk-accuracy}
    We consider the ZK equation
    \begin{equation}\label{eq:zk-source}
        u_t + \brac{\frac{u^2}{2}}_x + u_{xxx} + u_{xyy} = s(x,y,t)
    \end{equation}
    with periodic boundary conditions.
    We add a particular source term
    \begin{equation}
    	s(x,y,t) = 2 \pi \cos(2 \pi (x + y + t)) (1 - 8 \pi^2 + \sin(2 \pi (x + y + t)))
    \end{equation}
    such that the exact solution is
    \begin{equation}
        u(x,y,t) = \sin(2\pi(x+y+t)).
    \end{equation}
\end{exam}

In Table \ref{tab:zk-accuracy-full} and Table \ref{tab:zk-accuracy-sparse}, we present the convergence order for full grid and sparse grid in the case of $k=2$ and $k=3$, from which $(k+1)$-th order is clearly observed for the full grid. The convergence order for the sparse grid is between $k$ and $(k+1)$. The accuracy with the adaptive method is shown in Table \ref{tab:zk-accuracy-adaptive}. The $R_{\textrm{DOF}}$ is larger than the full grid case $(k+1)/d$. Moreover, it is observed that to reach the same error of magnitude, it takes much fewer DoFs of $k=3$ than $k=2$.
\begin{table}[!hbp]
\centering
\caption{Example \ref{exam:zk-accuracy}: accuracy test for the ZK equation \eqref{eq:zk-source}. Full grid, $k=2, 3$, $t=0.01$.}
\label{tab:zk-accuracy-full}
\begin{tabular}{c|c|c|c|c|c|c|c}
  \hline
  & $N$ & $L^1$-error & order & $L^2$-error & order & $L^{\infty}$-error & order \\
  \hline
\multirow{5}{3em}{$k$ = 2}
& 2 & 1.85e-01 & - & 2.24e-01 & - & 5.04e-01 & - \\
& 3 & 3.85e-02 & 2.26 & 4.40e-02 & 2.35 & 8.80e-02 & 2.52 \\
& 4 & 5.56e-03 & 2.79 & 6.25e-03 & 2.82 & 1.12e-02 & 2.97 \\
& 5 & 7.17e-04 & 2.95 & 8.04e-04 & 2.96 & 1.40e-03 & 3.00 \\
& 6 & 9.01e-05 & 2.99 & 1.01e-04 & 2.99 & 1.74e-04 & 3.01 \\
\hline
\multirow{5}{3em}{$k$ = 3}	
& 2 & 1.92e-02 & - & 2.29e-02 & - & 4.44e-02 & - \\
& 3 & 2.41e-03 & 2.99 & 2.73e-03 & 3.06 & 4.95e-03 & 3.16 \\
& 4 & 1.36e-04 & 4.15 & 1.54e-04 & 4.15 & 2.79e-04 & 4.15 \\
& 5 & 9.44e-06 & 3.85 & 1.06e-05 & 3.85 & 1.89e-05 & 3.89 \\
& 6 & 5.96e-07 & 3.99 & 6.71e-07 & 3.99 & 1.18e-06 & 4.00 \\
\hline
\end{tabular}
\end{table}

\begin{table}[!hbp]
\centering
\caption{Example \ref{exam:zk-accuracy}: accuracy test for the ZK equation \eqref{eq:zk-source}. Sparse grid, $k=2, 3$, $t=0.01$.}
\label{tab:zk-accuracy-sparse}
\begin{tabular}{c|c|c|c|c|c|c|c}
  \hline
  & $N$ & $L^1$-error & order & $L^2$-error & order & $L^{\infty}$-error & order \\
  \hline
\multirow{5}{3em}{$k$ = 2}
	& 2 & 2.80e-01 & - & 3.42e-01 & - & 1.14e+00 & - \\
	& 3 & 6.86e-02 & 2.03 & 8.50e-02 & 2.01 & 3.06e-01 & 1.89 \\
	& 4 & 1.23e-02 & 2.48 & 1.49e-02 & 2.51 & 5.13e-02 & 2.58 \\
	& 5 & 2.59e-03 & 2.25 & 3.28e-03 & 2.18 & 1.72e-02 & 1.57 \\
	& 6 & 2.92e-04 & 3.15 & 3.63e-04 & 3.17 & 2.14e-03 & 3.01 \\
\hline
\multirow{5}{3em}{$k$ = 3}
	& 2 & 3.28e-02 & - & 3.96e-02 & - & 1.23e-01 & - \\
	& 3 & 2.78e-03 & 3.56 & 3.32e-03 & 3.58 & 1.16e-02 & 3.40 \\
	& 4 & 1.84e-04 & 3.91 & 2.27e-04 & 3.87 & 9.22e-04 & 3.66 \\
	& 5 & 1.45e-05 & 3.67 & 1.82e-05 & 3.64 & 9.63e-05 & 3.26 \\
	& 6 & 9.50e-07 & 3.93 & 1.18e-06 & 3.95 & 5.55e-06 & 4.12 \\
\hline
\end{tabular}
\end{table}

\begin{table}[!hbp]
    \centering
    \caption{Example \ref{exam:zk-accuracy}, accuracy test for the ZK equation \eqref{eq:zk-source}. Adaptive scheme, $k=2$ and $k=3$. $t=0.01$.}
    \label{tab:zk-accuracy-adaptive}
    \begin{tabular}{c|c|c|c|c|c}
      \hline
      & $\epsilon$ & DoF & $L^2$-error & $R_{\textrm{DoF}}$ & $R_{\epsilon}$ \\
      \hline
    \multirow{4}{3em}{$k=2$}
	& 1e-01 & 108 & 1.97e-01 & - & - \\
	& 1e-02 & 288 & 3.26e-02 & 1.83 & 0.78 \\
	& 1e-03 & 720 & 4.54e-03 & 2.15 & 0.86 \\
	& 1e-04 & 1656 & 6.01e-04 & 2.43 & 0.88 \\
    \hline
    \multirow{4}{3em}{$k=3$}
	& 1e-01 & 96 & 1.50e-01 & - & - \\
	& 1e-02 & 192 & 2.30e-02 & 2.71 & 0.82 \\
	& 1e-03 & 320 & 2.75e-03 & 4.16 & 0.92 \\
	& 1e-04 & 768 & 3.71e-04 & 2.29 & 0.87 \\
    \hline
    \end{tabular}
\end{table}

\begin{exam}[cylindrically symmetric solitons for the ZK equation]\label{exam:zk-soliton}
    We investigate a cylindrically symmetric solitary solution and its evolutions as well as interactions for the ZK equation \cite{iwasaki1990cylindrical,feng1999conservative}
    \begin{equation}\label{eq:zk-example-soliton}
        u_t + (3u^2)_x + \sigma(u_{xxx} + u_{xyy}) = 0.
    \end{equation}
    This type of solitary solution, also called the bell-shaped pulse, has the initial value
    \begin{equation}
        u(x,y,t) = \frac{c}{3}\sum_{n=1}^{10}a_n\brac{\cos(2n\arccot(\frac{\sqrt{c}}{2}r))-1}
    \end{equation}
    where $c$ is the velocity of the soliton wave solution and $r=\sqrt{(x-x_0)^2+(y-y_0)^2}$. The coefficients are \cite{iwasaki1990cylindrical}
    \begin{align}\label{eq:zk-soliton-coefficient}
		& a_1 = -1.25529873, \quad a_2 = 0.21722635, \quad a_3 = 0.06452543, \nn \\
		& a_4 = 0.00540862, \quad a_5 = -0.00332515, \quad a_6 = -0.00281281,  \\
		& a_7 = -0.00138352, \quad a_8 = -0.00070289, \quad a_9 = -0.00020451, \nn \\
		& a_{10} = -0.00003053. \nn
    \end{align}
    In this test, we take $x_0=y_0=0.5$ and $\sigma=1/1024$ in \eqref{eq:zk-example-soliton}. The stable propagation of a single pulse is presented in Figure \ref{fig:zk-single}. The active elements automatically move with the soliton and the soliton shape is well preserved in the time evolution.	
\end{exam}

\begin{figure}
    \centering
    \subfigure[numerical solutions at $t=0$]{
    \begin{minipage}[b]{0.46\textwidth}
    \includegraphics[width=1\textwidth]{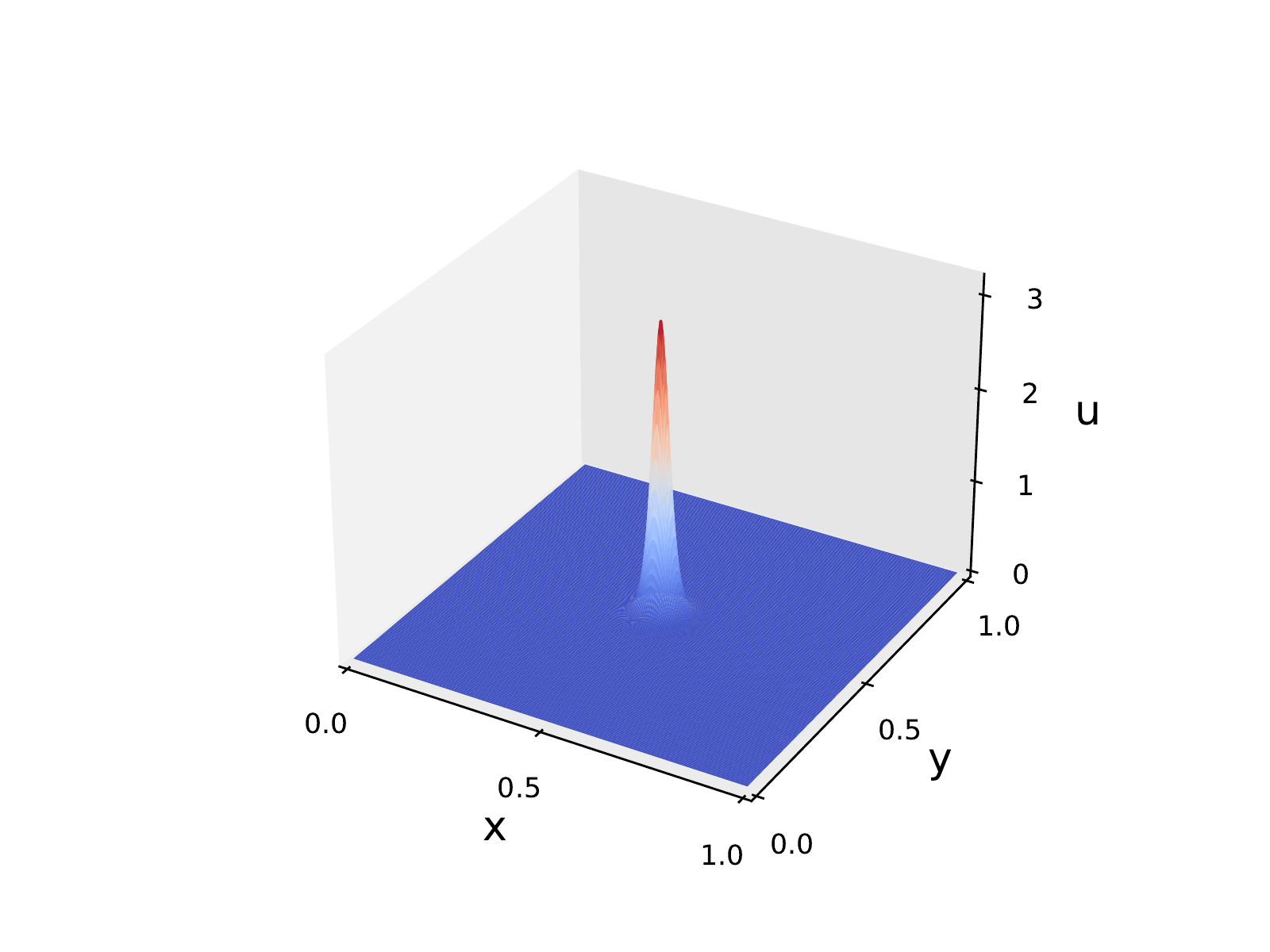}
    \end{minipage}
    }
    \subfigure[active elements at $t=0$]{
    \begin{minipage}[b]{0.46\textwidth}
    \includegraphics[width=1\textwidth]{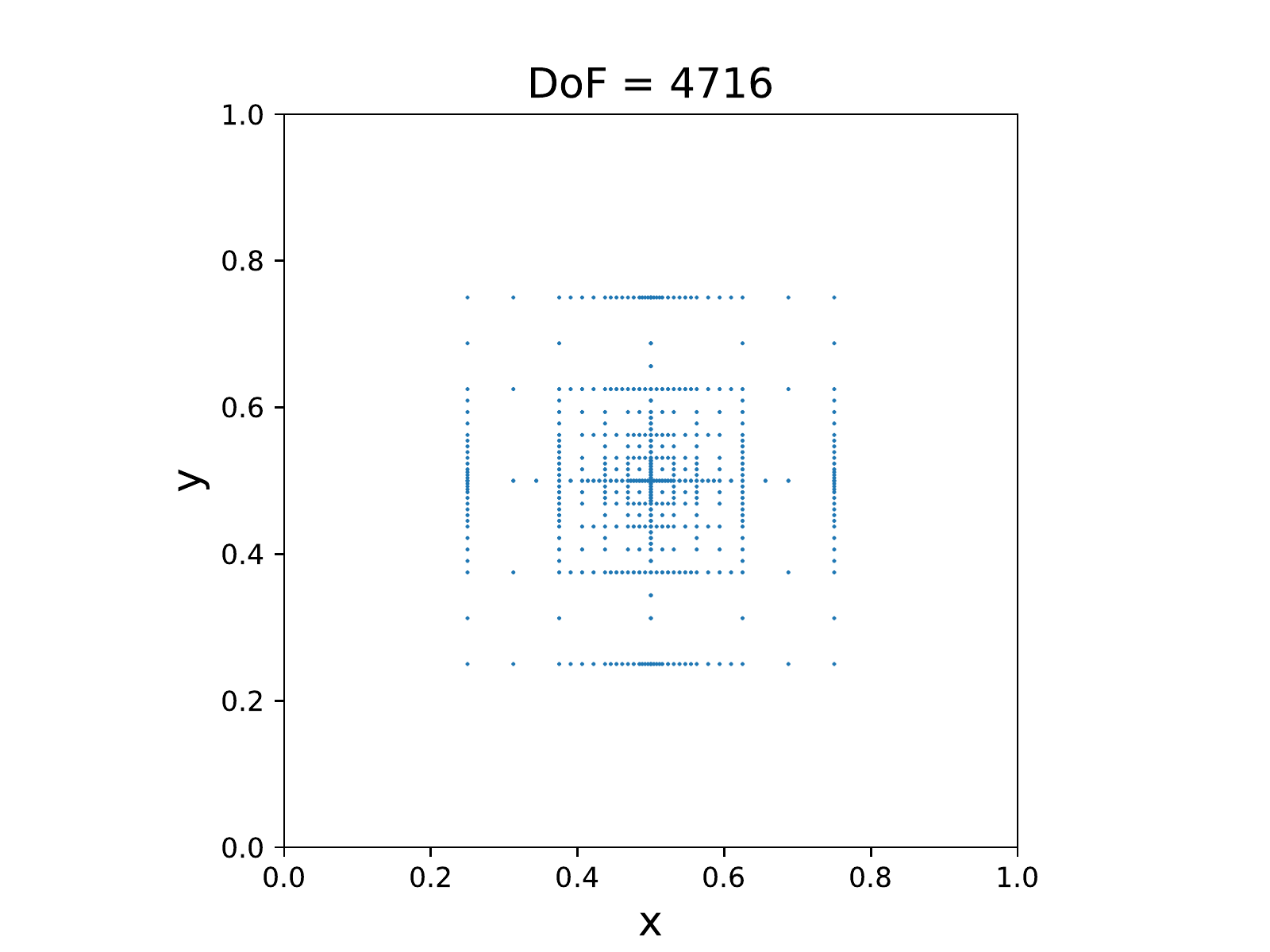}
    \end{minipage}
    }
    \bigskip
    \subfigure[numerical solutions at $t=0.1$]{
    \begin{minipage}[b]{0.46\textwidth}
    \includegraphics[width=1\textwidth]{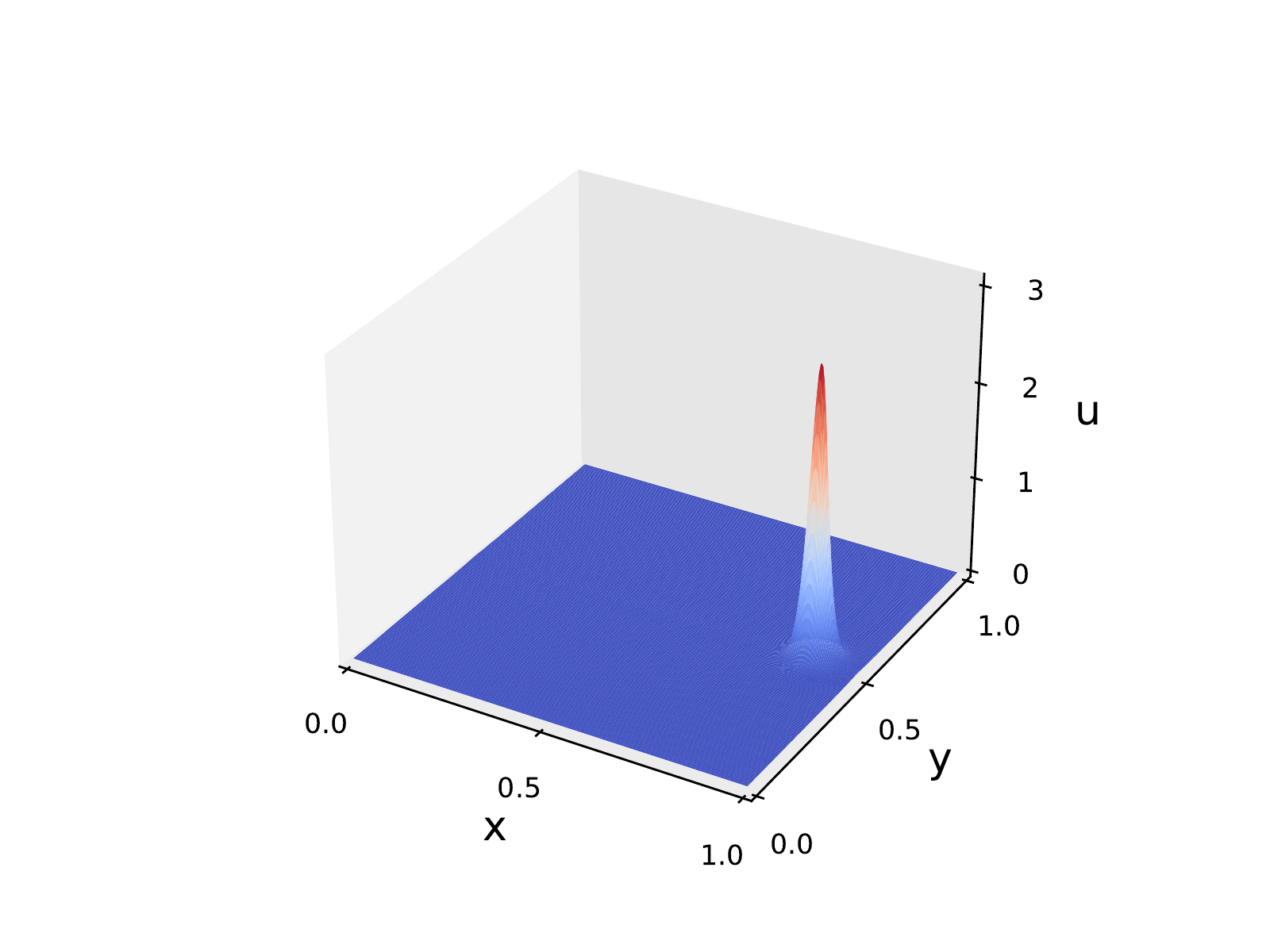}
    \end{minipage}
    }
    \subfigure[active elements at $t=0.1$]{
    \begin{minipage}[b]{0.46\textwidth}
    \includegraphics[width=1\textwidth]{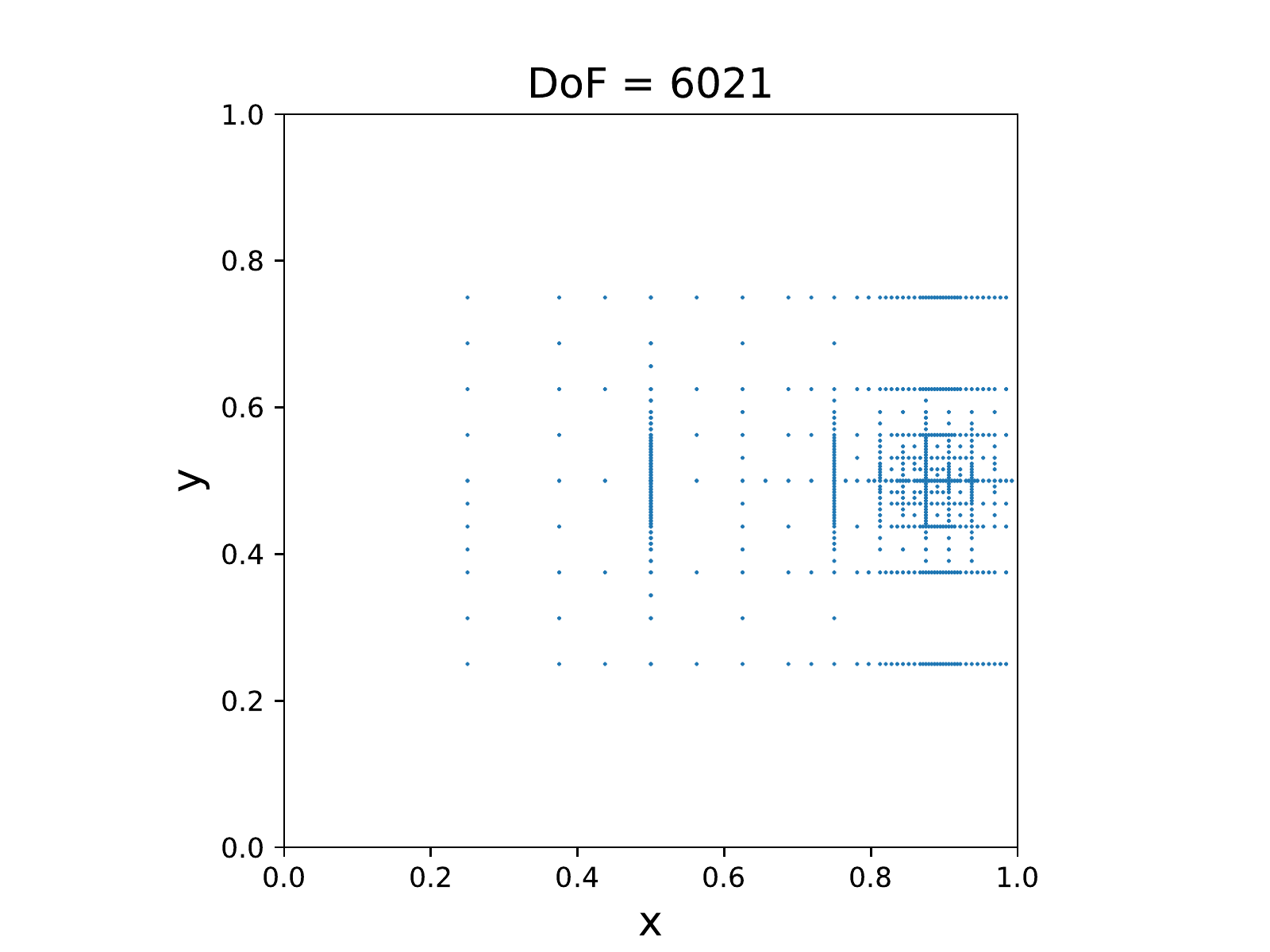}
    \end{minipage}
    }
    \bigskip
    \subfigure[numerical solutions at $t=0.2$]{
    \begin{minipage}[b]{0.46\textwidth}
    \includegraphics[width=1\textwidth]{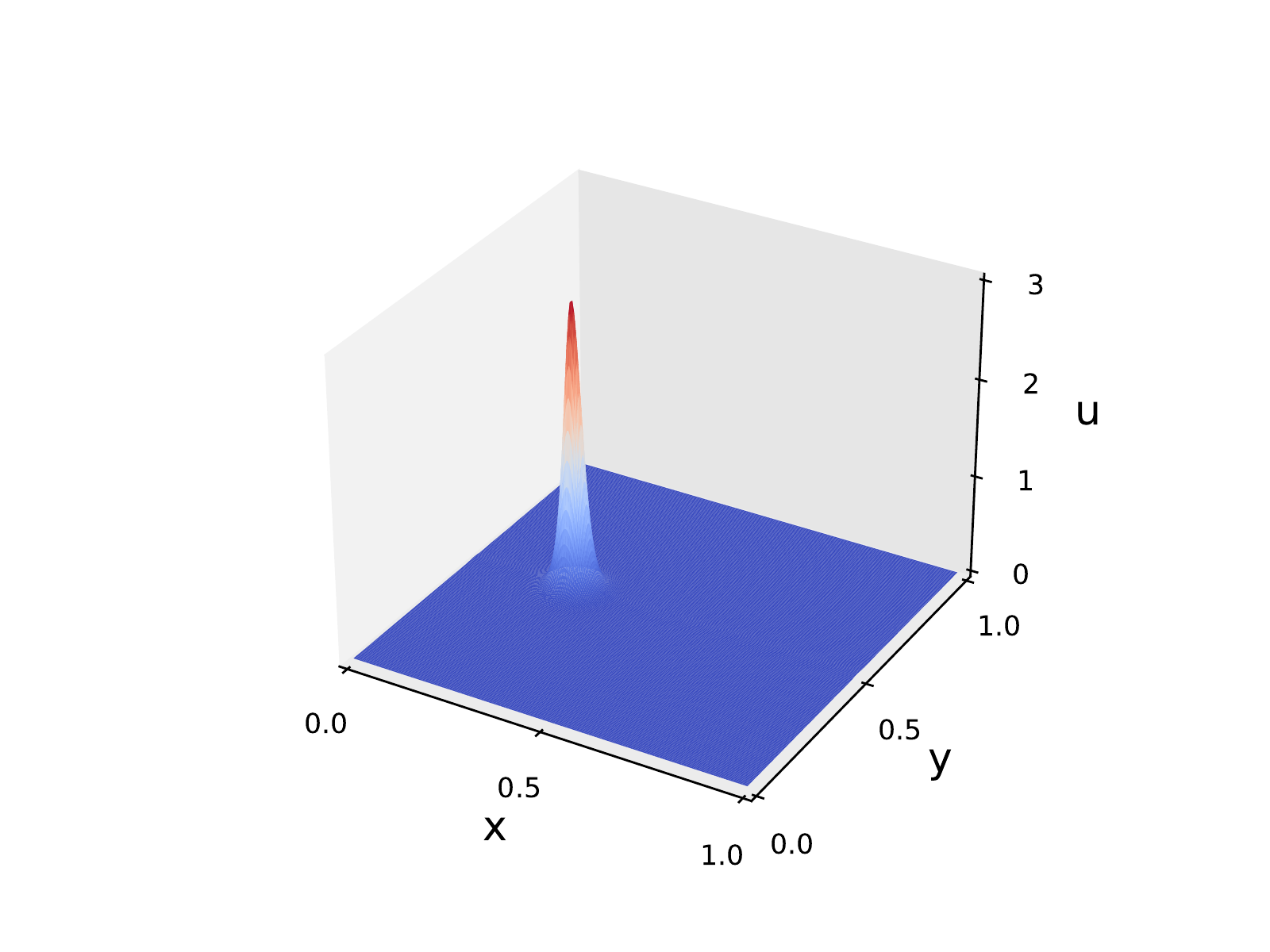}
    \end{minipage}
    }
    \subfigure[active elements at $t=0.2$]{
    \begin{minipage}[b]{0.46\textwidth}
    \includegraphics[width=1\textwidth]{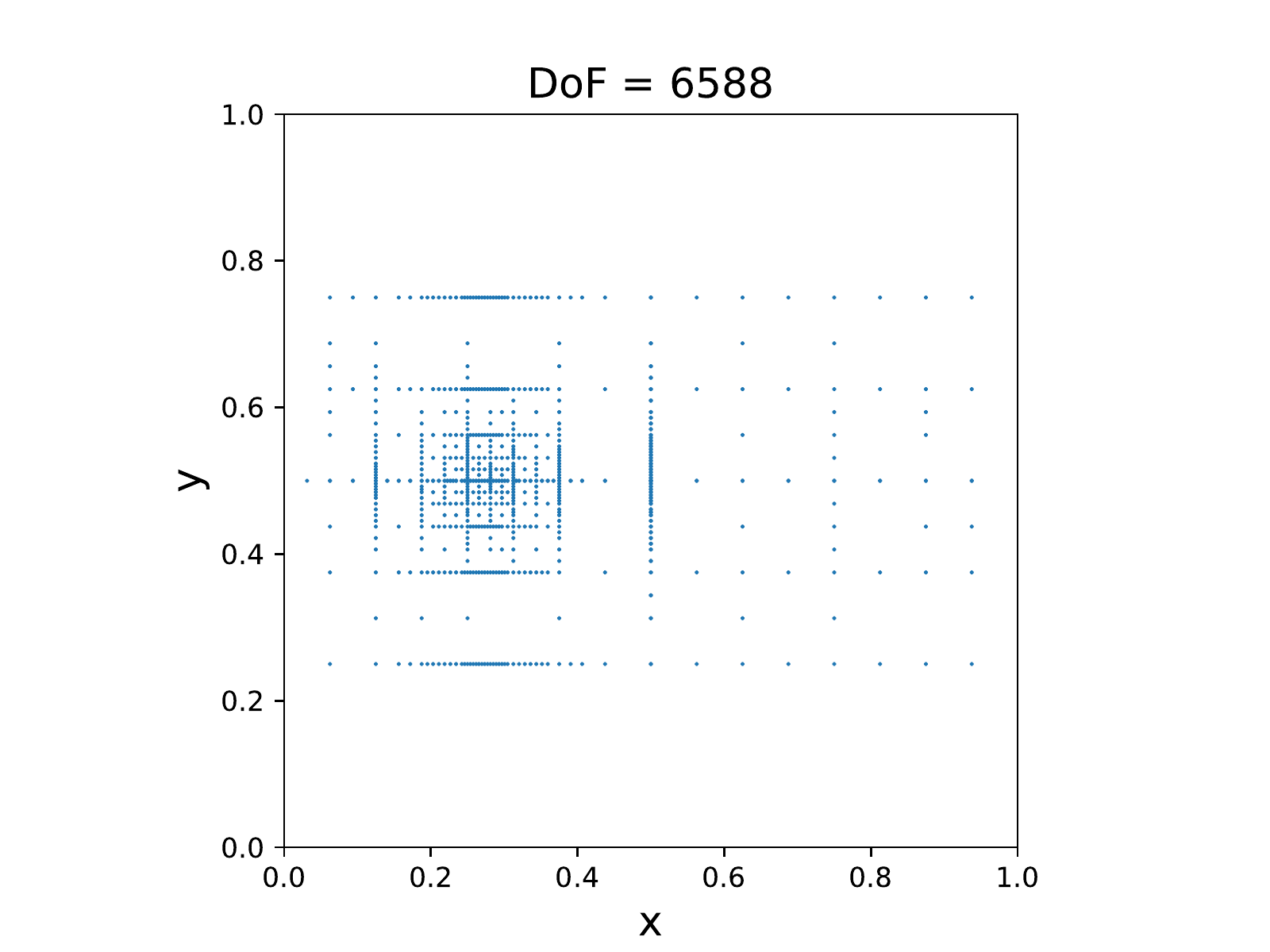}
    \end{minipage}
    }
    \caption{Example \ref{exam:zk-soliton}: ZK equation, single soliton. $t=0$, $0.1$ and $0.2$. $N=8$ and $\epsilon=10^{-4}$. Left: numerical solutions; right: active elements.}
    \label{fig:zk-single}
\end{figure}

\begin{exam}[soliton collisions for the ZK equation]\label{exam:zk-double-soliton}
    Next, we proceed to show the collision of two pulses with the initial condition:
    \begin{equation}
        u(x,y,t) = \sum_{j=1}^2\frac{c_j}{3}\sum_{n=1}^{10}a_n\brac{\cos(2n\arccot(\frac{\sqrt{c_j}}{2}r_j))-1}
    \end{equation}
    with $r_j=\sqrt{(x-x_j)^2+(y-y_j)^2}$, $j=1,2$ and the coefficients $a_n$ for $n=1,\dots,10$ are the same as those given in \eqref{eq:zk-soliton-coefficient}. Here, we simulate two cases with different parameters. The first case is the direct collision of two dissimilar pulses solution. The parameters are $c_1 = 4$ and $c_2 = 1$, $(x_1,y_1)=(1/2,1/2)$, $(x_2,y_2)=(5/8,1/2)$ and $\epsilon=1/4096$. The numerical results are shown in Figure \ref{fig:zk-double-case3}. It is observed that two pulses merge with each other and form a profile with only one-peak. Then two pulses with different amplitudes reappear by emitting ripples \cite{iwasaki1990cylindrical}. The evolution of the numerical solutions is similar to Figure 17 in \cite{xu2005local}. The second case is the deviated collision of two dissimilar pulses solution. The parameters are $c_1 = 4$ and $c_2 = 1$, $(x_1,y_1)=(1/4,7/16)$, $(x_2,y_2)=(1/2,1/2)$, $\epsilon=1/1024$. The numerical results are presented in Figure \ref{fig:zk-double-case4}. The performance is similar to Figure 18 in \cite{xu2005local}.
\end{exam}

\begin{figure}
    \centering
    \subfigure[numerical solutions at $t=0$]{
    \begin{minipage}[b]{0.46\textwidth}
    \includegraphics[width=1\textwidth]{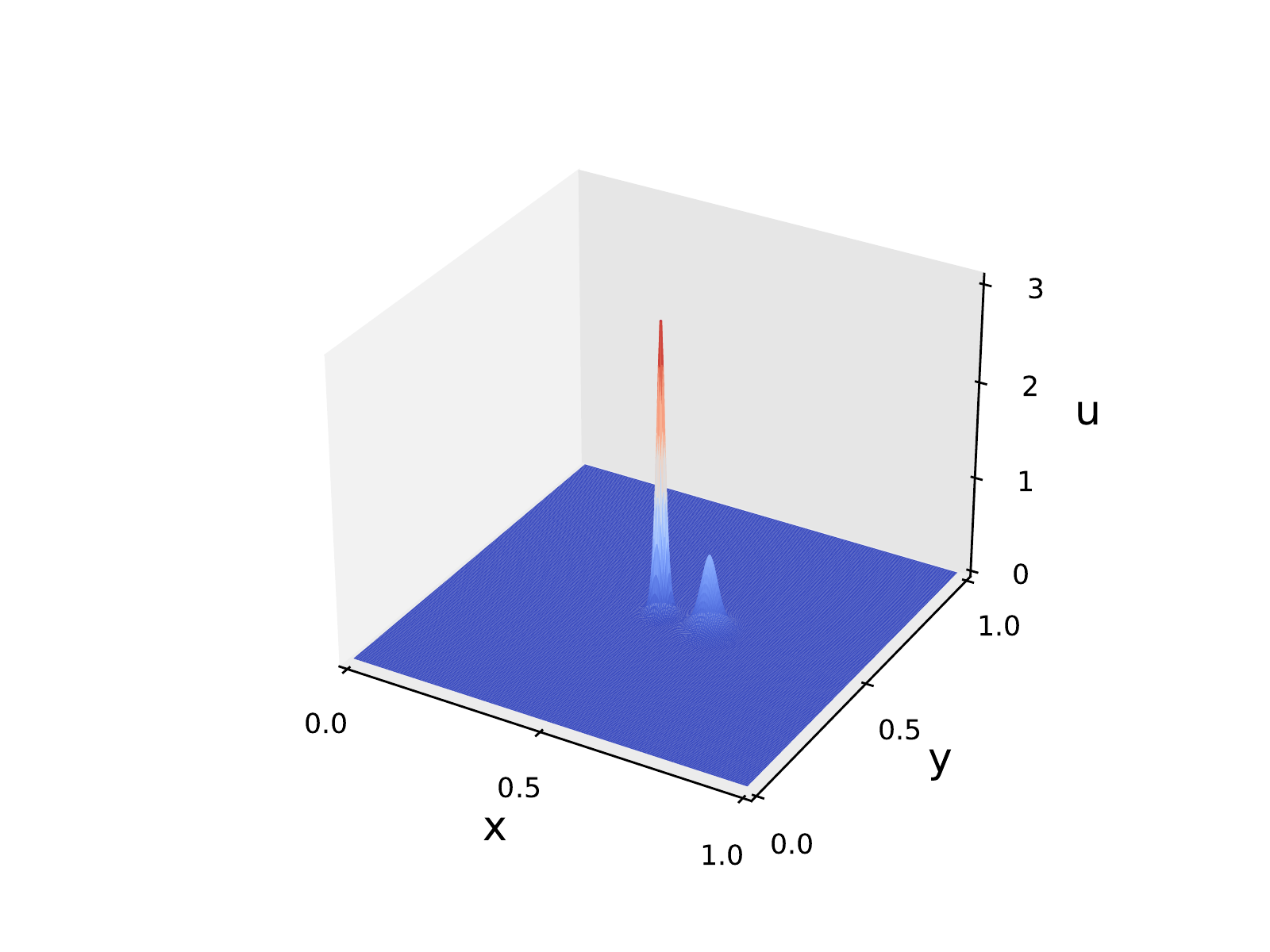}
    \end{minipage}
    }
    \subfigure[active elements at $t=0$]{
    \begin{minipage}[b]{0.46\textwidth}
    \includegraphics[width=1\textwidth]{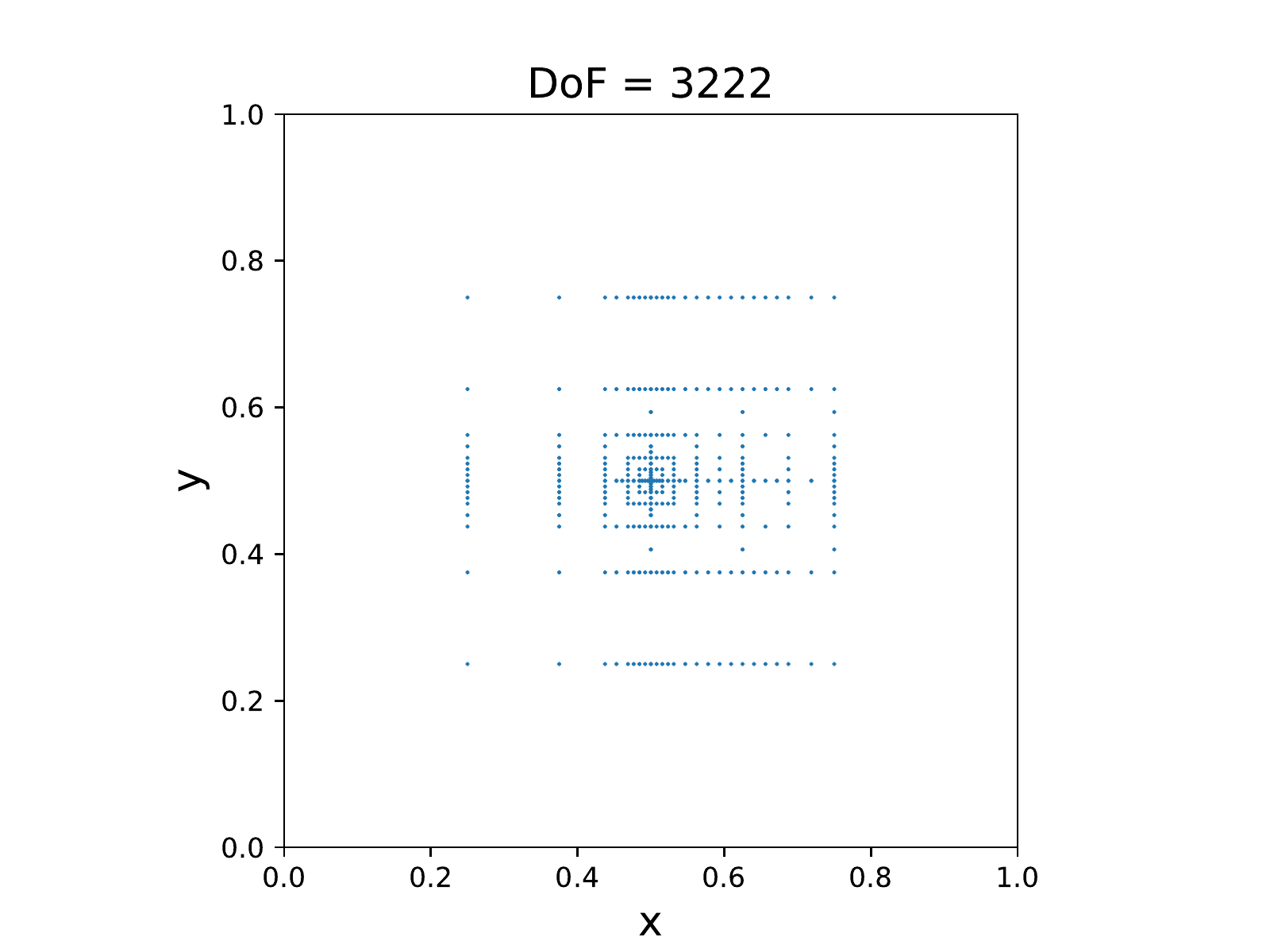}
    \end{minipage}
    }
    \bigskip
    \subfigure[numerical solutions at $t=0.04$]{
    \begin{minipage}[b]{0.46\textwidth}
    \includegraphics[width=1\textwidth]{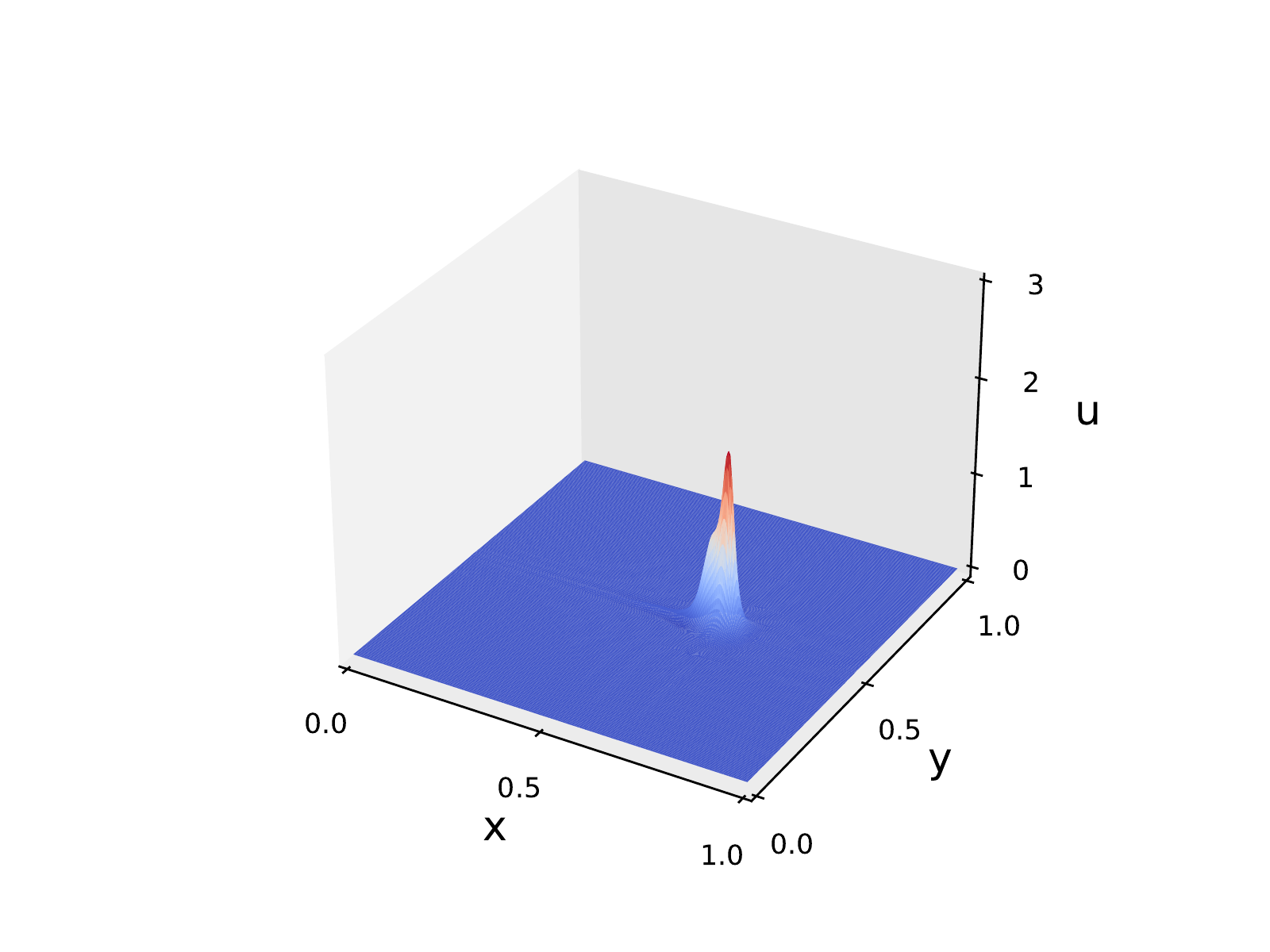}
    \end{minipage}
    }
    \subfigure[active elements at $t=0.04$]{
    \begin{minipage}[b]{0.46\textwidth}
    \includegraphics[width=1\textwidth]{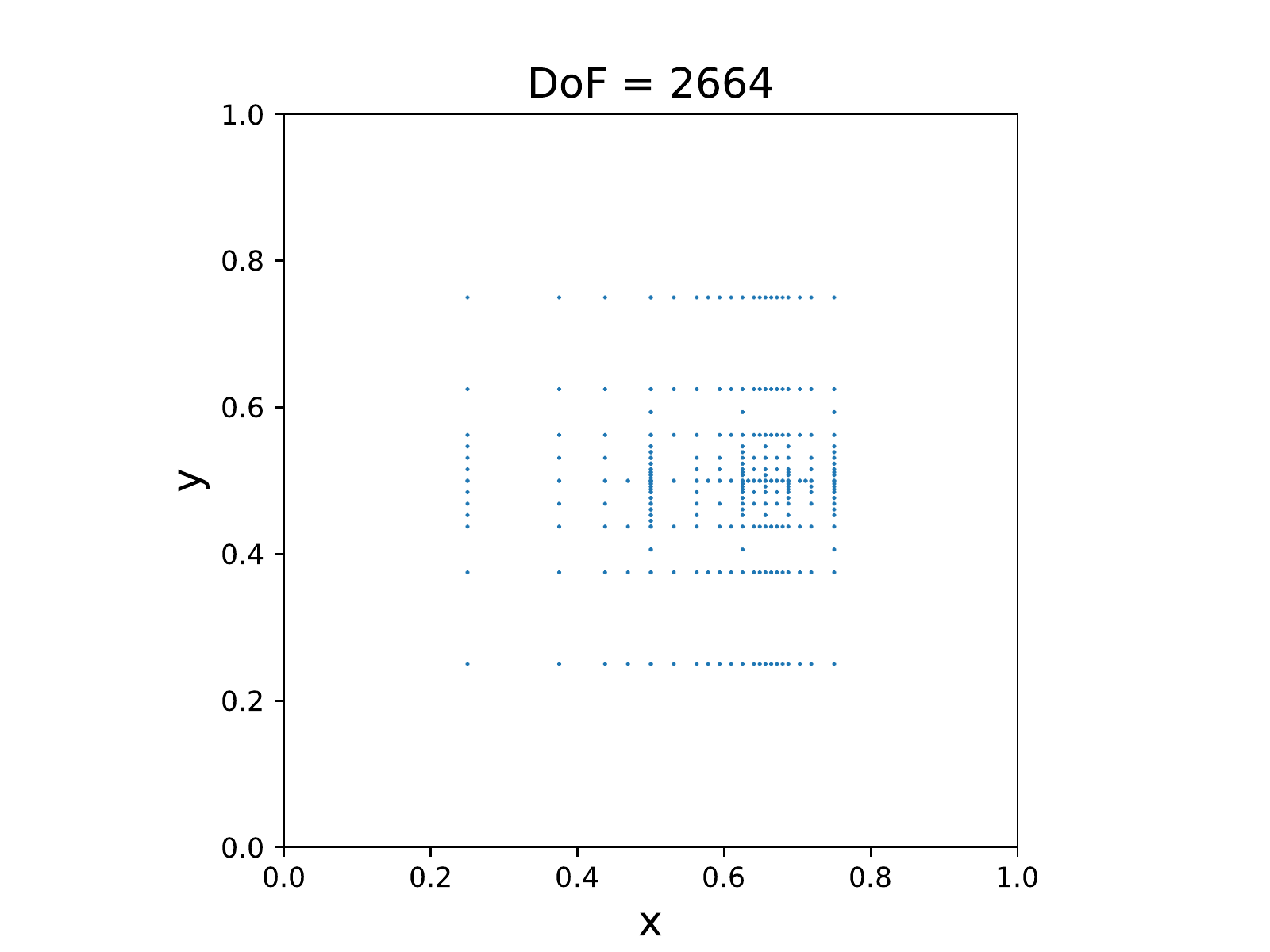}
    \end{minipage}
    }
    \bigskip
    \subfigure[numerical solutions at $t=0.1$]{
    \begin{minipage}[b]{0.46\textwidth}
    \includegraphics[width=1\textwidth]{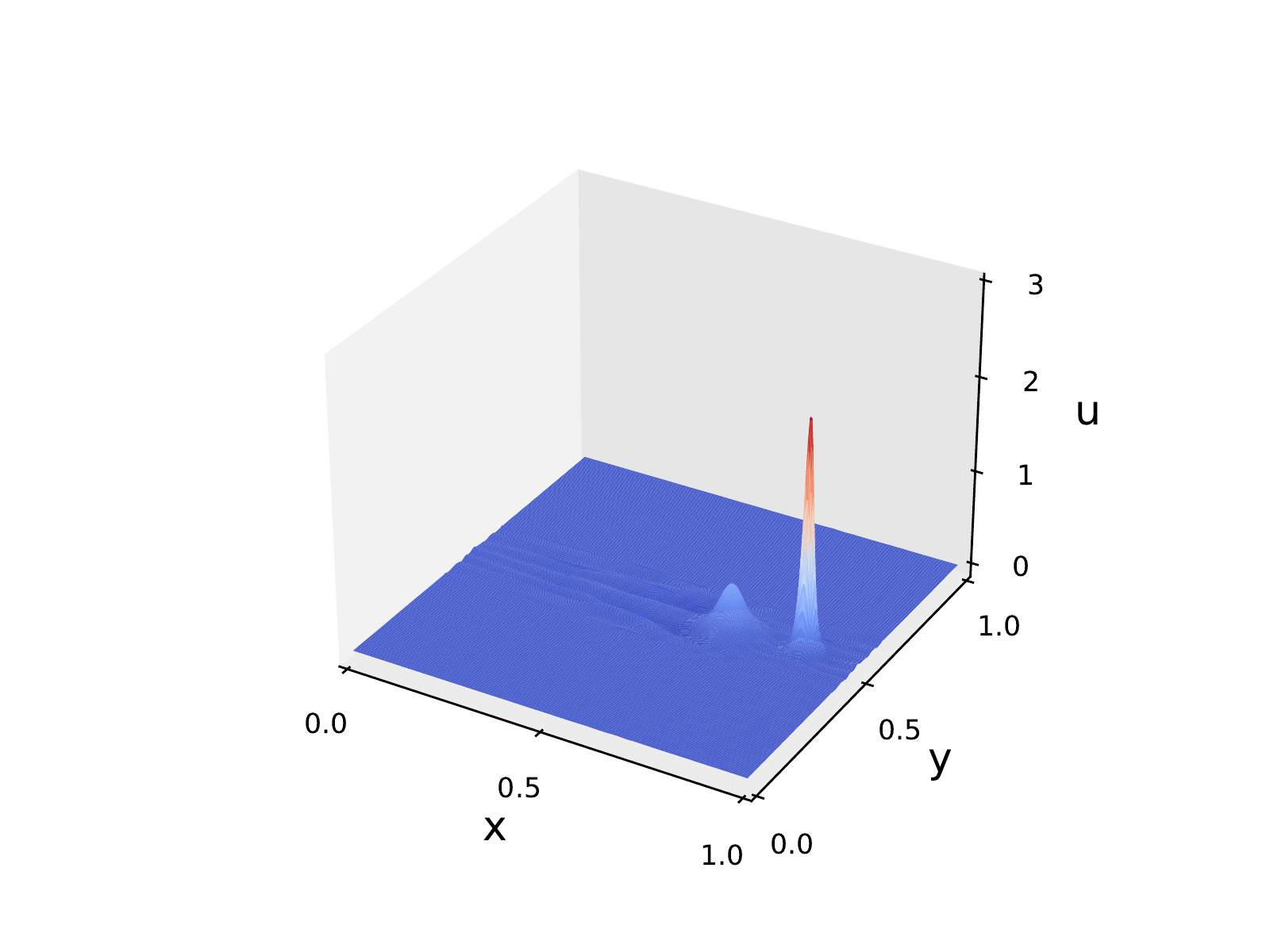}
    \end{minipage}
    }
    \subfigure[active elements at $t=0.1$]{
    \begin{minipage}[b]{0.46\textwidth}
    \includegraphics[width=1\textwidth]{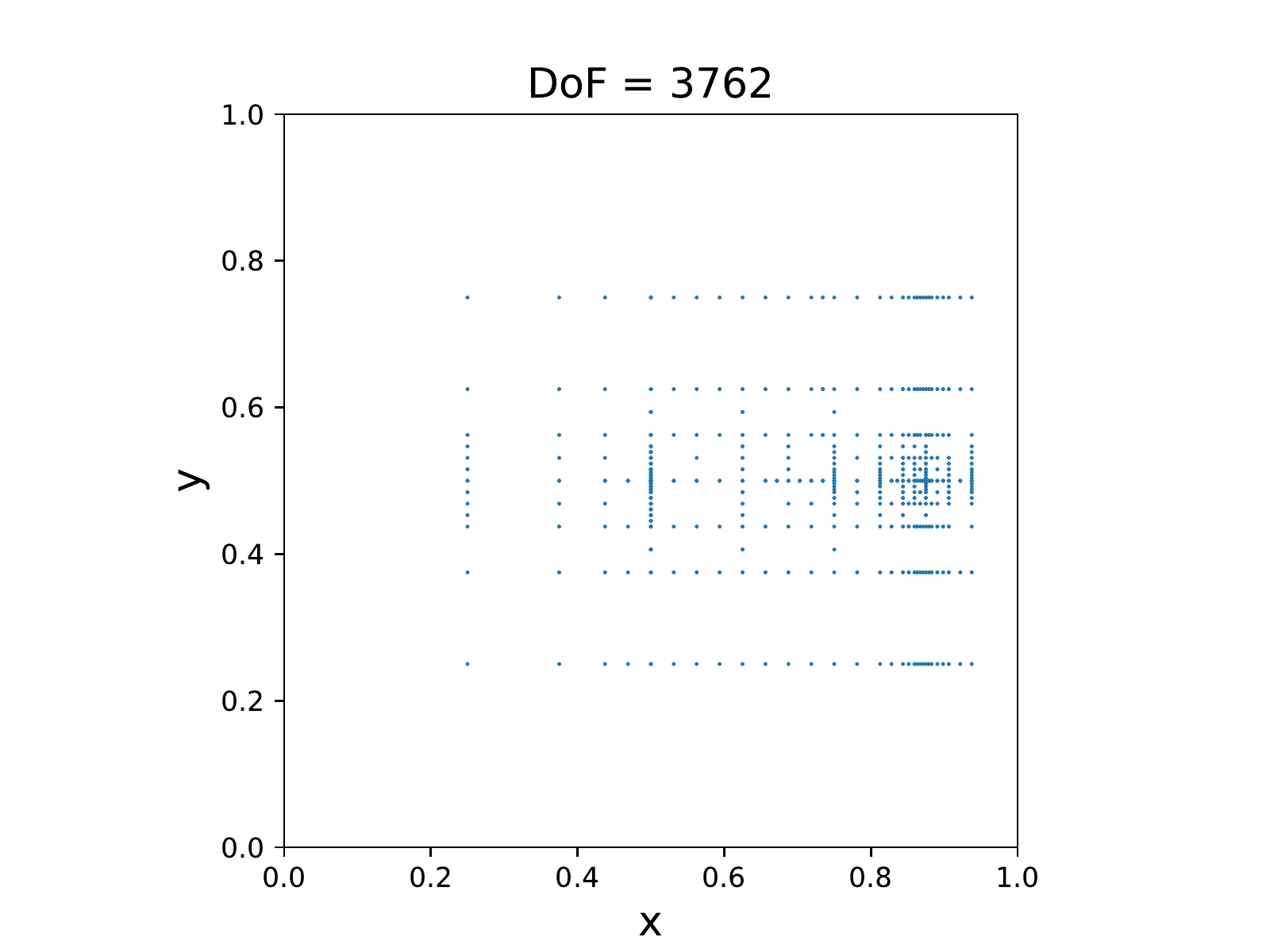}
    \end{minipage}
    }
    \caption{Example \ref{exam:zk-soliton}: ZK equation, direct collision of two dissimilar pulses solution. $t=0$, 0.04 and 0.1. $N=8$ and $\epsilon=10^{-4}$. Left: numerical solutions; right: active elements.}
    \label{fig:zk-double-case3}
\end{figure}

\begin{figure}
    \centering
    \subfigure[numerical solutions at $t=0$]{
    \begin{minipage}[b]{0.46\textwidth}
    \includegraphics[width=1\textwidth]{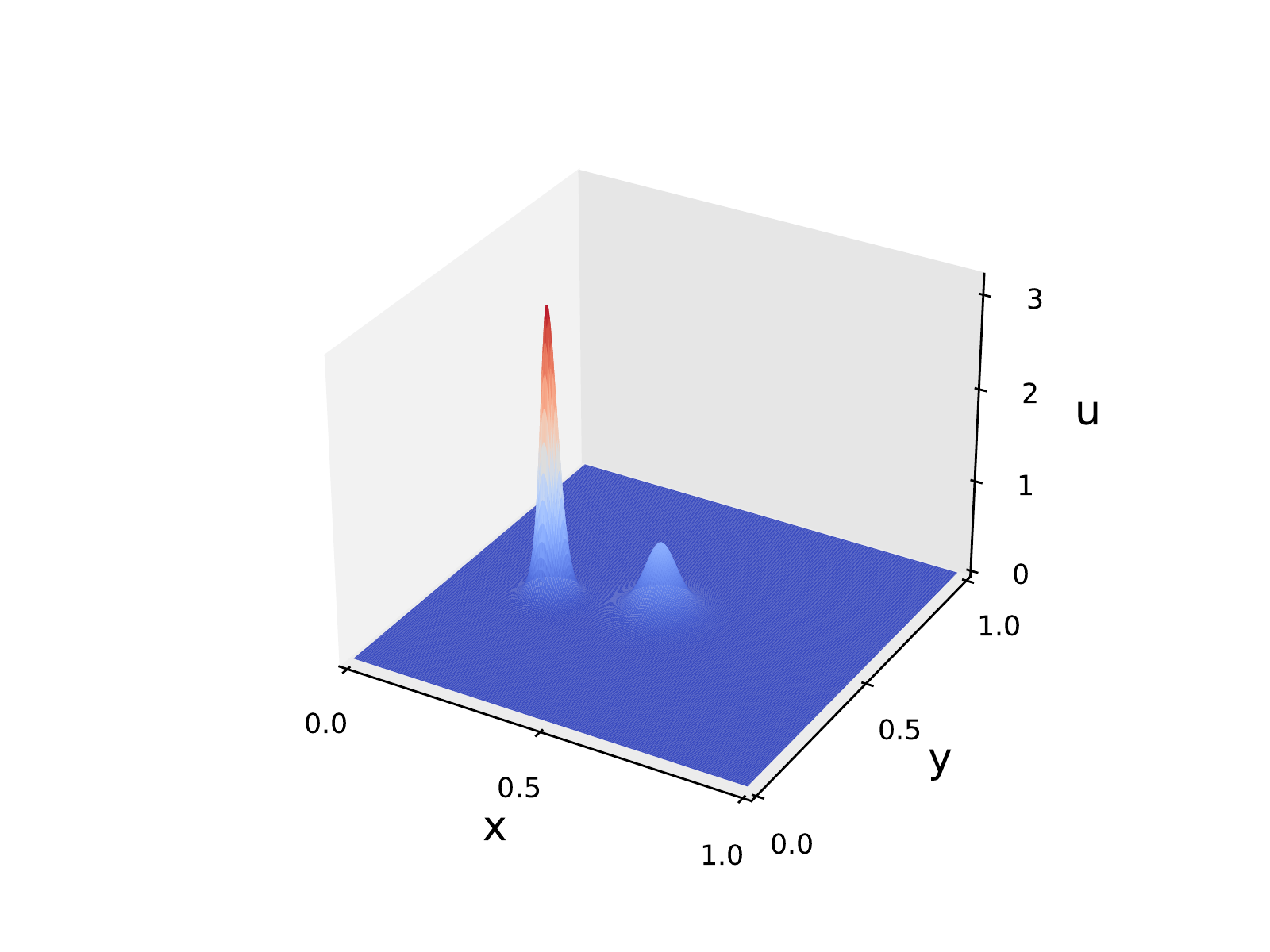}
    \end{minipage}
    }
    \subfigure[active elements at $t=0$]{
    \begin{minipage}[b]{0.46\textwidth}
    \includegraphics[width=1\textwidth]{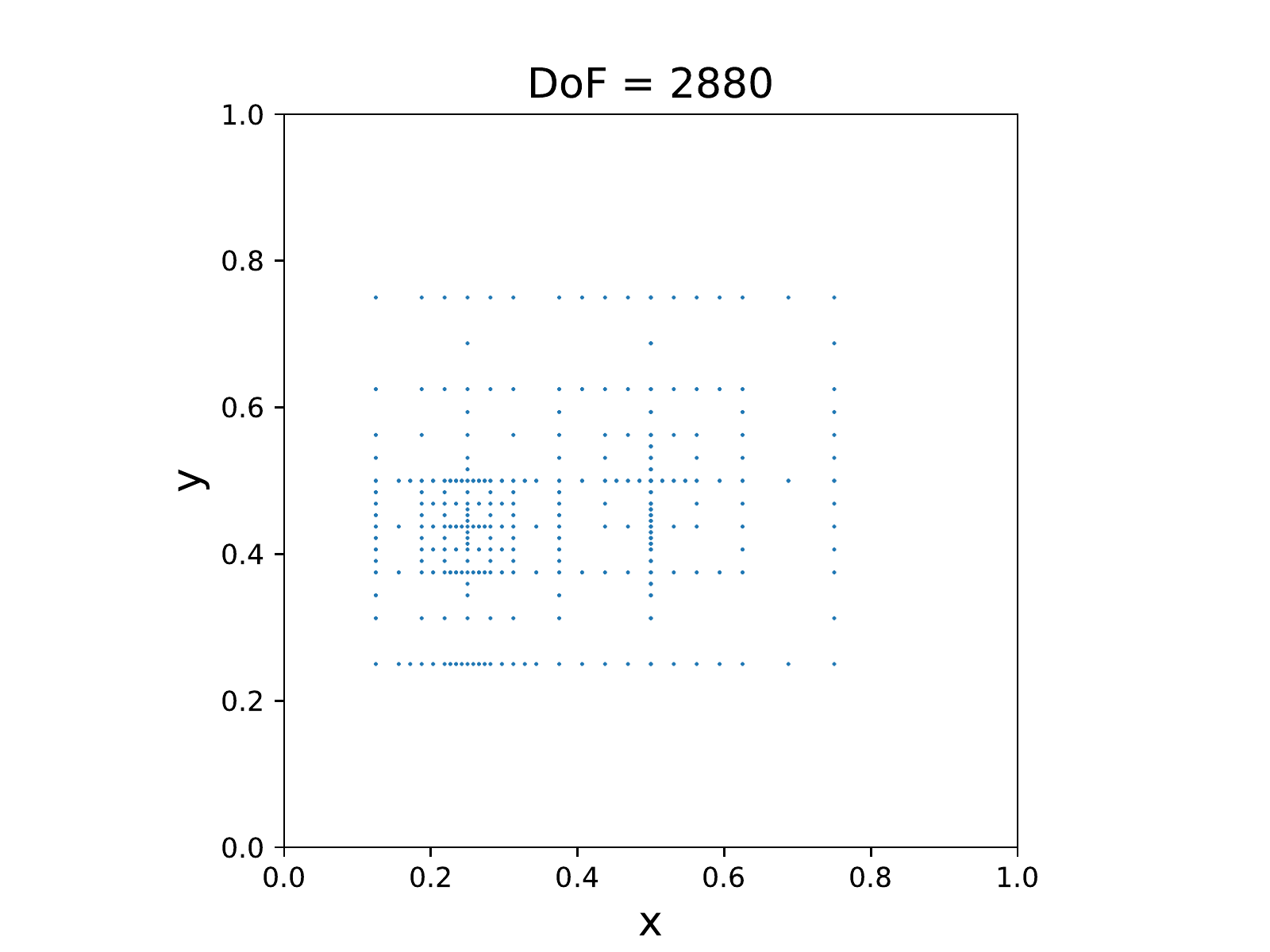}
    \end{minipage}
    }
    \bigskip
    \subfigure[numerical solutions at $t=0.1$]{
    \begin{minipage}[b]{0.46\textwidth}
    \includegraphics[width=1\textwidth]{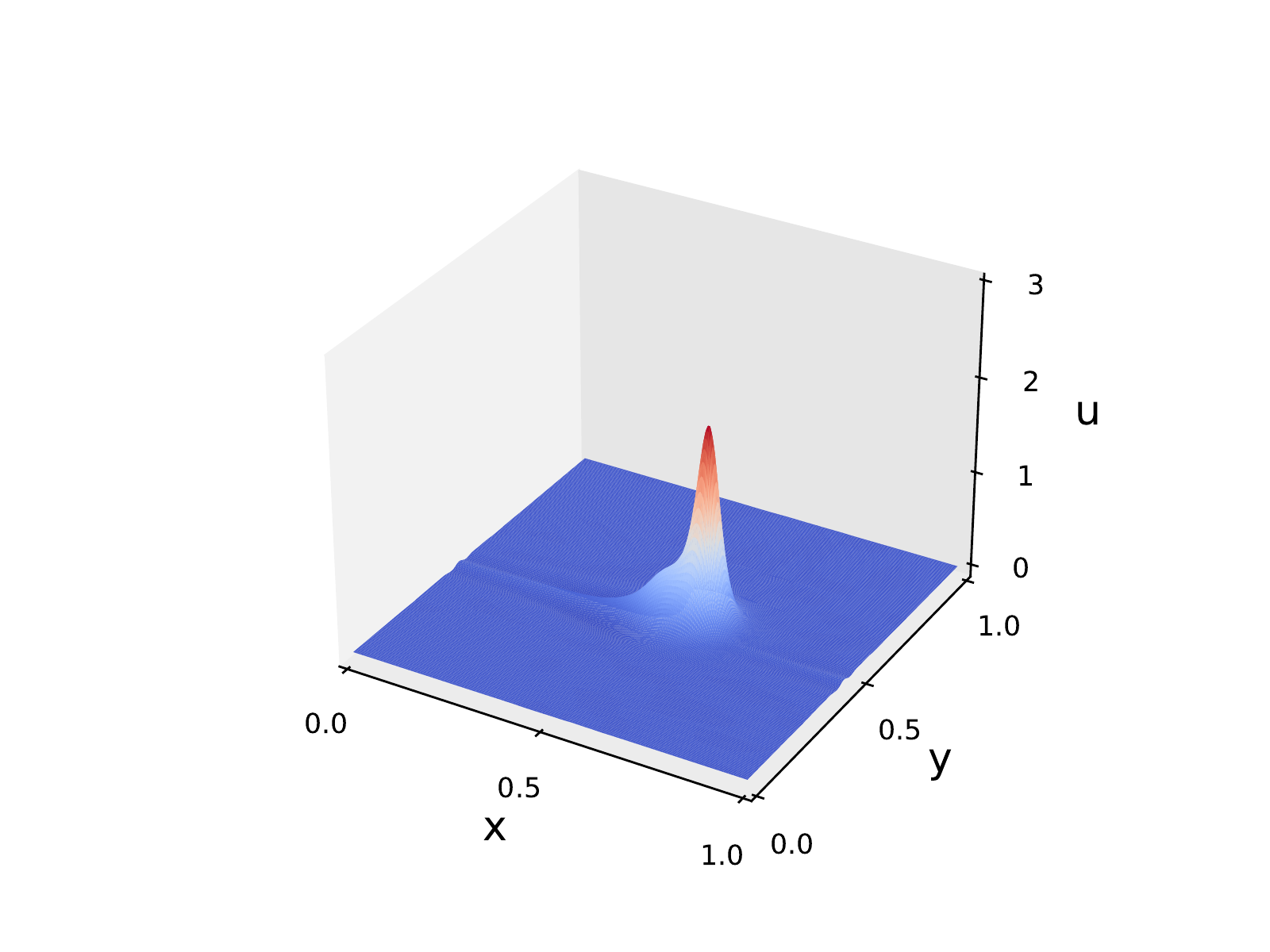}
    \end{minipage}
    }
    \subfigure[active elements at $t=0.1$]{
    \begin{minipage}[b]{0.46\textwidth}
    \includegraphics[width=1\textwidth]{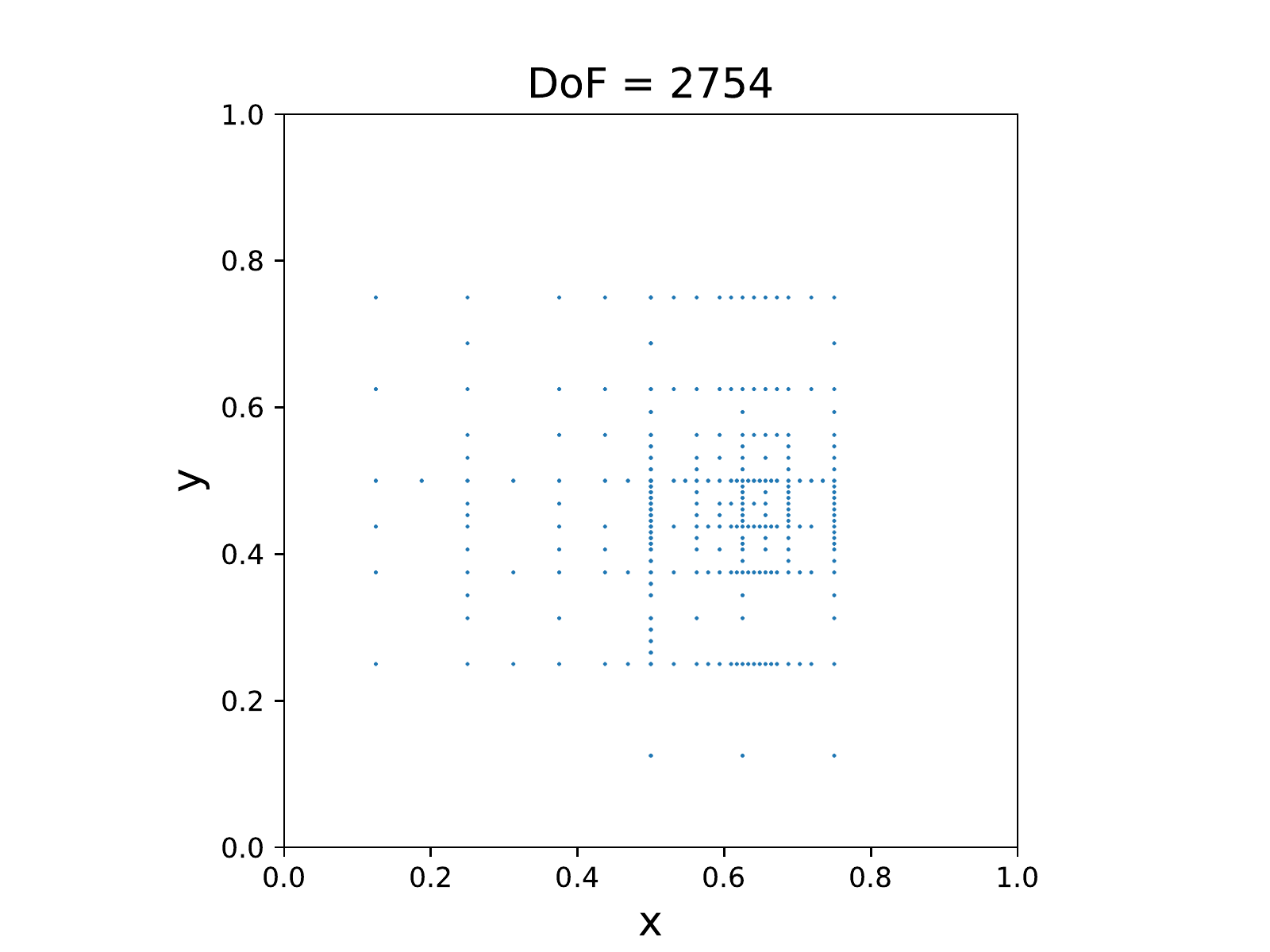}
    \end{minipage}
    }
    \bigskip
    \subfigure[numerical solutions at $t=0.15$]{
    \begin{minipage}[b]{0.46\textwidth}
    \includegraphics[width=1\textwidth]{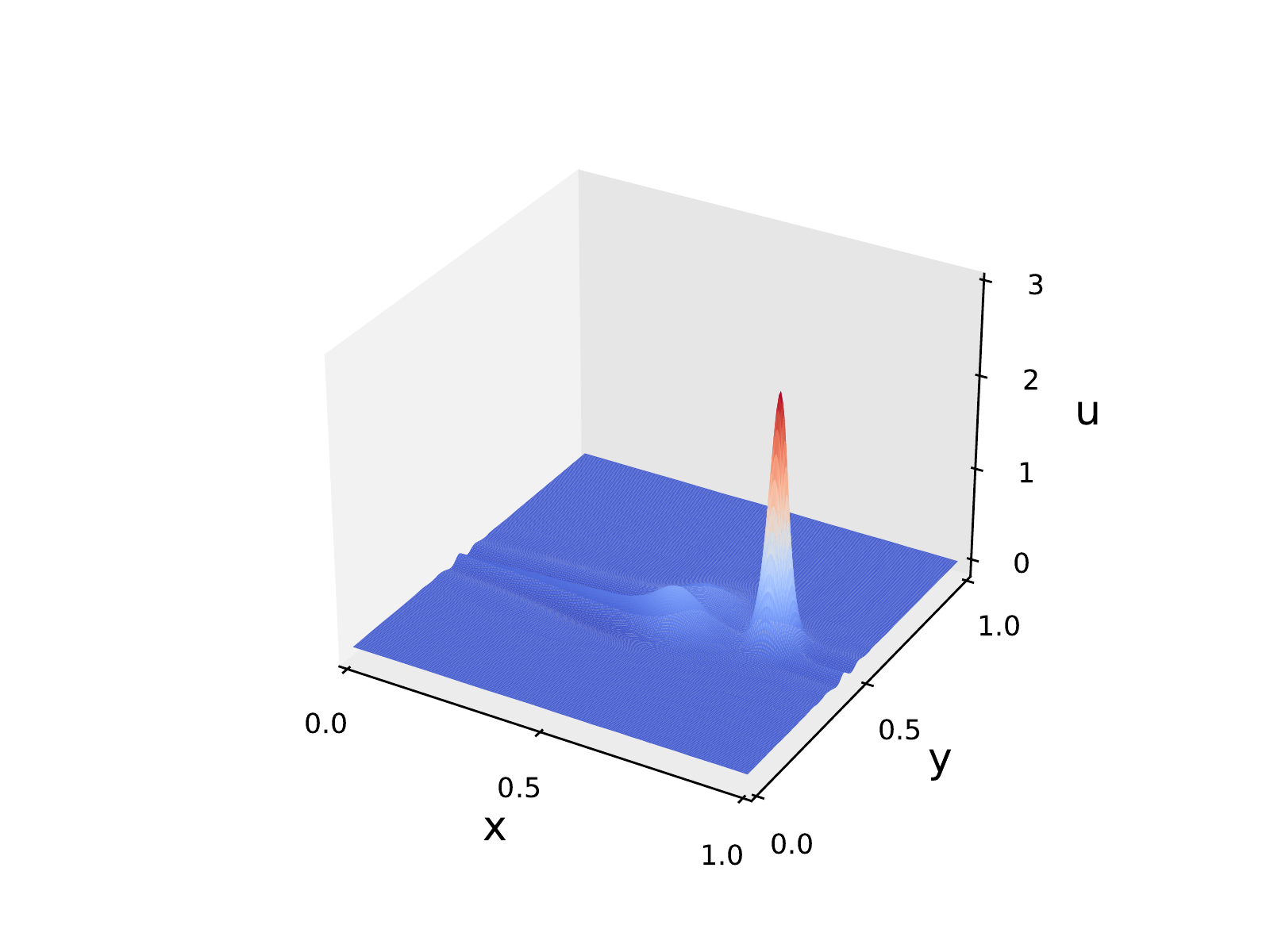}
    \end{minipage}
    }
    \subfigure[active elements at $t=0.15$]{
    \begin{minipage}[b]{0.46\textwidth}
    \includegraphics[width=1\textwidth]{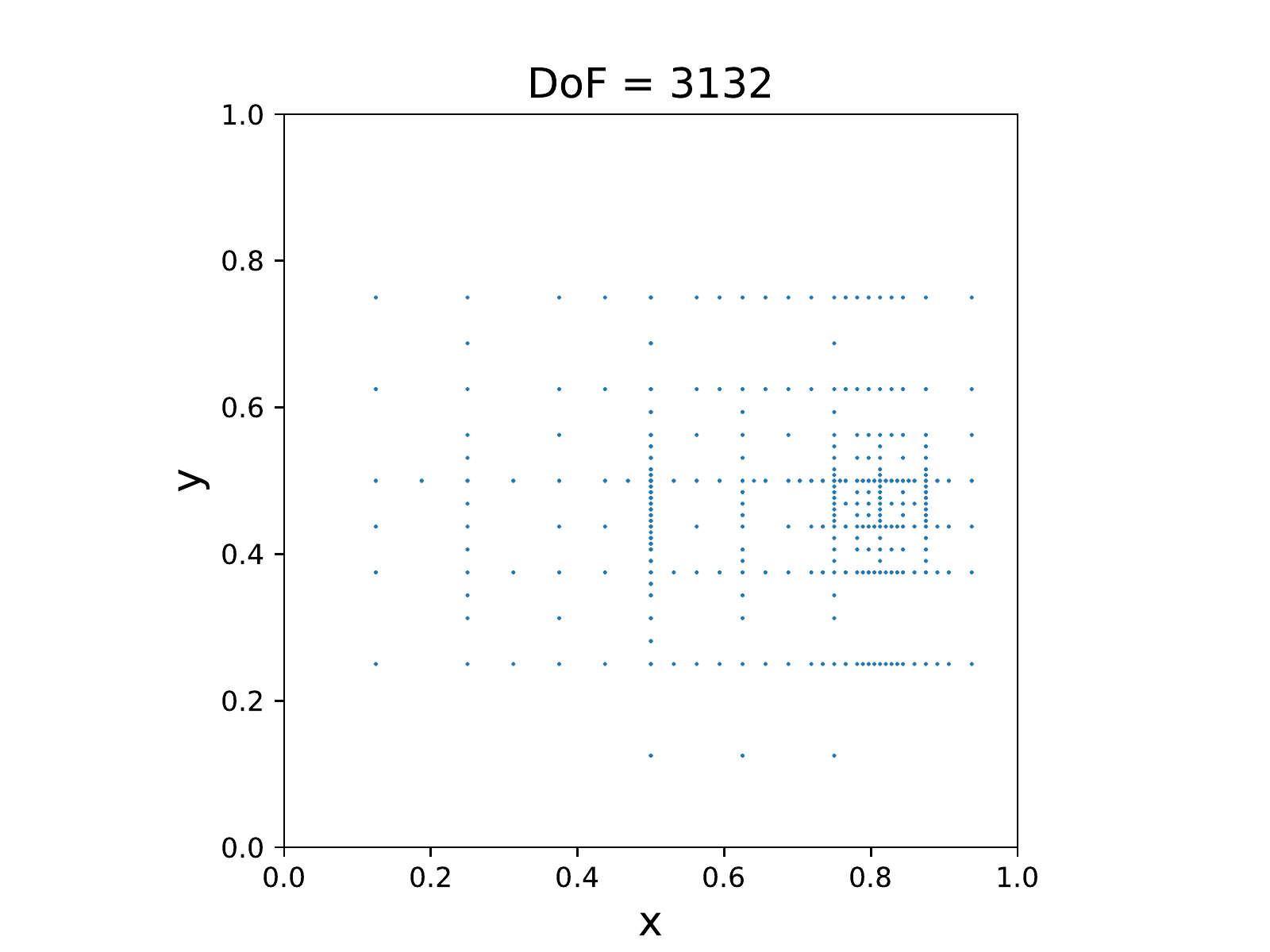}
    \end{minipage}
    }
    \caption{Example \ref{exam:zk-soliton}: ZK equation, deviated collision of two dissimilar pulses solution. $t=0$, 0.1 and 0.15. $N=8$ and $\epsilon=10^{-4}$. Left: numerical solutions; right: active elements.}
    \label{fig:zk-double-case4}
\end{figure}

\begin{exam}[lump solitons for the ZK equation]\label{exam:zk-lump}
    We consider the lump solutions for the ZK equation \cite{jorge2005evolution}:
    \begin{equation}
        u_t + (3u^2)_x + \sigma(u_{xxx} + u_{xyy}) = 0,
    \end{equation}
    with the initial condition
    \begin{equation}
        u(x,y,0) = A e^{-\kappa((x-x_0)^2+(y-y_0)^2)}
    \end{equation}
    with $\sigma={1}/{6400}$, $A = 0.4$, $\kappa=320$ and $(x_0,y_0)=(\half,\half)$.

    The numerical solutions and the active elements are presented in Figure \ref{fig:zk-lump}. It is observed that the lump initial condition evolves into a lump soliton followed by a tail of radiation within a caustic. The solution profiles are comparable to \cite{jorge2005evolution}.
\end{exam}

\begin{figure}
    \centering
    \subfigure[numerical solutions at $t=0$]{
    \begin{minipage}[b]{0.46\textwidth}
    \includegraphics[width=1\textwidth]{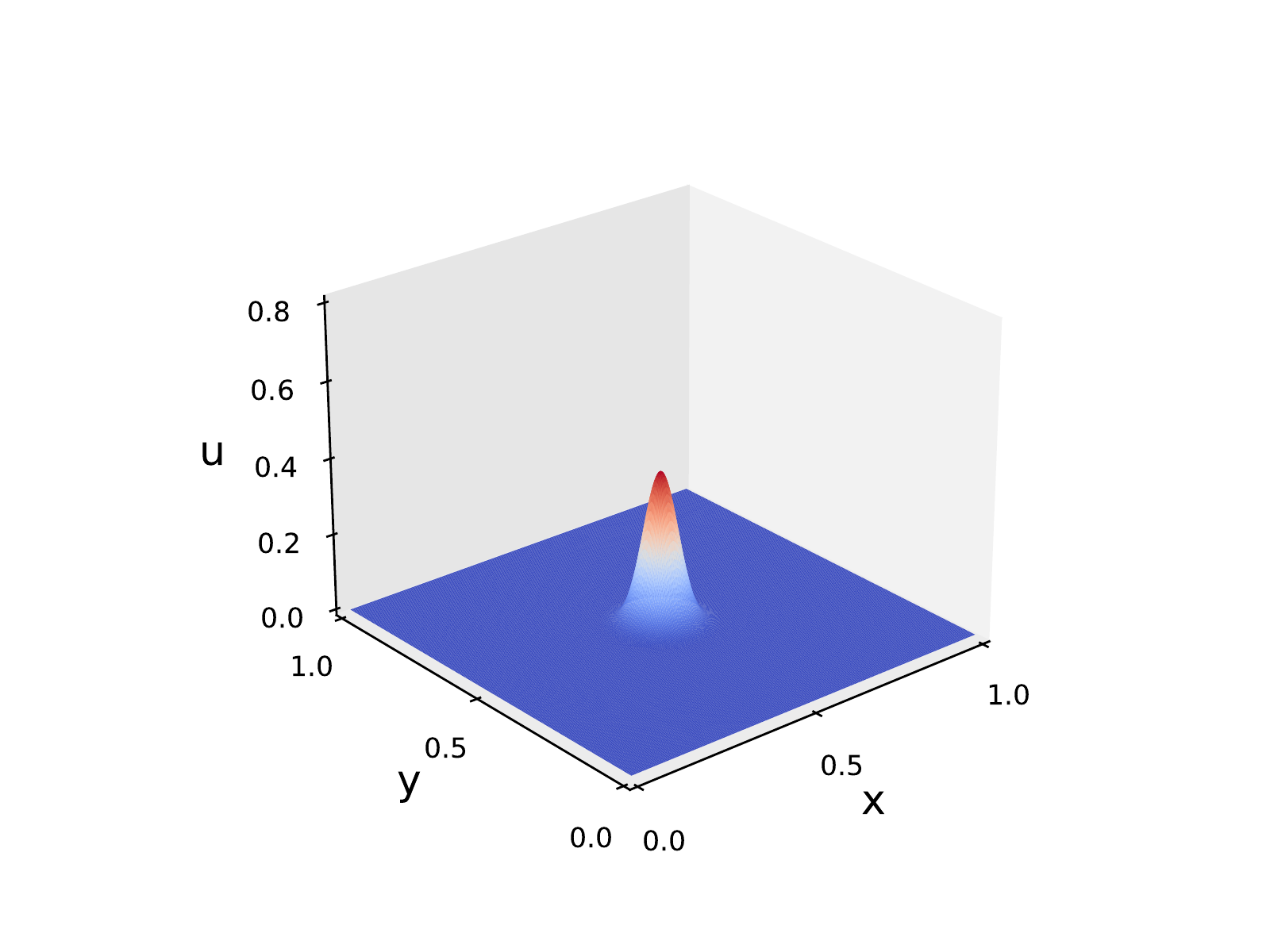}
    \end{minipage}
    }
    \subfigure[active elements at $t=0$]{
    \begin{minipage}[b]{0.46\textwidth}
    \includegraphics[width=1\textwidth]{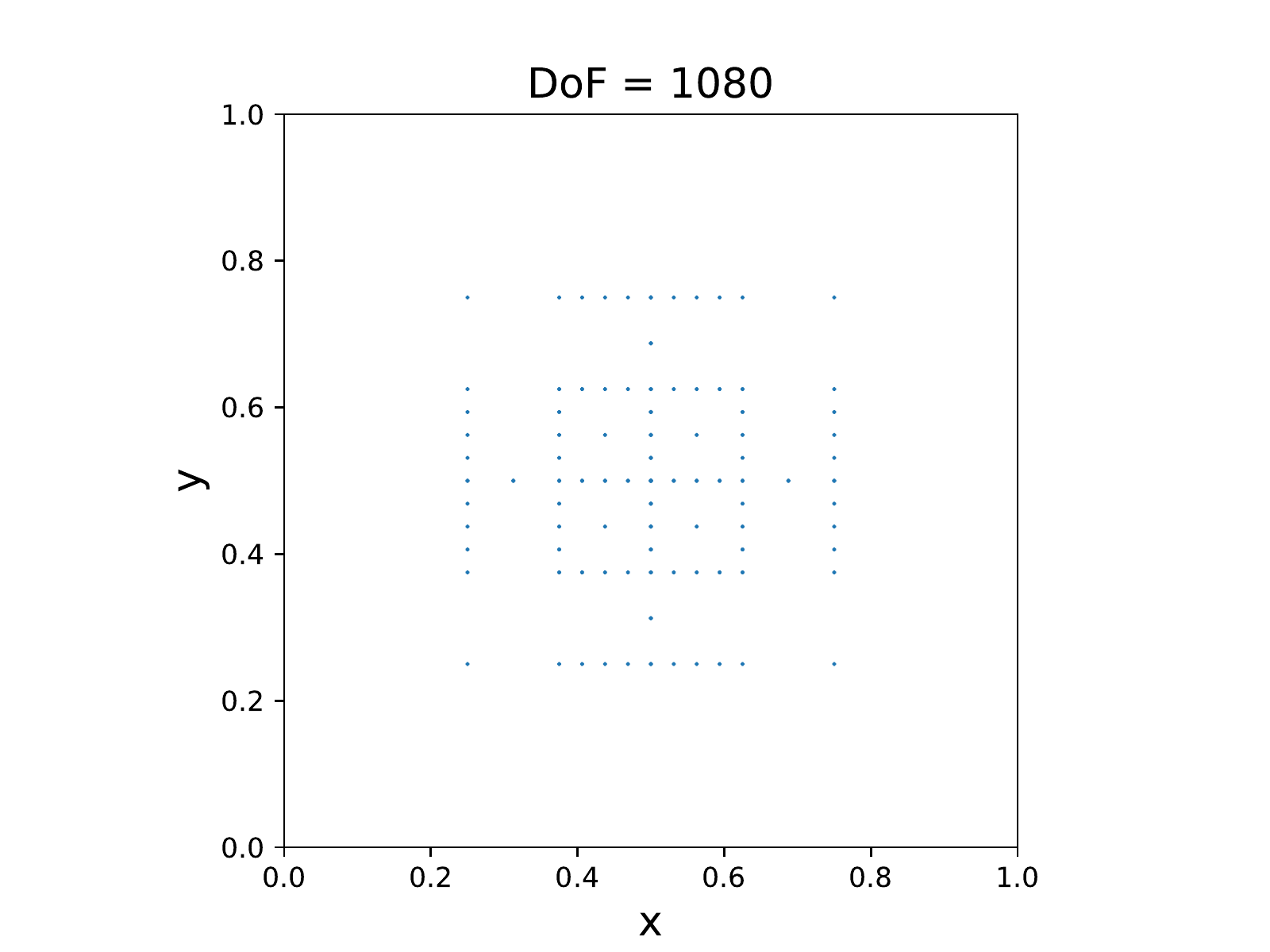}
    \end{minipage}
    }
    \bigskip
    \subfigure[numerical solutions at $t=0.2$]{
    \begin{minipage}[b]{0.46\textwidth}
    \includegraphics[width=1\textwidth]{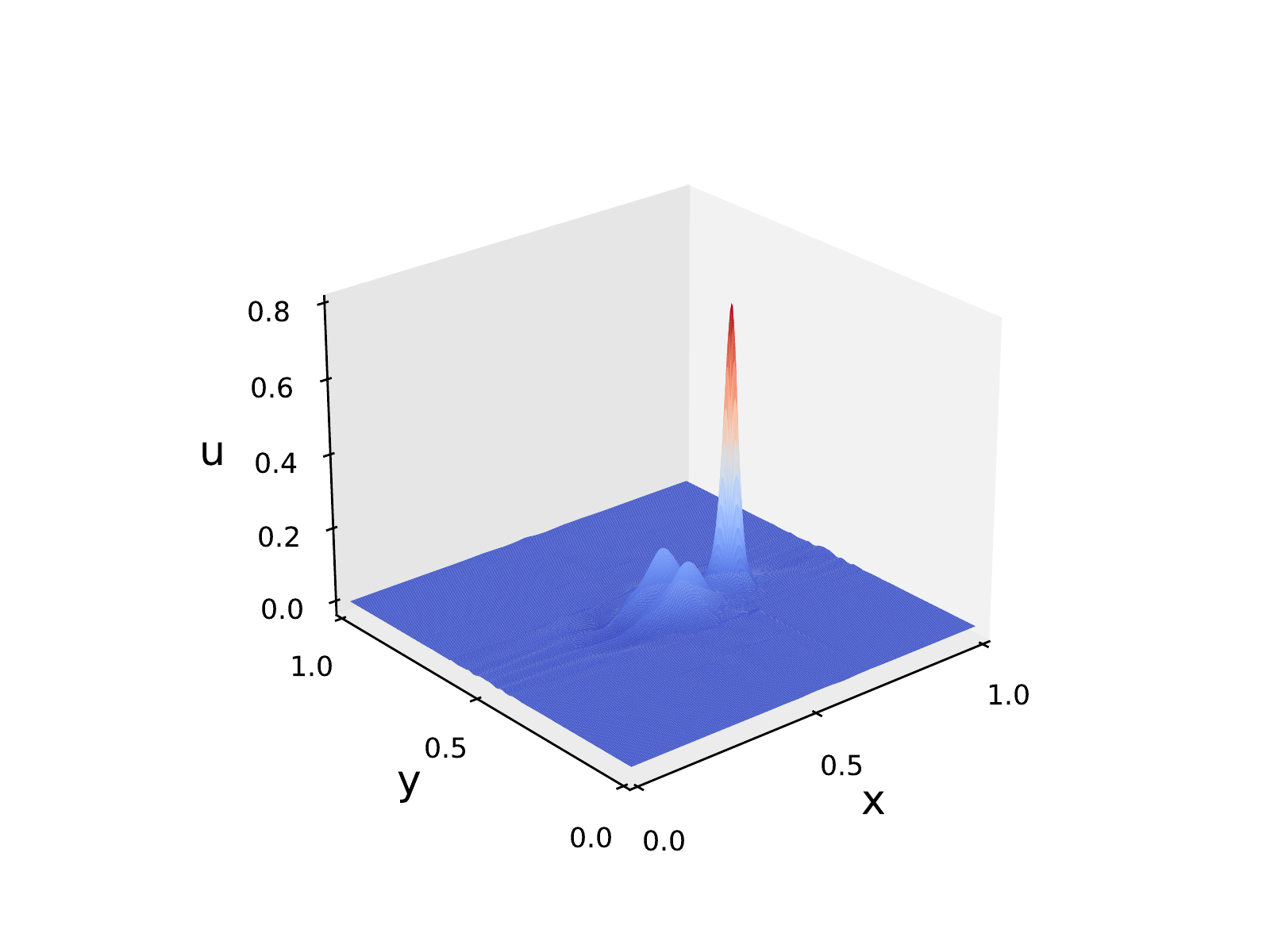}
    \end{minipage}
    }
    \subfigure[active elements at $t=0.2$]{
    \begin{minipage}[b]{0.46\textwidth}
    \includegraphics[width=1\textwidth]{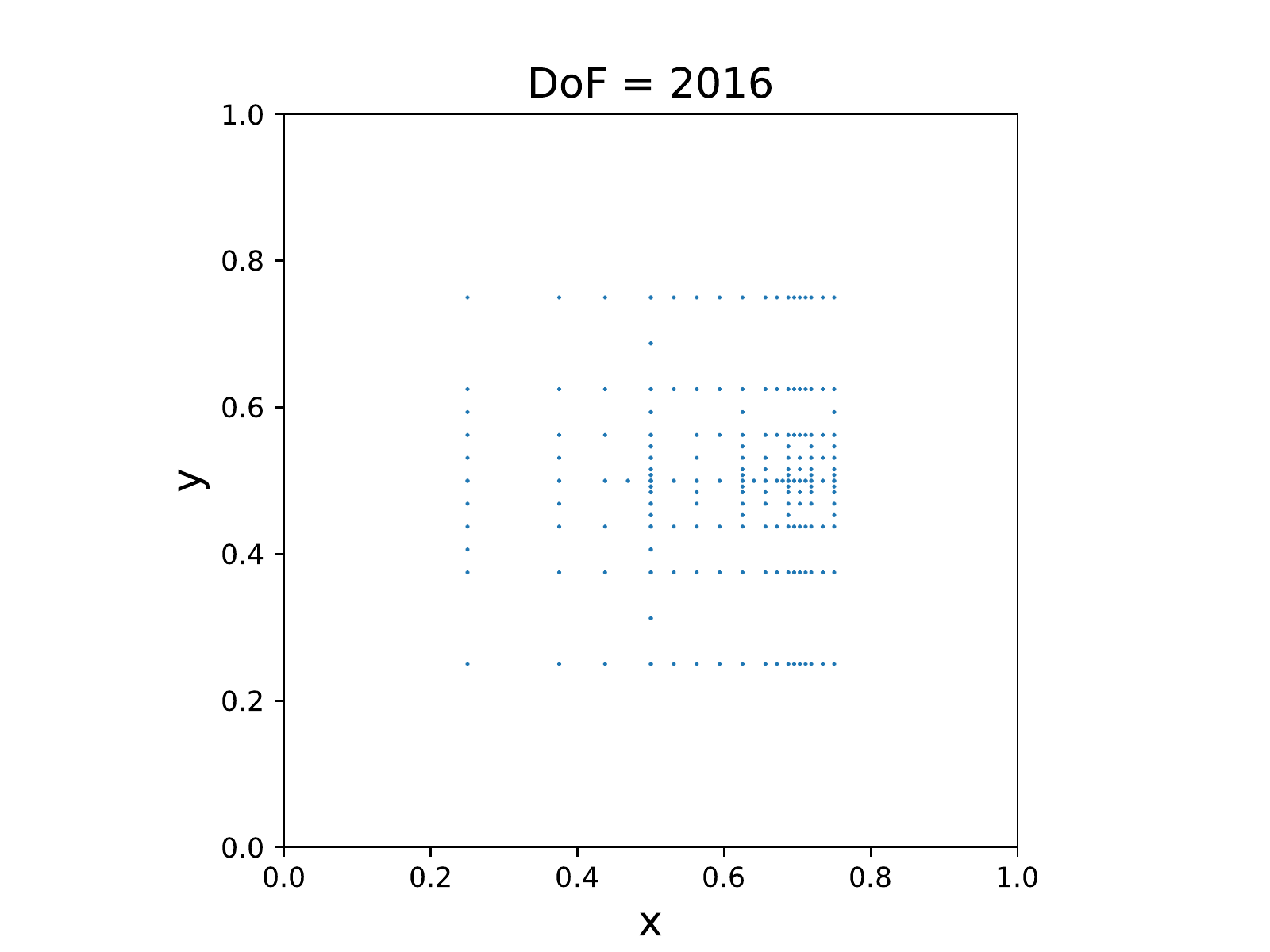}
    \end{minipage}
    }
    \caption{Example \ref{exam:zk-soliton}: ZK equation, lump solitons. $t=0$ and 0.2. $N=8$ and $\epsilon=10^{-4}$. Left: numerical solutions; right: active elements.}
    \label{fig:zk-lump}
\end{figure}

\section{Conclusion}\label{sec:conclusion}

In this work, we propose a class of adaptive multiresolution ultra-weak DG methods for solving dispersive equations including the KdV equation in one dimension and the ZK equation in two dimension. In particular, we propose a new ultra-weak DG method for the ZK equation. We prove the $L^2$ stability of the scheme and the optimal error estimate by a carefully designed local projection. Numerical examples are presented to illustrate the accuracy and capability of capturing the soliton waves. The code generating the results in this paper can be found at the GitHub link: \url{https://github.com/JuntaoHuang/adaptive-multiresolution-DG}.

\begin{appendices}

\section{Proof of Proposition \ref{prop:projection}}\label{sec:appendix-projection}

In this appendix, we present the proof of Proposition \ref{prop:projection}. We first prove the existence and uniqueness of the projection and then prove the approximation property.

Note that the procedure to find $\Pi^{\star}u\in Q^k(K_{i,j})$ is to solve a linear system with a square matrix, so the existence and uniqueness are equivalent. Thus, we only need to prove the uniqueness of the projection $\Pi^{\star}$. By assuming that $u(x,y)=0$, we would like to prove that $\Pi^{\star}u=0$.

Notice that $(\Pi^{\star}u)_y(x,y_{\jl}^+)\in P^{k}(I_i)$. Then, \eqref{proj2}, \eqref{proj6} and the existence and uniqueness of the 1D right Gauss-Radau projection \cite{castillo2002optimal} on $I_i$ implies
\begin{equation}\label{eq:Gauss-Radau-1D-right}
	(\Pi^{\star}u)_y(x,y_{\jl}^+)=0,\quad \forall x \in I_i.	
\end{equation}
Similarly, \eqref{proj3}, \eqref{proj5} and the existence and uniqueness of the 1D left Gauss-Radau projection \cite{castillo2002optimal} on $I_i$ implies
\begin{align}\label{eq:Gauss-Radau-1D-left}
\Pi^{\star}u(x,y_{\jr}^-)=0, \quad \forall x\in I_i.
\end{align}
By using \eqref{eq:Gauss-Radau-1D-right} and \eqref{eq:Gauss-Radau-1D-left}, we can write $\Pi^{\star}u(x,y)$ in the following formulation
\begin{equation}\label{eq:decompose-u-G}
	\Pi^{\star}u(x,y)=\int_y^{y_{\jr}}(s-y_{\jl})G(x,s)\, ds	
\end{equation}
for some $G(x,s)\in P^{k}(I_i)\otimes P^{k-2}(J_j)$.

Plugging in \eqref{eq:decompose-u-G} into \eqref{proj4}, we have
\begin{align}\label{equ1}
&\int_{y_{\jl}}^{y_{\jr}}\int_y^{y_{\jr}}(s-y_{\jl})G(x_{\il}^+,s)\, ds \varphi(y)\, dy\\
=&\int_{y_{\jl}}^{y_{\jr}}\int_{y_{\jl}}^y\varphi(s)\, ds(y-y_{\jl})G(x_{\il}^+,y)\,dy\nonumber \\
=&0,\quad \forall \varphi(y)\in P^{k-2}(J_j)\nonumber
\end{align}
Taking $\int_{y_{\jl}}^y\varphi(s)\, ds=(y-y_{\jl})G(x_{\il}^+,y)\in P^{k-1}(J_j)$ in \eqref{equ1}, we immediately get
\begin{align}
G(x_{\il}^+,y)=0, \quad \forall y\in J_j.
\end{align}
Therefore, we can make further decomposition of $G(x,y)$ in \eqref{eq:decompose-u-G} and have
\begin{equation}\label{eq:decompose-u-H}
	\Pi^{\star}u(x,y)=\int_y^{y_{\jr}}(s-y_{\jl})(x-x_{\il})H(x,s)\, ds,
\end{equation}
for some $H(x,s)\in P^{k-1}(I_i)\otimes P^{k-2}(J_j)$.

Finally, plugging in \eqref{eq:decompose-u-H} into \eqref{proj1}, we have
\begin{align}\label{equ2}
&\int_{x_{\il}}^{x_{\ir}}\int_{y_{\jl}}^{y_{\jr}}\int_y^{y_{\jr}}(s-y_{\jl})(x-x_{\il})H(x,s)\, dsv(x,y)\, dy dx\\
=&\int_{x_{\il}}^{x_{\ir}}\int_{y_{\jl}}^{y_{\jr}}\int_{y_{\jl}}^{y}v(x,s)\,ds(y-y_{\jl})(x-x_{\il})H(x,y)\,dy dx\nonumber\\
=&0 \quad \forall v(x,y)\in P^{k-1}(I_i)\otimes P^{k-2}(J_j). \nonumber
\end{align}
Taking $\int_{y_{\jl}}^{y}v(x,s)\,ds=(y-y_{\jl})H(x,y)\in P^{k-1}(I_i)\otimes P^{k-1}(J_j)$ in \eqref{equ2}, we have $H(x,y)\equiv0.$
Therefore $\Pi^{\star}u(x,y) \equiv0$. We finish the proof of the existence and uniqueness of the projection $\Pi^{\star}$.

We now turn to the proof of the approximation property. Note that the projection $\Pi^{\star}$ is a local projection on $K_{i,j}$. Thus, we do the standard scaling arguments $\xi:=\frac{2(x-x_i)}{h_{x_i}}$ and $\eta:=\frac{2(y-y_j)}{h_{y_j}}$ and only consider the reference cell $[-1,1]\times[-1,1]$. We denote
\begin{equation}\label{eq:scale}
	\Pi^{\star}u(x,y)=\sum_{m=1}^{(k+1)^2}a_mL_m(\xi,\eta),
\end{equation}
where $\{L_m\}_{m=1}^{(k+1)^2}$ is a set of basis functions of $Q^k([-1,1]\times[-1,1)$, e.g., $\{L_m\}_{m=1}^{(k+1)^2}=\{1,\xi,\eta,\ldots,\xi^{k}\eta^{k}\}$. We collect the coefficients in \eqref{eq:scale} in a vector $\bm{a}=(a_1,a_2,\ldots,a_{(k+1)^2})^T$. From the existence and uniqueness of the projection, we can solve a square linear system
\begin{align}
\mathbf{A}\bm{a}=\bm{b}
\end{align}
to get $\bm{a}=\mathbf{A}^{-1}\bm{b}$.

Notice that each component of $\bm{b}$ is bounded by some norms of $u$:
\begin{align*}
\|\bm{b}\|_{l^{\infty}}&\leq C (\|\hat{u}(\xi,\eta)\|_{L^{\infty}([-1,1]\times[-1,1])}+\|\hat{u}_\eta(\xi,\eta)\|_{L^{\infty}([-1,1]\times[-1,1])})\\
&=C(\|u(x,y)\|_{L^{\infty}(K_{i,j})}+h_{y_j}\|u_y(x,y)\|_{L^{\infty}(K_{i,j})})
\end{align*}
where $\hat{u}(\xi,\eta):=u(\frac{1}{2}h_{x_i}\xi+x_i,\frac{1}{2}h_{y_j}\eta+y_j)$. Moreover, $\mathbf{A}$ only depends on the constant $k$. We have
\begin{equation}
\|\bm{a}\|_{l^{\infty}}\leq C(\|u(x,y)\|_{L^{\infty}(K_{i,j})}+h_{y_j}\|u_y(x,y)\|_{L^{\infty}(K_{i,j})})
\end{equation}
and thus
\begin{align}
\|\Pi^{\star}u(x,y)\|_{L^{\infty}(K_{i,j})}\leq C(\|u(x,y)\|_{L^{\infty}(K_{i,j})}+h_{y_j}\|u_y(x,y)\|_{L^{\infty}(K_{i,j})}).
\end{align}
Since the projection is a local projection which preserves the polynomial up to degree $k$-th, the boundedness of the projection and standard approximation theory implies,
\begin{align}
\|\Pi^{\star}u(x,y)-u(x,y)\|_{L^2(K_{i,j})}\leq C h^{k+1}\|u\|_{H^{k+1}(K_{i,j})}.
\end{align}

\end{appendices}

\newpage
\bibliographystyle{abbrv}
\bibliography{ref_dispersive}

\end{document}